\documentclass[10pt]{amsart}

\usepackage{mathtools}
\usepackage[round]{natbib}
\setcitestyle{numbers,square}
\usepackage{hyperref}
\usepackage{xcolor}
\usepackage[normalem]{ulem}
\usepackage{tikz}
\tikzstyle{lien}=[->,>=stealth,rounded corners=5pt,thick]
\tikzset{individu/.style={draw,thick,fill=#1!25},
	individu/.default={green}}

\usepackage{accents}

\newcommand{\diffns}{\mathrm{d}}

\newcommand\redsout{\bgroup\markoverwith{\textcolor{red}{\rule[0.5ex]{2pt}{0.4pt}}}\ULon}

\newcommand{\E}{\mathbb{E}}
\newcommand{\N}{\mathbb{N}}

\newcommand{\Q}{\mathbb{Q}}
\newcommand{\R}{\mathbb{R}}

\newcommand{\Pb}{\mathbb{P}}

\def\={{\;\mathop{=}\limits^{\text{(law)}}\;}}

\newtheorem{theorem}{Theorem}[section]
\newtheorem{prop}[theorem]{Proposition}
\newtheorem{lemma}[theorem]{Lemma}
\newtheorem{defi}[theorem]{Definition}
\newtheorem{corollary}[theorem]{Corollary}

\newtheorem{remark}[theorem]{Remark}

\newtheorem{example}[theorem]{Example}

\numberwithin{equation}{section}

\usepackage[english]{babel}

\title[\texttt{Smoothness of solutions of HSPDEs with $L^{\infty}$ vector field and additive noise}]
{Smoothness of solutions of  hyperbolic stochastic partial differential equations with $L^{\infty}$-vector fields}

\author[A.-M. Bogso]{Antoine-Marie Bogso}
\address{University of Yaounde I,\\
	Faculty of Sciences, Department of Mathematics,\\
	P.O. Box 812, Yaounde, Cameroon\\
	AIMS Ghana, P.O. Box LGDTD 20046, Summerhill Estates, Eat Legon Hills, Santoe, Acrra}
\email{antoine.bogso@facsciences-uy1.cm}

\author[M. Dieye]{Moustapha Dieye}

\address{\'Ecole polytechnique de Thi\`es, D\'epartement tronc commun, BP A10, Thi\`es, S\'en\'egal}

\email{mdieye@ept.sn}

\author[O. Menoukeu-Pamen]{Olivier Menoukeu-Pamen}
\address{Institute for Financial and Actuarial Mathematics (IFAM), \\
	Department of Mathematical Sciences, University of Liverpool, \\
	Liverpool L69 7ZL, UK\\
	AIMS Ghana, P.O. Box LGDTD 20046, Summerhill Estates, East Legon Hills, Santoe, Acrra}
\email{menoukeu@liverpool.ac.uk}

\author[F. Proske]{Frank Proske}
\address{University of Oslo \\
	CMA, Department of Mathematical Sciences, University of Oslo \\
	Moltke Moes Vei 35, Po Box 1053, Blindern, 0316 Oslo, Norway}
\email{proske@math.uio.no}

\thanks{The project on which this publication is based has been carried out with funding provided by the Alexander von Humboldt Foundation, under the programme financed by the German Federal Ministry of Education and Research entitled German Research Chair No 01DG15010.}
\subjclass{Primary 60H07, 	60H50, 	60H17; Secondary 	60H15 }

\keywords{Brownian sheet, SDEs on the plane, Wave equations, Malliavin calculus}

	\date{} 
	
\begin{document}
		
		\begin{abstract}
			In this paper we are interested in a quasi-linear hyperbolic stochastic differential equation (HSPDE) when the vector field is merely bounded and measurable. Although the deterministic counterpart of such equation may be ill-posed (in the sense that uniqueness or even existence might not be valid), we show for the first time that the corresponding HSPDE has a unique (Malliavin differentiable) strong solution. Our approach for proving this result rests on: 1) tools from Malliavin calculus
			and 2) variational techniques introduced in Davie \cite{Da07} non trivially extended to the case of
			SDEs in the plane by using an algorithm for the selection of certain rectangles. As a by product, we also obtain the Sobolev differentiability of the solution with respect to its initial value. The results derived here constitute a significant improvement of those in the current literature on SDEs on the plane and can be regarded as an analogous equivalent of the pioneering works by Zvonkin \cite{Zvon74} and Veretennikov \cite{Ver79} in the case of one-parameter SDEs with singular drift.
		\end{abstract}

		\maketitle
		\tableofcontents
		\section{Introduction}
		\label{intro}
		In this article we are interested in studying solutions $X_{s,t},0\leq s,t\leq T
		$ to the following hyperbolic stochastic
		partial differential equation (HSPDE)%
		\begin{equation} 
			\left\{ 
			\begin{array}{ll}
				\dfrac{\partial ^{2}X(s,t)}{\partial s\partial t}=b(s,t,X(s,t))+\dfrac{%
					\partial ^{2}W_{s,t}}{\partial s\partial t}, & (s,t)\in \mathcal{T}^2  \\ &\\
				X(s,0)=X(0,t)=x & 
			\end{array}
			\right. ,  \label{HyperbolicPDE}
		\end{equation}%
		where  $\mathcal{T}=[0,T]$, $T>0,$ $W_{s,t},0\leq s,t\leq T$ is a Wiener sheet (see Section \ref{defcon} for a
		definition),  $b:\mathcal{T}^2 \times \mathbb{%
			R}^{d}\longrightarrow \mathbb{R}^{d}$ is a Borel measurable function and $\frac{\partial ^{2}W_{s,t}}{\partial s\partial t}$ denotes the Gaussian white noise in $s$ and $t$.

		This equation can also be regarded as the following stochastic differential equation (SDE) in the plane%
		\begin{equation} 
			X_{s,t}=x+\int_{0}^{s}%
			\int_{0}^{t}b(r_{1},r_{2},X_{r_{1},r_{2}})\diffns r_{1}\diffns r_{2}+W_{s,t},(s,t)\in 
			\mathcal{T}^2,\,x\in \mathbb{R}^{d},  \label{SDEWienerSheet}
		\end{equation}%

		More precisely, we aim at constructing  a unique (global) strong solution $%
		X_{\cdot ,\cdot }$ to the SDE \eqref{SDEWienerSheet}, when the vector field $b$ is
		merely bounded and measurable. By a strong solution we mean that the
		solution $X_{\cdot ,\cdot }$ is a progressively measurable functional of the
		driving noise, that is the Wiener sheet $W_{\cdot ,\cdot }$ (see Definition \ref
		{DefBSheet}). To the best of our knowledge and surprisingly existence and uniqueness of a strong solution to \eqref{HyperbolicPDE} for bounded drift is still open even though its analogue in the one parameter case has been solved years ago (see \cite{Zvon74, Ver79}).
		
		We mention that the SDE \eqref{SDEWienerSheet} has been investigated by
		several authors in the literature. For example Cairoli \cite{Ca72}
		and Yeh \cite{Ye81} analysed strong existence and pathwise uniqueness of
		solutions to SDEs in the plane with multiplicative noise%
		\begin{align}
			X_{s,t}=x&+\int_{0}^{s}%
			\int_{0}^{t}b(r_{1},r_{2},X_{r_{1},r_{2}})dr_{1}dr_{2}\notag\\&+\int_{0}^{s}%
			\int_{0}^{t}\sigma (r_{1},r_{2},X_{r_{1},r_{2}})dW_{r_{1},r_{2}},(s,t)\in 
			\mathcal{T}^2,x\in \mathbb{R}^{d}  \label{MultiplicativeSDE}
		\end{align}%
		for Lipschitz continuous vector fields $b$ and $\sigma $ of linear growth;
		see also \cite{Ye87} in the case of strong solutions with a deterministic
		boundary process and \cite{Ye85} in the case of weak solutions, when $b$
		is continuous, satisfying a growth condition. Further, we refer to Nualart,
		Sanz \cite{NuaSan85}, where the authors studied smoothness of solutions
		to \eqref{MultiplicativeSDE} in the sense of Malliavin differentiability for
		sufficiently regular vector fields.
		
		Let us observe that when $d=1$ and $b,\sigma :\mathbb{R}\longrightarrow \mathbb{R}$,
		a formal $\frac{\pi }{4}$ rotation (see for example Walsh \cite{Wa84} and Farr\'{e},
		Nualart \cite{FaNu93}) can be used to transform the corresponding
		version of the HSPDE for \eqref{MultiplicativeSDE} with
		multiplicative noise term $\sigma (X(s,x))\frac{\partial ^{2}W_{s,x}}{%
			\partial s\partial x}$ into a non-linear random wave equation of the form%
		\begin{equation}
			\frac{\partial ^{2}X(t,x)}{\partial t^{2}}-\frac{\partial ^{2}X(t,x)}{%
				\partial x^{2}}=\sigma (X(t,x))\frac{\partial ^{2}\widetilde W_{t,x}}{\partial
				t\partial x}+b(X(t,x)\text{,}  \label{WaveEquation}
		\end{equation}
		where $\widetilde W$ is a new Wiener sheet. 
		There is a rich mathematical literature on the solutions to (deterministic/random) wave equations and their continuous dependence on initial data with respect to some metrics. For example, Bressan and Chen \cite{BrCh17} constructed a distance functional which makes Lipschitz continuous the flow of conservative solutions to a nonlinear (deterministic) wave equation. We refer to \cite{BrCh17b,BCZ15} for further properties of conservative solutions for the aforementioned nonlinear wave equation. Using techniques from Malliavin calculus, Carmona, Nualart \cite%
		{CaNu88} were able to show for $b,\sigma \in C_{b}(\mathbb{R})$ that
		the stochastic wave equation \eqref{WaveEquation} has a unique (weak)
		solution $X_{\cdot ,\cdot }$ on finite intervals under a Dirichlet boundary
		condition. See also Quer-Sardanyons, Tindel \cite{QuTi07} in
		the fractional Brownian sheet case.
		
		Moreover, authors in Nualart, Tindel \cite{NuTi97} establish an
		existence and uniqueness result for strong solutions to \eqref{SDEWienerSheet} under growth and montonicity assumptions on $b$ (when $d=1$) by using a
		comparison theorem. See also Erraoui, Nualart, Ouknine \cite%
		{ENO03} in the case of fractional Brownian motion with
		parameters $H_{1},H_{2}\leq \frac{1}{2}$.
		
		Finally, let us also mention some recent results on path-by-path uniqueness
		of solutions to the SDE \eqref{SDEWienerSheet} with respect to
		non-Lipschitz continuous vector fields $b$. Here the concept of path-by-path
		uniquess, which is a much stronger concept than that of strong uniqueness, is
		to be understood in the sense of Davie \cite{Da07} and means that there
		exists a measurable set $\Omega ^{\ast }$ with probability mass $1$ such
		that for all $\omega \in \Omega ^{\ast }$ the SDE \eqref{SDEWienerSheet}
		has a unique deterministic solution in the space of continuous functions $C(\mathcal{T}^2;\mathbb{R}^{d})$. The first result in this direction
		for SDEs \ref{MultiplicativeSDE} with discontinuous coefficients $b$ was
		obtained by Bogso, Dieye and Menoukeu-Pamen \cite{BDM21a},
		where the authors assume that $b$ is of spatial linear growth and
		componentwise non-decreasing. The proof of their results is based on a
		local-time-space representation, the law of iterated logarithm for Wiener
		sheets and arguments in Davie \cite{Da07} generalized to the case of SDEs
		in the plane. See also the paper of Bogso, Menoukeu-Pamen \cite%
		{BMP22}, which in addition deals with the study of Malliavin
		smoothness of solutions to \eqref{SDEWienerSheet} assuming that the drift $b$ is of spatial linear growth and  the difference of two componenwise non-decreasing functions. In this context we
		also point out the recent article of Bechthold, Harang, Rana \cite%
		{BHR22} in the case of regularizing noise sheets as e.g. the
		fractional Brownian sheet. The authors in \cite{BHR22}
		establish path-by-path uniqueness for SDEs in the plane with additive noise
		and distributional vector fields $b$ in the Besov space $B_{p,q}^{\alpha }(%
		\mathbb{R}^{d})$. In this case, the drift term is given by a type of
		non-linear Young integral, which does not coincide with the Lebesgue
		integral, in general. Their approach is based on a multiparameter sewing
		Lemma, local time and techniques (e.g. averaging operator) in Catellier,
		Gubinelli \cite{CG16} generalised to the case of SDEs in the
		plane. Let us indicate that their results cannot be applied to the case of
		discontinuous vector fields in the Lebesgue integral setting and the case of
		a Wiener sheet in higher dimensions.

		The objective of this paper is two-fold: first we construct a unique (Malliavin
		differentiable) strong solution to the SDE \eqref{SDEWienerSheet}, when $%
		b\in L^{\infty }(\mathcal{T}^2\times \mathbb{R}^{d};\mathbb{R}^{d})$. Second, we prove that the solution is locally Sobolev differentiable with
		respect to the initial condition. The proof of these results is based on a
		compactness criterion for square integrable functionals of the Wiener sheet
		from Malliavin calculus and is inspired by variational techniques in Davie 
		\cite{Da07}. We comment on here that due to the nature of integration on the plane, it is not possible to directly apply classical integration by parts techniques as in \cite[Proposition 2.2]{Da07} to obtain certain expressions in terms of iterated integrals on a simplex. In order to overcome this difficulty, we devise an algorithm for the selection of specific rectangles with respect to the Wiener sheet, which are used in connection with an integration by parts argument (see for example Section \ref{secbasest}). 
		
		These results shed light on the regularisation effect of Gaussian white noise on ill-posed singular (noiseless) HPDEs, which do not admit existence, uniqueness or regularity of solutions, in general.
		We shall also allude here to several other results beyond the setting of stochastic HPDEs, which are based on regularisation by noise techniques and which have drawn attention. Gy\"ongy and Pardoux \cite{GP93a} (resp. \cite{GP93b}) proved well-posedness of a class of quasi-linear parabolic partial differential equations for space-dimension $d=1$ with additive space-time Gaussian white noise and measurable vector fields satisfying a local boundedness (resp. integrability) condition. Further, Beck, Flandoli, Gubinelli and Maurelli \cite{BFGM14} established Sobolev regularity of solutions to linear stochastic transport and continuity equations with drifts in critical $L^p$ spaces. We also point out the work of Butkovsky and Mytnik \cite{BM16}, who studied regularizing effects of Gaussian white noise on heat equations with non-Lipschitz vector fields in the sense of path-by-path uniqueness. See also \cite{DFPR14}, \cite{Wre17}, \cite{KR05}, \cite{MMNPZ13}, \cite{Pri18} or \cite{SvPr14} in the case of regularizing Markovian noise, that is (infinite dimensional) Wiener or L\'evy noise. The reader is also referred to the nice survey article by Gess \cite{Gess}.

		Regarding non-Markovian noise Catellier and Gubinelli \cite{CG16} as already mentioned above studied the regularisation by noise problem with respect to additive perturbations by fractional Brownian paths (see also Galeati and Gubinelli \cite{GG21}). Amine, Banos, Proske \cite{ABP17a} and Amine, Mansouri, Proske \cite{AMP20} investigated the regularisation by noise problem for ODEs (and transport/ continuity equations) perturbed by processes related to the fractional Brownian motion and obtained for bounded and measurable vector fields unique strong and path-by-path unique solutions, which are infinitely often differentiable with respect to the initial condition. See also Harang and Perkowski \cite{HP21} in the case of distributional vector fields $b$ in the Besov space $B_{p,q}^{\alpha }(\mathbb{R}^{d})$, where the drift term is given by a non-linear Young type of integral. In this context, we also refer to Kremp and Perkowski \cite{KP20a}, where the authors analyzed multidimensional SDEs with distributional drift driven by a symmetric $\alpha$-stable L\'evy processes for $\alpha\in (1,2]$. 
		
		Our results constitute a significant improvement to those in previous works on SDEs on the plane and can be seen as an analogous counterpart of those obtained by Zvonkin \cite{Zvon74}, Veretennikov \cite{Ver79}, Menoukeu-Pamen et al \cite{MMNPZ13} and Mohammed et al \cite{NilssenProske} (see also \cite{DFPR14,FlanGubiPrio10}) and thus close the gap on existence uniqueness and smoothness of solutions to multidimensional SDEs on the plane governed by merely bounded and measurable vector fields. 
		\newline
		
		Our paper is organized as follows: In Section \ref{defcon0} we introduce the basic
		mathematical framework for this article and present the main result
		(Theorem \ref{Mainresult}). In Section \ref{secbasest} we derive a central
		estimate (Proposition \ref{prop:DavieVarSheet}) for the proof of the main result
		based on the above mentioned algorithm. Finally, using tools from Malliavin
		calculus we prove the main result in Section \ref{mallcalgaus}. Section \ref{proofprere} is devoted to the proofs of preliminary results.  
		
		\section{Definitions and main results}\label{defcon0}
		\subsection{Weak and strong solutions to SDEs in the plane }\label{defcon}
		In this subsection we recall some basic definitions and concepts for solutions to SDEs driven by Wiener process in the plane that can be found in \cite{NuYe89,Ye81}. We start with the definitions of filtered probability space and $d$-dimensional Brownian sheet. We endow  $\mathcal{T}^2$   with the partial ordering $(s_1,t_1)\preceq (s_2,t_2)$ if and only if $s_1\leq s_2$ and $t_1\leq t_2$; $(s_1,t_1)\prec (s_2,t_2)$ if and only if $s_1< s_2$ and $t_1<t_2$.
		\begin{defi}
			We call a filtered probability space any probability space $(\Omega,\mathcal{F},\Pb)$ with a family   $(\mathcal{F}_{s,t},(s,t)\in \mathcal{T}^2)$ of sub-$\sigma$-algebras of $\mathcal{F}$ such that
			\begin{enumerate}
				\item $\mathcal{F}_{0,0}$ contains all null sets in $(\Omega,\mathcal{F},\Pb)$,
				\item $\{\mathcal{F}_{s,t},(s,t)\in \mathcal{T}^2\}$ is nondecreasing in the sense that  $\mathcal{F}_{s_1,t_1}\subset\mathcal{F}_{s_2,t_2}$ when $(s_1,t_1)\preceq(s_2,t_2)$,
				\item $(\mathcal{F}_{s,t},(s,t)\in \mathcal{T}^2)$ is a right-continuous system in the sense that
				$$
				\mathcal{F}_{s_1,t_1}=\bigcap\limits_{\{(s_2,t_2):\,(s_1,t_1)\prec(s_2,t_2)\}}\mathcal{F}_{s_2,t_2}.
				$$
			\end{enumerate}
		\end{defi}
		
		\begin{defi}\label{DefBSheet}
			We call a one-dimensional $(\mathcal{F}_{s,t})$-Brownian sheet on a filtered probability space $(\Omega,\mathcal{F},(\mathcal{F}_{s,t},(s,t)\in \mathcal{T}^2),\Pb)$ any real valued two-parameter stochastic process $W^{(0)}=(W^{(0)}_{s,t},(s,t)\in \mathcal{T}^2)$ satisfying the following conditions:
			\begin{enumerate}
				\item $W^{(0)}$ is $(\mathcal{F}_{s,t})$-adapted, i.e. $W^{(0)}_{s,t}$ is $\mathcal{F}_{s,t}$-measurable for every $(s,t)\in \mathcal{T}^2$.
				\item Every sample function $(s,t)\longmapsto W^{(0)}_{s,t}(\omega)$ of $W^{(0)}$ is continuous on $\mathcal{T}^2$.
				\item For every finite rectangle of the type $\Pi=]s_1,s_2]\times]t_1,t_2]\subset \mathcal{T}^2$, the random variable $$W^{(0)}(\Pi):=W^{(0)}_{s_2,t_2}-W^{(0)}_{s_1,t_2}-W^{(0)}_{s_2,t_1}+W^{(0)}_{s_1,t_1}$$
				is centered, Gaussian with variance $(s_2-s_1)(t_2-t_1)$ and independent of \textcolor{blue}{$\mathcal{F}_{s_1,T}\vee\mathcal{F}_{T,t_1}$.}
			\end{enumerate} 
			We call a $d$-dimensional Brownian sheet any $\R^d$-valued two-parameter process $W=(W^{(1)},\ldots,W^{(d)})$ such that $W^{(i)}$, $i=1,\ldots,d$ are independent one-dimensional Brownian sheets.
		\end{defi}

		In the following, we discuss the notions of weak and strong solutions to the SDE \eqref{SDEWienerSheet} (see for example \cite[Section 2]{Ye81}).  We start with the definition of a weak solution.  
		\begin{defi}\label{DefWeakSol}
			A weak solution to the SDE \eqref{SDEWienerSheet} is a system $(\Omega,\mathcal{F},(\mathcal{F}_{s,t}),W,X=(X_{s,t}),\Pb)$ such that 
			\begin{enumerate}
				\item $(\Omega,\mathcal{F},(\mathcal{F}_{s,t},(s,t)\in \mathcal{T}^2),\Pb)$ is a filtered probability space,
				\item $W=(W_{s,t},(s,t)\in \mathcal{T}^2)$ is a $d$-dimensional $(\mathcal{F}_{s,t})$-Brownian sheet with $\partial W=0$, where for a random field $Y=\{Y_{s,t},\,\,(s,t)\in\mathcal{T}^2\}$, the boundary $\partial Y$ is defined as $\partial Y:=Y_{\cdot,0}+Y_{0,\cdot}-Y_{0,0}$
				\item $X$ is $(\mathcal{F}_{s,t})$-adapted, has continuous sample paths and, $\Pb$-a.s.,
				\begin{align*}
					&	X_{s,t}-X_{s,0}-X_{0,t}+X_{0,0}\\
					&=\int_0^t\int_0^sb(s_1,t_1,X_{s_1,t_1})\mathrm{d}s_1\mathrm{d}t_1+\int_0^t\int_0^sa(s_1,t_1,X_{s_1,t_1})\,\mathrm{d}W_{s_1,t_1},\quad\forall\,(s,t)\in \mathcal{T}^2.
				\end{align*}
				
			\end{enumerate}
		\end{defi}

		We now turn to the notion of strong solution. Let $\mathcal{B}(\mathcal{V})$ (respectively $\mathcal{B}(\partial \mathcal{V})$) be the $\sigma$-algebra of Borel sets in the space $\mathcal{V}$ (respectively the boundary $\partial \mathcal{V}$ of $\mathcal{V}$) of all continuous $\R^d$-valued functions on $\mathcal{T}^2$ (respectively $\partial \mathcal{T}^2$) with respect to the metric topology of uniform convergence on compact subsets of $\mathcal{T}^2$. The subsequent definitions are borrowed from \cite{NuYe89}.
		\begin{defi}
			Let $\overline{\mathcal{B}(\mathcal{V})}$ be the completion of $\mathcal{B}(\mathcal{V})$ with respect to the Wiener measure $m_\mathcal{V}$ on $(\mathcal{V},\mathcal{B}(\mathcal{V}))$ concentrated on $\mathcal{V}_0$. For every $(s,t)\in \mathcal{T}^2$, we denote by $\mathcal{B}_{s,t}(\mathcal{V})$ the $\sigma$-algebra generated by the cylinder sets of the type $\{w\in \mathcal{V};\,w(s_1,t_1)\in E\}$ for some $(s_1,t_1)\preceq(s,t)$ and $E\in\mathcal{B}(R^d)$ and by $\overline{\mathcal{B}_{s,t}(\mathcal{V})}$ the $\sigma$-algebra generated by $\mathcal{B}_{s,t}(\mathcal{V})$ and all the null sets in $(\mathcal{V},\overline{\mathcal{B}(\mathcal{V})},m_\mathcal{V})$. Let $\overline{\mathcal{B}(\partial \mathcal{V}\times \mathcal{V})}^{\lambda\times m_\mathcal{V}}$ be the completion of $\mathcal{B}(\partial \mathcal{V}\times \mathcal{V})$ with respect to the product measure $\lambda\times m_\mathcal{V}$ for any probability measure $\lambda$ on $\partial \mathcal{V}$.
		\end{defi}
		\begin{defi}\label{DefiClassTransf}
			Let $\mathbf{T}(\partial \mathcal{V}\times \mathcal{V})$ be the class of transformations $F$ of $\partial \mathcal{V}\times \mathcal{V}$ into $\mathcal{V}$ which satisfies the condition that for every probability measure $\lambda$ on $(\partial \mathcal{V},\mathcal{B}(\partial \mathcal{V}))$, there exists a transformation $F_{\lambda}$ of $\partial \mathcal{V}\times \mathcal{V}$ into $\mathcal{V}$ such that
			\begin{enumerate}
				\item $F_{\lambda}$ is $\overline{\mathcal{B}(\partial \mathcal{V}\times \mathcal{V})}^{\lambda\times m_\mathcal{V}}/\mathcal{B}(\mathcal{V})$ measurable,
				\item For every $x\in\partial \mathcal{V}$, $F_{\lambda}[x,\cdot]$ is $\overline{\mathcal{B}_{s,t}(\mathcal{V})}/\mathcal{B}_{s,t}(\mathcal{V})$ measurable, for every $(s,t)\in \mathcal{T}^2$,
				\item There exists a null set $N_{\lambda}$ in $(\partial \mathcal{V},\mathcal{B}(\partial \mathcal{V}),\lambda)$ such that $F[x,w]=F_{\lambda}[x,w]$ for almost all $w$ in $(\mathcal{V},\overline{\mathcal{B}(\mathcal{V})},m_\mathcal{V})$ and all $x\in \partial \mathcal{V}\setminus N_{\lambda}$.
			\end{enumerate}
		\end{defi}

		Here, by a strong solution to the SDE \eqref{SDEWienerSheet} we mean a solution process, which is an adapted measurable functional of the driving
		noise given by the Wiener sheet. More precisely, the definition of such a
		solution concept is as follows:

		\begin{defi}
			Let $(X,W)$ be a weak solution to the SDE \eqref{SDEWienerSheet} on a filtered probability space $(\Omega,\mathcal{F},\{\mathcal{F}_{s,t},(s,t)\in \mathcal{T}^2\},\Pb)$ and let $\lambda$ be the probability distribution of $\partial X$. We call $(X,W)$ a strong solution to \eqref{SDEWienerSheet} if there exists a transformation $F_{\lambda}$ of $\partial \mathcal{V}\times \mathcal{V}$ into $\mathcal{V}$ satisfying Conditions 1 and 2 of Definition \ref{DefiClassTransf} such that
			\begin{align*}
				X=F_{\lambda}[\partial X,W]\,\text{ }\Pb\text{-a.s. on }\Omega.
			\end{align*}
		\end{defi}
		
		Here is a well known concept of uniqueness associated to strong solutions of \eqref{SDEWienerSheet} provided such solutions exist.
		\begin{defi}
			We say that the SDE \eqref{SDEWienerSheet} has a unique strong solution if there exists $F\in\mathbf{T}(\partial \mathcal{V}\times \mathcal{V})$ such that,
			\begin{enumerate}
				\item if $(\Omega,\mathcal{F},(\mathcal{F}_{s,t},(s,t)\in \mathcal{T}^2),\Pb)$ is a filtered probability space on which an $\R^d$-valued $(\mathcal{F}_{s,t},(s,t)\in \mathcal{T}^2)$-Brownian sheet $W$ with $\partial W=0$ exists, then for every continuous $(\mathcal{F}_{s,t},(s,t)\in \mathcal{T}^2)$-adapted boundary process $Z$ on  $(\Omega,\mathcal{F},(\mathcal{F}_{s,t},(s,t)\in \mathcal{T}^2),\Pb)$ whose probability distribution is denoted by $\lambda$, $(X,W)$ with $X=F(Z,W)$ is a weak solution of \eqref{SDEWienerSheet} with $\partial X=Z$ $\Pb$-a.s. on $\Omega$.
				\item if $(X,W)$ is a weak solution of \eqref{SDEWienerSheet} on a filtered probability space $(\Omega,\mathcal{F},(\mathcal{F}_{s,t},(s,t)\in \mathcal{T}^2),\Pb)$ and the probability distribution of $\partial X$ is denoted by $\lambda$, then $X=F_{\lambda}[\partial X,W]$ $\Pb$-a.s. on $\Omega$.
			\end{enumerate}
		\end{defi}

		\begin{remark}
			\bigskip Just as in the one-parameter case, we also observe that a strong
			solution is a weak solution (see Definition \ref{WeakSolution} in the
			Appendix). Conversely, a weak solution is not necessarily an adapted
			measurable functional of the Wiener sheet, and hence not a strong solution. 
		\end{remark}
		
		\bigskip 
		
		There are two classical notions of uniqueness associated to weak solutions (see e.g. \cite[Definitions 1.2 and 1.7]{NuYe89}). 
		\begin{defi}
			We say that the solution to the SDE \eqref{SDEWienerSheet} is unique in the sense of probability distribution if whenever $(X,W)$ and $(X^{\prime},W^{\prime})$ are two solutions of \eqref{SDEWienerSheet} on two possibly different  filtered probability spaces and $\partial X=x=\partial X^{\prime}$ for some $x\in\partial \mathcal{V}$, then $X$ and $X^{\prime}$ have the same probability distribution on $(\mathcal{V},\mathcal{B}(\mathcal{V}))$.
		\end{defi}
		\begin{defi}
			We say that the pathwise uniqueness of solutions to the SDE \eqref{SDEWienerSheet} holds if whenever $(X,W)$ and $(X^{\prime},W)$ with the same $W$ are two solutions to \eqref{SDEWienerSheet} on the same probability space and $\partial X=\partial X^{\prime}$, then $X=X^{\prime}$ for $\Pb$-a.e. $\omega\in\Omega$.
		\end{defi}

		\subsection{Preliminary on Malliavin calculus}

		In this subsection we recall some known facts on Malliavin calculus for two-parameter Brownian motion. We refer the reader to \cite{NuaSan85}.  We denote by $(z_1,z_2]$ the rectangle $\{z\in \mathcal{T}^2: z_1\prec z\preceq z_2\}$. We set $R_z=[0,z]$, and $z_1\otimes z_2=(s_1,t_2)$ if $z_1=(s_1,t_1)$ and $z_2=(s_2,t_2)$. The increment of a function $f:\mathbb{R}_+\rightarrow \mathbb{R}$ on the rectangle $(z_1,z_2]$ is given by $f((z_1,z_2])=f(z_1)-f(z_1\otimes z_2)-f(z_2\otimes z_1)+f(z_2)$. We denote by $\lambda_d$ the Lebesgue measure in $ \mathbb{R}^d$ for $d\geq 1$.
		
		Let $C_{0,A}(\mathcal{T}^2,\mathbb{R}^d)$ be the space of all continues function that vanish on the axes,$\Pb$ be the two-parameter Wiener measure and $\mathcal{F}$ be the completion of the Borel $\sigma$-algebra of $\Omega$ with respect to $\mathbb{P}$. The probability space $(\Omega,\mathcal{F},\Pb)$ is the canonical space associated with the $d$-dimensional two-parameter Wiener process, that is $\Omega=C_{0,A}(\mathcal{T}^2,\mathbb{R}^d)$ and $\mathcal{F},\Pb$ are given as above. We denote by $\mathbb{F}=\{\mathcal{F}_z,\,\,z\in \mathcal{T}^2\}$ the family of increasing $\sigma$-field generated by the function $\{\omega(r), \, \omega\in \Omega, r\preceq z\}$ and the $\Pb$-null sets. Let us remark that $\mathbb{F}$ satisfies the usual conditions in \cite{CW75}. We now introduce a set similar to the Cameron-Martin space in the one parameter case.
		\begin{align}
			H:=&\Big\{h\in\Omega; \,\text{there exists } \dot{h}^i\in L^2(\mathcal{T}^2), i=1\ldots,d, \text{ such that }\notag\\
			& h^i(z)=\int_{R_z}\dot{h}^i(r)\mathrm{d}r, \text{ for any } z\in\mathcal{T}^2 \text{ and for any }i\Big\}.
		\end{align}
		
		The set $H$ is a Hilbert space with the inner product
		$$
		\langle h_1,h_2\rangle_H=\int_{\mathcal{T}^2}\underset{i=1}{\overset{d}\sum}\dot{h}_1^i(r)\dot{h}_2^i(r)\mathrm{d}r.
		$$
		A Wiener functional is any measurable function defined on $(\Omega,\mathcal{F},\Pb)$. A Wiener functional $F: \Omega\rightarrow \mathbb{R}$ is smooth if there exist some $n\geq 1$ and an infinitely differentiable function $f$ on $\mathbb{R}^n$ such that
		\begin{enumerate}
			\item $f$ and all its derivatives are at most of polynomial growth, 
			\item $F(\omega)=f(\omega(z_1),\ldots,\omega(z_n))$ for some $z_1,\ldots, z_n \in\mathcal{T}^2$.
		\end{enumerate}
		Every smooth functional F is Fr\'echet-differentiable, and the derivative of $F$ in the direction of any vector $h\in H$ is given by
		\begin{align}
			DF(h)=&\underset{j=1}{\overset{d}\sum}\underset{i=1}{\overset{n}\sum}\frac{\partial f}{\partial x_i^j}(\omega(z_1),\ldots,\omega(z_n))h^j(z_i)\notag\\
			=&\int_{\mathcal{T}^2}\underset{j=1}{\overset{d}\sum}\underset{i=1}{\overset{n}\sum}\frac{\partial f}{\partial x_i^j}(\omega(z_1),\ldots,\omega(z_n)) 1_{R_{z_i}}(r)\dot{h}^j(r)\mathrm{d}r.
		\end{align}
		We denote by $\mathcal{D}_{1,2}$ the closed hull of family of smooth functionals with respect to the following norm:
		\begin{equation}
			\left\Vert F\right\Vert _{1,2}^{2}:=\left\Vert F\right\Vert _{L^{2}(P
				)}^{2}+\E[\left\Vert DF\right\Vert _{HS}^2],  \label{normmal1}
		\end{equation}
		where $\|\cdot\|$ is the Hilbert-Schmidt norm defined by
		$$
		\left\Vert DF\right\Vert _{HS}=\Big(\int_{\mathcal{T}^2}\underset{j=1}{\overset{d}\sum}\Big\{\underset{i=1}{\overset{n}\sum}\frac{\partial f}{\partial x_i^j}(\omega(z_1),\ldots,\omega(z_n)) 1_{R_{z_i}(r)}\Big\}^2\mathrm{d}r\Big)^{\frac{1}{2}}.
		$$

		Let $W=\{W(z);\,z\in\mathcal{T}^2\}$ be a Brownian sheet on the above probability space $(\Omega,\mathcal{F}, \Pb)$. Again for each $z\in \mathcal{T}^2$, $\mathcal{F}_z$ is the $\sigma$-algebra generated by the random variables $\{W(\xi);\xi\preceq z\}$ completed with respect to  $\Pb$. We say that a stochastic process $X=\{X(z),z\in \mathcal{T}^2\}$ is adapted if $X(z)$ is $\mathcal{F}_z$-measurable for all $z\in \mathcal{T}^2$. For $h\in L^2(\mathcal{T}^2)$, we define
		$$
		W(h)=\int_{\mathcal{T}^2}h(z)\dot W(\mathrm{d}z).
		$$
		We recall that   $\dot{W}=(\dot{W}^{(1)},\ldots,\dot{W}^{(d)})$ is the mean-zero Gaussian process indexed by the Borel field $\mathcal{B}(\mathcal{T}^2)$ on $\mathcal{T}^2$ with covariance functions
		$$
		\E\Big[\dot{W}^{(i)}(A)\dot{W}^{(j)}(B)\Big]=\delta_{i,j}\text{Leb}(A\cap B),\quad\forall\,A,B\in\mathcal{B}(\mathcal{T}^2).
		$$ 
		where $\text{Leb}(\cdot)$ denotes the Lebesgue measure on $\mathcal{T}^2$ and  $\delta_{i,j}=1$ if $i=j$ and $\delta_{i,j}=0$ otherwise. 
		
		Let $C^\infty_b(\mathbb{R}^n)$ be the space of infinitely differentiable functions with bounded derivatives of all orders. Denote by $\mathcal{S}$ the space of smooth random variables, that is the space of functions $F$ such that there exist an integer $n$ and $f\in C^\infty_b(\mathbb{R}^n)$ such that 
		$$
		F=f(W(h_1),\ldots,W(h_n)).
		$$
		For a smooth $F$ given above, its derivative is defined as a random field on $\mathcal{T}^2$ and given by
		\begin{align}
			D_zF=&\underset{i=1}{\overset{d}\sum}\frac{\partial f}{\partial x_i}(W(h_1),\ldots,W(h_n))h_i(z),
		\end{align}
		and we denote by $\mathcal{D}^{1,2}$ the closure of $\mathcal{S}$ with respect to the seminorm
		\begin{equation}
			\left\Vert F\right\Vert _{1,2}^{2}:=\left\Vert F\right\Vert _{L^{2}(P
				)}^{2}+\E[\left\Vert D_zF\right\Vert_{L^2(\mathcal{T}^2)}^2].  \label{normmal2}
		\end{equation}

		\subsection{\protect\bigskip Main results}\label{secmainres}
		
		In this subsection we present the main result of our paper on the (global)
		existence and uniqueness of strong solutions to the SDE \eqref{SDEWienerSheet}, whose proof will be given in the following sections step by step.
		
		We are coming to our main result:
		
		\begin{theorem}
			\label{Mainresult}Let $b\in L^{\infty }(\mathcal{T}^2\times \mathbb{%
				R}^{d};\mathbb{R}^{d})$. Then for all initial values $x\in \mathbb{R}^{d}$
			there exists a global unique strong solution $X_{\cdot,\cdot }^{x}$ to the SDE \eqref{SDEWienerSheet}. Moreover for all $(s,t)\in\mathcal{T}^2$ the solution $%
			X_{s,t}^{x}$ is Malliavin differentiable and 
			\begin{equation*}
				(x\longmapsto X_{s,t}^{x})\in W_{loc}^{1,2}(\mathbb{R}^{d})\text{.}
			\end{equation*}
		\end{theorem}
		\begin{remark}
			\bigskip In a forthcoming paper (see \cite{BMPP23}), it is even shown (under
			the conditions of Theorem \ref{Mainresult}) that%
			\begin{equation*}
				(x\longmapsto X_{s,t}^{x})\in W_{loc}^{2,2}(\mathbb{R}^{d})\text{ }a.e.
			\end{equation*}
		\end{remark}
		
		The idea for the proof of the existence and uniqueness result in Theorem \ref%
		{Mainresult} goes back to \cite{MBP10} and \cite{MMNPZ13} in the
		one-parameter case. The proof is based on a compactness criterion for square
		integrable functionals of the Wiener sheet from Malliavin calculus (see
		Theorem \ref{CompactnessCriterion} in the Appendix \ref{Appen}) and
		consists of the following steps:  
		
		\textbf{\underline{Step 1}:} Let $b\in L^{\infty }(\mathcal{T}^2\times \mathbb{R}%
		^{d};\mathbb{R}^{d})$ and $b_{n}:\mathcal{T}^2
		\times \mathbb{R}^{d}\longrightarrow \mathbb{R}^{d},n\geq 1$ be a sequence
		of compactly supported smooth vector fields such that%
		\begin{equation*}
			b_{n}(s,t,x)\underset{n\rightarrow \infty }{\longrightarrow }b(s,t,x)\text{ }%
			(s,t,x)\text{-a.e.}
		\end{equation*}%
		and 
		\begin{equation*}
			\sup_{n\geq 1}\left\Vert b_{n}\right\Vert _{\infty }<\infty \text{.}
		\end{equation*}%
		Further, let $X_{\cdot,\cdot }^{x,n},n\geq 1$ be the global unique strong
		solutions to the SDE (\ref{SDEWienerSheet}) associated with the vector
		fields $b_{n},n\geq 1$. By applying the Malliavin derivative $D_{\cdot,\cdot }$ in
		the direction of the Wiener sheet $W_{\cdot,\cdot }$ (see Section \ref{mallcalgaus}), we obtain
		for $0\leq u\leq s,0\leq v\leq t$ the linear equation%
		\begin{equation*}
			D_{u,v}X_{s,t}^{x,n}=\mathcal{I}_{d\times
				d}+\int_{u}^{s}\int_{v}^{t}b'_{n}(r,l,X_{r,l}^{x,n})D_{u,v}X_{r,l}^{x,n}\mathrm{d}l\mathrm{d}r\text{,}
		\end{equation*}%
		where $\mathcal{I}_{d\times d}\in \mathbb{R}^{d\times d}$ is the unit matrix and $%
		b'_{n}$ the spatial Fr\'{e}chet derivative of $b_{n}$. Using
		Picard iteration, we get the following representation of the Malliavin
		derivative:%
		\begin{equation*}
			D_{u,v}X_{s,t}^{x,n}=\mathcal{I}_{d\times d}+\sum_{n\geq
				1}\int_{
				\begin{subarray}{}
					u<r_n<\cdots<r_1<s\\
					v<l_n<\cdots<l_1<t
			\end{subarray} }b'_{n}(r_{n},l_{n},X_{r_{n},l_{n}}^{x,n})\cdots b'_{n}(r_{1},l_{1},X_{r_{1},l_{1}}^{x,n})\mathrm{d}l_{n}\mathrm{d}r_{n}\ldots \mathrm{d}l_{1}\mathrm{d}r_{1}
		\end{equation*}%
		in $L^{1}(\Omega \times \left[ u,T\right] \times \left[ v,T\right] )$ $(u,v)
		$-a.e.
		
		\textbf{\underline{Step 2}:} 
		Using the latter representation of the Malliavin derivative combined
		with Girsanov's theorem for the Wiener sheet (see Theorem \ref{Girsanov} in
		the Appendix \ref{Appen}) and an integration by parts argument, 
		we establish the estimate given in Proposition \ref%
		{prop:DavieVarSheet}. As mentioned earlier, such a bound was first obtained by Davie in \cite{Da07} in the one dimensional case and further generalised in \cite{MMNPZ13} to the $d$-dimensional case. Due to the properties of integral on the plane, one cannot simply invoke the above results to prove the estimate in Proposition \ref%
		{prop:DavieVarSheet}. To overcome this inherent difficulty, we use an algorithm for the selection of certain rectangles (see for example Section \ref{algo1}) in the plane along with Lemma \ref{lemmainres1} to verify the bounds \eqref{(i)} and \eqref%
		{(ii)} for $D_{u,v}X_{s,t}^{x,n}$ with respect to the compactness criterion
		in Theorem \ref{CompactnessCriterion}. So it follows that there exists a
		subsequence $n_{k}$ (depending on $s,t$) such that $X_{s,t}^{x,n_{k}}$
		converges to a random variable $X_{s,t}^{x}$ in $L^{2}(\Omega ;\mathbb{R}%
		^{d})$.
		
		\textbf{\underline{Step 3}:} By using the existence of a unique weak solution (see Proposition \ref%
		{UniqueWeakSolution} in the Appendix), we can finally argue that $%
		X_{s,t}^{x,n}$ converges to $X_{s,t}^{x}$ in $L^{2}(\Omega ;\mathbb{R}^{d})$
		for all $s,t$ and that $X_{\cdot,\cdot }^{x}$ is indeed the unique strong solution
		to the SDE \eqref{SDEWienerSheet}.

		\section{Basic estimates and preliminary results}\label{secbasest}
		In this section we establish the basic estimates and auxiliary results which are used in \textbf{Step 2} of the proof of Theorem \ref{Mainresult} for the derivation of certain bounds of the Malliavin derivative of the approximating solutions with respect to the compactness criterion
		in Theorem \ref{CompactnessCriterion} (see Appendix \ref{Appen}).  

		\subsection{An estimation for finite-dimensional distributions of the real valued Brownian sheet} 
		The main result of this section is Proposition \ref{prop:DavieVarSheet} which provides an estimate of the distribution of finite-dimensional projections of the Brownian sheet. It can be regarded as a non trivial generalization of the bound in \cite[Proposition 2.2]{Da07} with respect the Wiener process.
		
		\begin{prop}\label{prop:DavieVarSheet}
			Let $W=(W^{(1)},\ldots,W^{(d)})$ be a $\R^d$-valued Brownian Sheet starting from the origin and $b=(b^{(1)},\,\ldots,\,b^{(d)})$ be a compactly supported continuously differentiable $\R^d$-valued function  on $\mathcal{T}^2\times\R^d$.  Then there exists universal positive constant $C_0$ such that, for any $n\in\N$, any $l_0,l_1,\ldots,l_n\in\{1,\ldots,d\}$, any finite-dimensional marginal $(W_{s^{}_1,t^{}_1},\ldots,W_{s^{}_n,t^{}_n})$, $n\geq2$ of $W$ satisfying $s_i\neq s_j$ and $t_i\neq t_j$, for each $(i,j)$ with $i\neq j$, one has 
			\begin{align}\label{EqDavieVarSheetdD}
				\Big|\E\Big[\prod\limits_{i=1}^n\frac{\partial b^{(l_{i-1})}}{\partial x_{l_i}}(s^{}_i,t^{}_i,W_{s^{}_i,t^{}_i})\Big]\Big|\leq C_0^n\Vert b\Vert_{\infty}^n\prod\limits_{i=1}^n(s_{\gamma(i)}-s_{\gamma(i-1)})^{-1/2}(t_{\theta(i)}-t_{\theta(i-1)})^{-1/2},
			\end{align}
			where $s_{\gamma(0)}=0=t_{\theta(0)}$ and  $(\gamma,\theta)$ is a pair of permutations on $\{1,\cdots,n\}$ such that $s^{}_{\gamma(1)}<\ldots<s^{}_{\gamma(n)}$ and $t^{}_{\theta(1)}<\ldots<t^{}_{\theta(n)}$. 
		\end{prop}
		
		\begin{proof}
			See Section \ref{proofprere}.
		\end{proof}

		\begin{corollary}\label{corol:DavieVarSheet}
			Let W be an $\R^d$-valued Brownian Sheet starting from the origin and $b:\,\mathcal{T}^2\times\R^d\to\R^d$ be compactly supported continuously differentiable function. There is a positive constant $C_1$ independent of $b$ such that for any $n\in\N$, any $l_0,l_1,\ldots,l_n\in\{1,\ldots,d\}$, any $0\leq\bar{r}< r<s\leq T$ and any $0\leq\bar{u}<u<t\leq T$,  
			one has
			\begin{align} 
				&\notag	\int_{\begin{subarray}{c}
						r<s_{\gamma(n)}<\cdots<s_{\gamma(1)}<s\\
						u<t_{\theta(n)}<\cdots<t_{\theta(1)}<t	
					\end{subarray}
				}\,\Big|\E\Big[\prod\limits_{i=1}^n \frac{\partial b^{(l_{i-1})}}{\partial x_{l_i}}\left(s_i,t_{i},W_{s_i,t_{i}}\right)\Big]\Big| \mathrm{d}s_1\mathrm{d}t_1\ldots \mathrm{d}s_n\mathrm{d}t_n\\
				\leq& \frac{C_1^n\Vert b\Vert^n_{\infty}(s-r)^{n/2}(t-u)^{n/2}}{ \Gamma\left(\frac{n+1}{2}\right)^2},\label{eq:MMNPZmD}
			\end{align}
			\begin{align} 
				&\int_{\begin{subarray}{c}
						\bar{r}<s_{\sigma(k+n)}<\cdots<s_{\sigma(k+1)}<r<s_{\sigma(k)}<\cdots<s_{\sigma(1)}<s\\
						\bar{u}<t_{\pi(k+n)}<\cdots<t_{\pi(k+1)}<t_{\pi(k)}<\cdots<t_{\pi(1)}<t	
					\end{subarray}
				}\,\Big|\E\Big[\prod\limits_{i=1}^{k+n} \frac{\partial b^{(l_{i-1})}}{\partial x_{l_i}}\left(s_i,t_{i},W_{s_i,t_{i}}\right)\Big]\Big|\mathrm{d}s_1\mathrm{d}t_1\ldots \mathrm{d}s_{k+n}\mathrm{d}t_{k+n}\nonumber \\
				&	\leq \frac{C_1^{k+n}\Vert b\Vert^{k+n}_{\infty}(r-\bar{r})^{n/2}(s-r)^{k/2}(t-\bar{u})^{(k+n)/2}}{ \Gamma\left(\frac{n+1}{2}\right)\Gamma\left(\frac{k+1}{2}\right)\Gamma\left(\frac{k+n+1}{2}\right)},\label{eq:MMNPZmD2}
			\end{align}
			and
			\begin{align} 
				&	\int_{\begin{subarray}{c}
						\bar{r}<s_{\sigma(k+n)}<\cdots<s_{\sigma(k+1)}<s_{\sigma(k)}<\cdots<s_{\sigma(1)}<s\\
						\bar{u}<t_{\pi(k+n)}<\cdots<t_{\pi(k+1)}<u<t_{\pi(k)}<\cdots<t_{\pi(1)}<t	
					\end{subarray}
				}\,\Big|\E\Big[\prod\limits_{i=1}^{k+n} \frac{\partial b^{(l_{i-1})}}{\partial x_{l_i}}\left(s_i,t_{i},W_{s_i,t_{i}}\right)\Big]\Big|\mathrm{d}s_1\mathrm{d}t_1\ldots \mathrm{d}s_{k+n}\mathrm{d}t_{k+n}\notag \\
				&	\leq \frac{C_1^{k+n}\Vert b\Vert^{k+n}_{\infty}(s-\bar{r})^{(k+n)/2}(u-\bar{u})^{n/2}(t-u)^{k/2}}{ \Gamma\left(\frac{k+n+1}{2}\right)\Gamma\left(\frac{n+1}{2}\right)\Gamma\left(\frac{k+1}{2}\right)},\label{eq:MMNPZmD3}
			\end{align}
			where $\Gamma$ is the usual Gamma function, $(\gamma,\theta)$ is a couple of permutations on $\{1,\ldots,n\}$ and $(\sigma,\pi)$ is a  couple of permutations on $\{1,\ldots,k+n\}$.  
		\end{corollary}
		
		\begin{proof}
			See Section \ref{proofprere}.
		\end{proof}
		We also need the following auxiliary results whose proofs are deferred to Section \ref{proofprere}.
		
		\begin{lemma}\label{lemma:Shuffle}
			Let $k,\,m\in\N$, $T>0$, $\N_k=\{1,\ldots,k\}$, $\mathcal{T}=[0,T]$, $0\leq r<s\leq T$, $0\leq u<t\leq T$ and let $\mathcal{P}_k$  denote the set of permutations on $\N_k$.  
			Consider the set
			\begin{align*}
				\nabla^{(k)}_{r, s					}=\{(s_1,\ldots,s_k)\in\mathcal{T}^{k}:\,r<s_k<\ldots<s_1<s\}.
			\end{align*}
			Then, for any $f\in L^1(\mathcal{T})$ and $m\in\N$,
			\begin{align*}
				&\Big(\int_{
					\nabla^{(k)}_{ r, s			}}\int_{\nabla^{(k)}_{ u, t			}
				}\,\prod\limits_{i=1}^kf(s_i,t_{i})\,\mathrm{d}s_1\ldots\mathrm{d}s_k \mathrm{d}t_1\ldots\mathrm{d}t_k\Big)^m\\=&\sum\limits_{\sigma,\gamma\in\widehat{\mathcal{P}}_{km}}\int_{\nabla^{(mk,\sigma)}_{
						r, s
				}}\int_{\nabla^{(mk,\gamma)}_{
						u, t
				}}\,\prod\limits_{i=1}^{mk}f(s_i,t_{i})\,\mathrm{d}s_1\ldots\mathrm{d}s_{mk} \mathrm{d}t_1\ldots\mathrm{d}t_{mk}.
			\end{align*}
			where
			\begin{align*}
				\widehat{\mathcal{P}}_{km} =\Big\{\sigma\in\mathcal{P}^{}_{km}:\,\sigma(1+ik)<\ldots<\sigma((1+i)k),\text{ for all }i\in\{0,1,\ldots,m-1\}\Big\}
			\end{align*}
			and
			\begin{align*}
				\nabla^{(mk,\sigma)}_{
					r, s
				}=\{(s_1,\ldots,s_{mk})\in \mathcal{T}^{mk}:\,r<s_{\sigma^{-1}(mk)}<\ldots<s_{\sigma^{-1}(1)}<s\}.
			\end{align*}
		\end{lemma}
		\begin{lemma}\label{lemme:Shuffle2}
			For $k,\,\ell,\,m\in\N$, define the sets
			\begin{align*}
				\N^{\ast}_{mk}:=\{\xi_{i,j}:\,i=1,\ldots,m,\,j=1,\ldots,k\} 
			\end{align*}
			and
			\begin{align*}
				\N^{\ast\ast}_{m\ell}:=\{\zeta_{i,j}:\,i=1,\ldots,m,\,j=1,\ldots,\ell\},
			\end{align*}
			where $\xi_{i,j}=j+(i-1)(k+\ell)$ and $\zeta_{i,j}=k+j+(i-1)(k+\ell)$. Let $\mathcal{P}^{\ast}_{mk}$ (respectively $\mathcal{P}^{\ast\ast}_{m\ell}$) denote the set of permutations on $\N^{\ast}_{mk}$ (respectively $\N^{\ast\ast}_{m\ell}$) and define $\widehat{\mathcal{P}}^{\ast}_{mk}$ and $\widehat{\mathcal{P}}^{\ast\ast}_{m\ell}$ as
			\begin{align*}
				\widehat{\mathcal{P}}^{\ast}_{mk}=\{\pi\in\mathcal{P}^{\ast}_{mk}:\,\pi(\xi_{i,k})<\ldots<\pi(\xi_{i,1})\text{ for all }i=1,\ldots,m\}
			\end{align*}
			and
			\begin{align*}
				\widehat{\mathcal{P}}^{\ast\ast}_{m\ell}=\{\rho\in\mathcal{P}^{\ast}_{m\ell}:\,\rho(\zeta_{i,\ell})<\ldots<\rho(\zeta_{i,1})\text{ for all }i=1,\ldots,m\}.
			\end{align*}
			For $0\leq\bar r<r<s\leq T$,
			we consider the following subset of $\nabla_{\bar r,s}^{(k+\ell)}$:
			\begin{align*}
				\Delta^{(k+\ell)}_{\bar r,r,s}=&\{(s_1\ldots,s_{k+\ell}):\,\bar r<s_{k+\ell}<\ldots<s_{k+1}<r<s_{k}<\ldots<s_{1}<s\}.
			\end{align*}
			Then for any $0\leq \bar u<t\leq T$ and $f\in L^1(\mathcal{T}\times\mathcal{T})$,
			\begin{align*}
				&	\Big(\int_{
					\Delta_{\bar r,r,s}^{(k+\ell)}
				}\int_{\nabla_{\bar u,t}^{(k+\ell)}}\prod\limits_{j=1}^{k+\ell}f(s_j,t_{j})\,\mathrm{d}t_1\ldots\mathrm{d}t_{k+\ell}\mathrm{d}s_1\ldots\mathrm{d}s_{k+\ell}\Big)^m\\&=\sum\limits_{
					\begin{subarray}{c}
						(\pi,\rho)\in\widehat{\mathcal{P}}^{\ast}_{mk}\times\widehat{\mathcal{P}}^{\ast\ast}_{m\ell}\\\sigma\in\widehat{\mathcal{P}}_{m(k+\ell)}
					\end{subarray}
				}  \int_{\Delta_{\bar r,r,s}^{(m,k,\ell,\pi,\rho)}}\int_{\nabla_{\bar u,t}^{(m(k+\ell),\sigma)}}\prod\limits_{j=1}^{m(k+\ell)}f(s_j,t_{j})\,\mathrm{d}t_1\ldots\mathrm{d}t_{m(k+\ell)}\mathrm{d}s_1\ldots\mathrm{d}s_{m(k+\ell)},
			\end{align*}
			where  
			\begin{align*}
				&\Delta_{\bar r,r,s}^{(m,k,\ell,\pi,\rho)}\\ :=&\Big\{(s_1,\ldots,s_{m(k+\ell)}):\,\bar r<s_{\rho^{-1}(\zeta_{m,\ell})}<\ldots<s_{\rho^{-1}(\zeta_{m,1})}<s_{\rho^{-1}(\zeta_{m-1,\ell})}<\ldots< s_{\rho^{-1}(\zeta_{1,1})}\text{ }\\&\qquad<r<s_{\pi^{-1}(\xi_{m,k})}<\ldots<s_{\pi^{-1}(\xi_{m,1})}<s_{\pi^{-1}(\xi_{m-1,k})}<\ldots<s_{\pi^{-1}(\xi_{1,1})}<s\Big\}.
			\end{align*}
		\end{lemma}

		In the next subsection, we explain for the convenience of the reader the main idea of the proof of Proposition \ref{prop:DavieVarSheet}.
		\subsection{ An algorithm for the selection of certain rectangles in the plane for $d=1$}\label{algo1}
		The idea introduced in this subsection is based on an algorithm for the selection of certain rectangles in the plane. The selection of these rectangles (denoted by $z_{i,j}$, see below) is essential to prevent the occurrence of ``exploding" terms in the context of an integration by parts argument. We now aim to explain or illustrate this approach. We assume that $d=1$ and recall that we want to derive an estimate of the form%
		\begin{equation}
			\Big|\mathbb{E}\Big[ \prod\limits_{i=1}^{n}f_{i}^{\prime
			}(W_{s_{i},t_{\sigma (i)}})\Big] \Big|\leq
			C^{n}\prod\limits_{i=1}^{n}\left\Vert f_{i}\right\Vert _{L^{\infty }(%
				\mathbb{R}^{d})}\prod\limits_{i=1}^{n}((s_{i}-s_{i-1})(t_{i}-t_{i-1})^{-1/2}%
			\text{,}  \label{Estimate}
		\end{equation}%
		where $C$ is a constant, $\sigma
		\in S_{n}$ is a permutation on $\left\{ 1,\ldots,n\right\} $, $f_{i}^{\prime }$ are the derivatives of compactly supported smooth
		functions $f_{i}:\mathbb{R}^{d}\longrightarrow \mathbb{R},i=1,\ldots,n$ , $%
		0<s_{1}<\ldots<s_{n}$, $0<t_{1}<\ldots<t_{n}$. 

	Consider a rectangle and let us have a look at the following grid points of this rectangle 
			\begin{equation*}
				\mathcal{R}_{\tau _{n,\sigma }}:=\left\{ \left( s_{i},t_{\sigma (i)}\right)
				:i=1,\ldots,n\right\} \cup \left\{ (0,0)\right\} \text{.}
			\end{equation*}%
			Suppose for example that $n=3, \sigma(1)=2, \sigma(2)=1$ and $\sigma(3)=3$ then the following figure is a visualisation of the rectangle with grid points 
			$$
			\mathcal{R}_{\tau _{n,\sigma }}:=\left\{ (0,0);\left( s_{1},t_{\sigma (1)}\right);\left( s_{2},t_{\sigma (2)}\right);\left( s_{3},t_{\sigma (3)}\right)\right\} \text{.}
			$$

		\begin{center}

			\tikzset{every picture/.style={line width=0.75pt}} 
			
			\begin{tikzpicture}[x=0.30pt,y=0.30pt,yscale=-1,xscale=1]
				
				\draw  [draw opacity=0] (223.46,48) -- (427.46,48) -- (427.46,252) -- (223.46,252) -- cycle ; \draw   (291.46,48) -- (291.46,252)(359.46,48) -- (359.46,252) ; \draw   (223.46,116) -- (427.46,116)(223.46,184) -- (427.46,184) ; \draw   (223.46,48) -- (427.46,48) -- (427.46,252) -- (223.46,252) -- cycle ;
				\draw  [color={rgb, 255:red, 0; green, 0; blue, 0 }  ,draw opacity=1 ][fill={rgb, 255:red, 0; green, 0; blue, 0 }  ,fill opacity=1 ] (288.79,184) .. controls (288.79,182.53) and (289.99,181.33) .. (291.46,181.33) .. controls (292.93,181.33) and (294.13,182.53) .. (294.13,184) .. controls (294.13,185.47) and (292.93,186.67) .. (291.46,186.67) .. controls (289.99,186.67) and (288.79,185.47) .. (288.79,184) -- cycle ;
				\draw  [color={rgb, 255:red, 0; green, 0; blue, 0 }  ,draw opacity=1 ][fill={rgb, 255:red, 0; green, 0; blue, 0 }  ,fill opacity=1 ] (356.79,116) .. controls (356.79,114.53) and (357.99,113.33) .. (359.46,113.33) .. controls (360.93,113.33) and (362.13,114.53) .. (362.13,116) .. controls (362.13,117.47) and (360.93,118.67) .. (359.46,118.67) .. controls (357.99,118.67) and (356.79,117.47) .. (356.79,116) -- cycle ;
				\draw  [color={rgb, 255:red, 0; green, 0; blue, 0 }  ,draw opacity=1 ][fill={rgb, 255:red, 0; green, 0; blue, 0 }  ,fill opacity=1 ] (424.79,249.33) .. controls (424.79,247.86) and (425.99,246.67) .. (427.46,246.67) .. controls (428.93,246.67) and (430.13,247.86) .. (430.13,249.33) .. controls (430.13,250.81) and (428.93,252) .. (427.46,252) .. controls (425.99,252) and (424.79,250.81) .. (424.79,249.33) -- cycle ;
				\draw    (484.33,47) -- (361.21,114.37) ;
				\draw [shift={(359.46,115.33)}, rotate = 331.31] [color={rgb, 255:red, 0; green, 0; blue, 0 }  ][line width=0.75]    (10.93,-3.29) .. controls (6.95,-1.4) and (3.31,-0.3) .. (0,0) .. controls (3.31,0.3) and (6.95,1.4) .. (10.93,3.29)   ;
				\draw    (524.33,194) -- (431.85,248.32) ;
				\draw [shift={(430.13,249.33)}, rotate = 329.57] [color={rgb, 255:red, 0; green, 0; blue, 0 }  ][line width=0.75]    (10.93,-3.29) .. controls (6.95,-1.4) and (3.31,-0.3) .. (0,0) .. controls (3.31,0.3) and (6.95,1.4) .. (10.93,3.29)   ;
				\draw    (253.33,274) -- (290.03,185.85) ;
				\draw [shift={(290.79,184)}, rotate = 112.6] [color={rgb, 255:red, 0; green, 0; blue, 0 }  ][line width=0.75]    (10.93,-3.29) .. controls (6.95,-1.4) and (3.31,-0.3) .. (0,0) .. controls (3.31,0.3) and (6.95,1.4) .. (10.93,3.29)   ;
				
				\draw (169,11) node [anchor=north west][inner sep=0.75pt]    {$( 0,0)$};
				\draw (282,23) node [anchor=north west][inner sep=0.75pt]    {$s_{1}$};
				\draw (190,169) node [anchor=north west][inner sep=0.75pt]    {$t_{2}$};
				\draw (416,24) node [anchor=north west][inner sep=0.75pt]    {$s_{3}$};
				\draw (350,23) node [anchor=north west][inner sep=0.75pt]    {$s_{2}$};
				\draw (190,236) node [anchor=north west][inner sep=0.75pt]    {$t_{3}$};
				\draw (192,101) node [anchor=north west][inner sep=0.75pt]    {$t_{1}$};
				\draw (482,23) node [anchor=north west][inner sep=0.75pt]    {$( s_{2} ,t_{\sigma ( 2)})$};
				\draw (507,171) node [anchor=north west][inner sep=0.75pt]    {$( s_{3} ,t_{\sigma ( 3)})$};
				\draw (229,272) node [anchor=north west][inner sep=0.75pt]    {$( s_{1} ,t_{\sigma ( 1)})$};

			\end{tikzpicture}
			
		\end{center}
		
		For a given grid point $\left( s_{i},t_{\sigma (i)}\right) \in \mathcal{R}%
		_{\tau _{n,\sigma }}\setminus \left\{ (0,0)\right\} $, we define by 
		\begin{equation*}
			span(\left( s_{i},t_{\sigma (i)}\right) )=\left\{
			z_{l,k}:k=1,\ldots,i,l=1,\ldots,\sigma (i)\right\} \text{ }
		\end{equation*}%
		the rectangle ``spanned" by the point $\left( s_{i},t_{\sigma (i)}\right) $,
		where each ``variable" $z_{l,k}$ \textsl{stands for} the rectangle with corners $\left\{
		(s_{k},t_{l}),(s_{k},t_{l-1}),(s_{k-1},t_{l}),(s_{k-1},t_{l-1})\right\} $.
		Here $s_{0}=t_{0}:=0$. For example, suppose $n=3, \sigma (1)=2,\sigma (2)=1$ and $\sigma (3)=3$. The ``spanned" rectangles $\text{Span}(s_i,t_{\sigma(i)}), i=1,2,3$ can also be visualised as follows
		\begin{center}

			\tikzset{every picture/.style={line width=0.75pt}} 
			
			\begin{tikzpicture}[x=0.38pt,y=0.38pt,yscale=-1,xscale=1]
				
				\draw  [draw opacity=0] (439.33,36) -- (592.33,36) -- (592.33,189) -- (439.33,189) -- cycle ; \draw   (490.33,36) -- (490.33,189)(541.33,36) -- (541.33,189) ; \draw   (439.33,87) -- (592.33,87)(439.33,138) -- (592.33,138) ; \draw   (439.33,36) -- (592.33,36) -- (592.33,189) -- (439.33,189) -- cycle ;
				\draw  [color={rgb, 255:red, 0; green, 0; blue, 0 }  ,draw opacity=1 ][fill={rgb, 255:red, 0; green, 0; blue, 0 }  ,fill opacity=1 ] (92.33,177) .. controls (92.33,175.53) and (93.53,174.33) .. (95,174.33) .. controls (96.47,174.33) and (97.67,175.53) .. (97.67,177) .. controls (97.67,178.47) and (96.47,179.67) .. (95,179.67) .. controls (93.53,179.67) and (92.33,178.47) .. (92.33,177) -- cycle ;
				\draw  [color={rgb, 255:red, 0; green, 0; blue, 0 }  ,draw opacity=1 ][fill={rgb, 255:red, 0; green, 0; blue, 0 }  ,fill opacity=1 ] (306.33,110.67) .. controls (306.33,109.19) and (307.53,108) .. (309,108) .. controls (310.47,108) and (311.67,109.19) .. (311.67,110.67) .. controls (311.67,112.14) and (310.47,113.33) .. (309,113.33) .. controls (307.53,113.33) and (306.33,112.14) .. (306.33,110.67) -- cycle ;
				\draw  [color={rgb, 255:red, 0; green, 0; blue, 0 }  ,draw opacity=1 ][fill={rgb, 255:red, 0; green, 0; blue, 0 }  ,fill opacity=1 ] (589.33,188) .. controls (589.33,186.53) and (590.53,185.33) .. (592,185.33) .. controls (593.47,185.33) and (594.67,186.53) .. (594.67,188) .. controls (594.67,189.47) and (593.47,190.67) .. (592,190.67) .. controls (590.53,190.67) and (589.33,189.47) .. (589.33,188) -- cycle ;
				\draw    (308.33,152) -- (308.97,112.67) ;
				\draw [shift={(309,110.67)}, rotate = 90.92] [color={rgb, 255:red, 0; green, 0; blue, 0 }  ][line width=0.75]    (10.93,-3.29) .. controls (6.95,-1.4) and (3.31,-0.3) .. (0,0) .. controls (3.31,0.3) and (6.95,1.4) .. (10.93,3.29)   ;
				\draw    (593.33,234) -- (592.06,192.67) ;
				\draw [shift={(592,190.67)}, rotate = 88.24] [color={rgb, 255:red, 0; green, 0; blue, 0 }  ][line width=0.75]    (10.93,-3.29) .. controls (6.95,-1.4) and (3.31,-0.3) .. (0,0) .. controls (3.31,0.3) and (6.95,1.4) .. (10.93,3.29)   ;
				\draw  [draw opacity=0] (33.33,53) -- (95.33,53) -- (95.33,177) -- (33.33,177) -- cycle ; \draw    ; \draw   (33.33,115) -- (95.33,115) ; \draw   (33.33,53) -- (95.33,53) -- (95.33,177) -- (33.33,177) -- cycle ;
				\draw  [draw opacity=0] (210,60) -- (310,60) -- (310,110) -- (210,110) -- cycle ; \draw   (260,60) -- (260,110) ; \draw    ; \draw   (210,60) -- (310,60) -- (310,110) -- (210,110) -- cycle ;
				\draw [color={rgb, 255:red, 74; green, 144; blue, 226 }  ,draw opacity=1 ]   (62.33,264) -- (62.33,182) ;
				\draw [shift={(62.33,180)}, rotate = 90] [color={rgb, 255:red, 74; green, 144; blue, 226 }  ,draw opacity=1 ][line width=0.75]    (10.93,-3.29) .. controls (6.95,-1.4) and (3.31,-0.3) .. (0,0) .. controls (3.31,0.3) and (6.95,1.4) .. (10.93,3.29)   ;
				\draw    (121.33,204) -- (98.98,178.5) ;
				\draw [shift={(97.67,177)}, rotate = 48.76] [color={rgb, 255:red, 0; green, 0; blue, 0 }  ][line width=0.75]    (10.93,-3.29) .. controls (6.95,-1.4) and (3.31,-0.3) .. (0,0) .. controls (3.31,0.3) and (6.95,1.4) .. (10.93,3.29)   ;
				\draw [color={rgb, 255:red, 74; green, 144; blue, 226 }  ,draw opacity=1 ]   (521.33,271) -- (521.33,193) ;
				\draw [shift={(521.33,191)}, rotate = 90] [color={rgb, 255:red, 74; green, 144; blue, 226 }  ,draw opacity=1 ][line width=0.75]    (10.93,-3.29) .. controls (6.95,-1.4) and (3.31,-0.3) .. (0,0) .. controls (3.31,0.3) and (6.95,1.4) .. (10.93,3.29)   ;
				\draw [color={rgb, 255:red, 74; green, 144; blue, 226 }  ,draw opacity=1 ][fill={rgb, 255:red, 74; green, 144; blue, 226 }  ,fill opacity=1 ]   (262.33,262) -- (258.06,117) ;
				\draw [shift={(258,115)}, rotate = 88.31] [color={rgb, 255:red, 74; green, 144; blue, 226 }  ,draw opacity=1 ][line width=0.75]    (15.3,-4.61) .. controls (9.73,-1.96) and (4.63,-0.42) .. (0,0) .. controls (4.63,0.42) and (9.73,1.96) .. (15.3,4.61)   ;
				
				\draw (397,8) node [anchor=north west][inner sep=0.75pt]    {$( 0,0)$};
				\draw (481,11) node [anchor=north west][inner sep=0.75pt]    {$s_{1}$};
				\draw (413,123) node [anchor=north west][inner sep=0.75pt]    {$t_{2}$};
				\draw (584,12) node [anchor=north west][inner sep=0.75pt]    {$s_{3}$};
				\draw (533,11) node [anchor=north west][inner sep=0.75pt]    {$s_{2}$};
				\draw (415,174) node [anchor=north west][inner sep=0.75pt]    {$t_{3}$};
				\draw (417,70) node [anchor=north west][inner sep=0.75pt]    {$t_{1}$};
				\draw (266,155) node [anchor=north west][inner sep=0.75pt]    {$( s_{2} ,t_{\sigma ( 2)})$};
				\draw (563,230) node [anchor=north west][inner sep=0.75pt]    {$( s_{3} ,t_{\sigma ( 3)})$};
				\draw (91,201) node [anchor=north west][inner sep=0.75pt]    {$( s_{1} ,t_{\sigma ( 1)})$};
				\draw (471,265) node [anchor=north west][inner sep=0.75pt]  [font=\large]  {$\textcolor[rgb]{0.29,0.56,0.89}{span(( s_{3} ,t_{\sigma ( 3)}))}$};
				\draw (7,262) node [anchor=north west][inner sep=0.75pt]  [font=\large]  {$\textcolor[rgb]{0.29,0.56,0.89}{span}\textcolor[rgb]{0.29,0.56,0.89}{((}\textcolor[rgb]{0.29,0.56,0.89}{s}\textcolor[rgb]{0.29,0.56,0.89}{_{1}}\textcolor[rgb]{0.29,0.56,0.89}{,t}\textcolor[rgb]{0.29,0.56,0.89}{_{\sigma ( 1)}}\textcolor[rgb]{0.29,0.56,0.89}{))}$};
				\draw (-2,20) node [anchor=north west][inner sep=0.75pt]    {$( 0,0)$};
				\draw (6,159) node [anchor=north west][inner sep=0.75pt]    {$t_{2}$};
				\draw (7,98) node [anchor=north west][inner sep=0.75pt]    {$t_{1}$};
				\draw (87,29) node [anchor=north west][inner sep=0.75pt]    {$s_{1}$};
				\draw (48,65) node [anchor=north west][inner sep=0.75pt]  [font=\normalsize]  {$z_{11}$};
				\draw (45,123) node [anchor=north west][inner sep=0.75pt]  [font=\normalsize]  {$z_{21}$};
				\draw (552,148) node [anchor=north west][inner sep=0.75pt]  [font=\normalsize]  {$z_{33}$};
				\draw (501,149) node [anchor=north west][inner sep=0.75pt]  [font=\normalsize]  {$z_{32}$};
				\draw (450,151) node [anchor=north west][inner sep=0.75pt]  [font=\normalsize]  {$z_{31}$};
				\draw (551,99) node [anchor=north west][inner sep=0.75pt]  [font=\normalsize]  {$z_{23}$};
				\draw (501,97) node [anchor=north west][inner sep=0.75pt]  [font=\normalsize]  {$z_{22}$};
				\draw (446,97) node [anchor=north west][inner sep=0.75pt]  [font=\normalsize]  {$z_{21}$};
				\draw (552,48) node [anchor=north west][inner sep=0.75pt]  [font=\normalsize]  {$z_{13}$};
				\draw (499,48) node [anchor=north west][inner sep=0.75pt]  [font=\normalsize]  {$z_{12}$};
				\draw (448,47) node [anchor=north west][inner sep=0.75pt]  [font=\normalsize]  {$z_{11}$};
				\draw (162,30) node [anchor=north west][inner sep=0.75pt]    {$( 0,0)$};
				\draw (299,37) node [anchor=north west][inner sep=0.75pt]    {$s_{2}$};
				\draw (250,36) node [anchor=north west][inner sep=0.75pt]    {$s_{1}$};
				\draw (183,98) node [anchor=north west][inner sep=0.75pt]    {$t_{1}$};
				\draw (264,66) node [anchor=north west][inner sep=0.75pt]  [font=\normalsize]  {$z_{12}$};
				\draw (218,68) node [anchor=north west][inner sep=0.75pt]  [font=\normalsize]  {$z_{11}$};
				\draw (232,263) node [anchor=north west][inner sep=0.75pt]  [font=\large]  {$\textcolor[rgb]{0.29,0.56,0.89}{span}\textcolor[rgb]{0.29,0.56,0.89}{((}\textcolor[rgb]{0.29,0.56,0.89}{s}\textcolor[rgb]{0.29,0.56,0.89}{_{2}}\textcolor[rgb]{0.29,0.56,0.89}{,t}\textcolor[rgb]{0.29,0.56,0.89}{_{\sigma ( 2)}}\textcolor[rgb]{0.29,0.56,0.89}{))}$};

			\end{tikzpicture}

		\end{center}


		Further for a given point $\left( s_{\sigma ^{-1}(i)},t_{i}\right) $,
		denote by $\mathcal{L}^{n,i}$ the set of all $j>i$ such that $z_{i,\sigma
			^{-1}(i)}\in span\big(\big( s_{_{\sigma ^{-1}(j)}},t_{j}\big) \big)$. So for example, if $n=3, \sigma(1)=2, \sigma(2)=1,\sigma(3)=3, i=1$, then $(s_{\sigma^{-1}(1)},t_1)=(s_2,t_1)$ and $z_{1,\sigma^{-1}(1)}=z_{1,2}$ (which stands for the rectangle with right lower corner point  $(s_{\sigma^{-1}(1)},t_1)$). In this case $\text{span}((s_{\sigma^{-1}(2)},t_2))=\text{span}((s_1,t_2))$ i.e.,
		
		\begin{center}

			\tikzset{every picture/.style={line width=0.75pt}} 
			
			\begin{tikzpicture}[x=0.45pt,y=0.45pt,yscale=-1,xscale=1]
				
				\draw  [color={rgb, 255:red, 0; green, 0; blue, 0 }  ,draw opacity=1 ][fill={rgb, 255:red, 0; green, 0; blue, 0 }  ,fill opacity=1 ] (329.33,177) .. controls (329.33,175.53) and (330.53,174.33) .. (332,174.33) .. controls (333.47,174.33) and (334.67,175.53) .. (334.67,177) .. controls (334.67,178.47) and (333.47,179.67) .. (332,179.67) .. controls (330.53,179.67) and (329.33,178.47) .. (329.33,177) -- cycle ;
				\draw  [draw opacity=0] (270.33,53) -- (332.33,53) -- (332.33,177) -- (270.33,177) -- cycle ; \draw    ; \draw   (270.33,115) -- (332.33,115) ; \draw   (270.33,53) -- (332.33,53) -- (332.33,177) -- (270.33,177) -- cycle ;
				\draw    (358.33,204) -- (335.98,178.5) ;
				\draw [shift={(334.67,177)}, rotate = 48.76] [color={rgb, 255:red, 0; green, 0; blue, 0 }  ][line width=0.75]    (10.93,-3.29) .. controls (6.95,-1.4) and (3.31,-0.3) .. (0,0) .. controls (3.31,0.3) and (6.95,1.4) .. (10.93,3.29)   ;
				
				\draw (301,201) node [anchor=north west][inner sep=0.75pt]  [font=\small]  {$( s_{1} ,t_{2}) =( s_{1} ,t_{\sigma ( 1)})$};
				\draw (235,20) node [anchor=north west][inner sep=0.75pt]    {$( 0,0)$};
				\draw (243,159) node [anchor=north west][inner sep=0.75pt]    {$t_{2}$};
				\draw (244,98) node [anchor=north west][inner sep=0.75pt]    {$t_{1}$};
				\draw (324,29) node [anchor=north west][inner sep=0.75pt]    {$s_{1}$};
				\draw (285,65) node [anchor=north west][inner sep=0.75pt]  [font=\normalsize]  {$z_{11}$};
				\draw (282,123) node [anchor=north west][inner sep=0.75pt]  [font=\normalsize]  {$z_{21}$};

			\end{tikzpicture}
		\end{center}
		doesn't contain the ``rectangle" $z_{1,\sigma^{-1}(1)}=z_{1,2}$. However, $\text{span}((s_{\sigma^{-1}(3)},t_3))=\text{span}((s_3,t_3))$, i.e.,
		\begin{center}

			\tikzset{every picture/.style={line width=0.75pt}} 
			
			\begin{tikzpicture}[x=0.40pt,y=0.40pt,yscale=-1,xscale=1]
				
				\draw  [draw opacity=0] (225,63) -- (378,63) -- (378,216) -- (225,216) -- cycle ; \draw   (276,63) -- (276,216)(327,63) -- (327,216) ; \draw   (225,114) -- (378,114)(225,165) -- (378,165) ; \draw   (225,63) -- (378,63) -- (378,216) -- (225,216) -- cycle ;
				\draw  [color={rgb, 255:red, 0; green, 0; blue, 0 }  ,draw opacity=1 ][fill={rgb, 255:red, 0; green, 0; blue, 0 }  ,fill opacity=1 ] (375.33,216) .. controls (375.33,214.53) and (376.53,213.33) .. (378,213.33) .. controls (379.47,213.33) and (380.67,214.53) .. (380.67,216) .. controls (380.67,217.47) and (379.47,218.67) .. (378,218.67) .. controls (376.53,218.67) and (375.33,217.47) .. (375.33,216) -- cycle ;
				\draw    (413.33,217) -- (381.33,216.06) ;
				\draw [shift={(379.33,216)}, rotate = 1.68] [color={rgb, 255:red, 0; green, 0; blue, 0 }  ][line width=0.75]    (10.93,-3.29) .. controls (6.95,-1.4) and (3.31,-0.3) .. (0,0) .. controls (3.31,0.3) and (6.95,1.4) .. (10.93,3.29)   ;
				
				\draw (183,36) node [anchor=north west][inner sep=0.75pt]    {$( 0,0)$};
				\draw (267,39) node [anchor=north west][inner sep=0.75pt]    {$s_{1}$};
				\draw (199,151) node [anchor=north west][inner sep=0.75pt]    {$t_{2}$};
				\draw (370,40) node [anchor=north west][inner sep=0.75pt]    {$s_{3}$};
				\draw (319,39) node [anchor=north west][inner sep=0.75pt]    {$s_{2}$};
				\draw (201,202) node [anchor=north west][inner sep=0.75pt]    {$t_{3}$};
				\draw (203,98) node [anchor=north west][inner sep=0.75pt]    {$t_{1}$};
				\draw (415,207) node [anchor=north west][inner sep=0.75pt]  [font=\small]  {$( s_{3} ,t_{3}) =( s_{3} ,t_{\sigma ( 3)})$};
				\draw (338,176) node [anchor=north west][inner sep=0.75pt]  [font=\normalsize]  {$z_{33}$};
				\draw (287,177) node [anchor=north west][inner sep=0.75pt]  [font=\normalsize]  {$z_{32}$};
				\draw (236,179) node [anchor=north west][inner sep=0.75pt]  [font=\normalsize]  {$z_{31}$};
				\draw (337,127) node [anchor=north west][inner sep=0.75pt]  [font=\normalsize]  {$z_{23}$};
				\draw (287,125) node [anchor=north west][inner sep=0.75pt]  [font=\normalsize]  {$z_{22}$};
				\draw (232,125) node [anchor=north west][inner sep=0.75pt]  [font=\normalsize]  {$z_{21}$};
				\draw (338,76) node [anchor=north west][inner sep=0.75pt]  [font=\normalsize]  {$z_{13}$};
				\draw (285,76) node [anchor=north west][inner sep=0.75pt]  [font=\normalsize]  {$z_{12}$};
				\draw (234,75) node [anchor=north west][inner sep=0.75pt]  [font=\normalsize]  {$z_{11}$};

			\end{tikzpicture}
		\end{center}
		contains $z_{1,\sigma^{-1}(1)}=z_{1,2}$. Hence $\mathcal{L}^{n,i}=\mathcal{L}^{3,1}=\{3\}$. Further, we say
		for $l\geq 1$ that $z_{i,\sigma ^{-1}(i)+l}$ is a \emph{substitution variable%
		} for $z_{i,\sigma ^{-1}(i)}$, if $z_{i,\sigma ^{-1}(i)+l}\in span\big(\big(
		s_{_{\sigma ^{-1}(j)}},t_{j}\big)\big)$ for all $j\in \mathcal{L}^{n,i}$.
		Obviously, if $\mathcal{L}^{n,i}\neq \emptyset $, then $z_{i,\sigma
			^{-1}(i)+1}\in span\big(\big( s_{_{\sigma ^{-1}(j)}},t_{j}\big)\big)$ for all $%
		j\in \mathcal{L}^{n,i}$. So $z_{i,\sigma ^{-1}(i)+1}$ is a substitution
		variable for $z_{i,\sigma ^{-1}(i)}$. So using the previous example for $i=1$ with $\mathcal{L}^{3,1}=\{3\},$ we have only one substitution variable of
		$$
		z_{i,\sigma^{-1}(i)}=z_{1,2} \text{ i.e., } z_{i,\sigma^{-1}(i)+1}=z_{1,3}
		$$
		\begin{center}

			\tikzset{every picture/.style={line width=0.75pt}} 
			
			\begin{tikzpicture}[x=0.42pt,y=0.42pt,yscale=-1,xscale=1]
				
				\draw  [draw opacity=0] (225,64) -- (378,64) -- (378,217) -- (225,217) -- cycle ; \draw   (276,64) -- (276,217)(327,64) -- (327,217) ; \draw   (225,115) -- (378,115)(225,166) -- (378,166) ; \draw   (225,64) -- (378,64) -- (378,217) -- (225,217) -- cycle ;
				\draw  [color={rgb, 255:red, 0; green, 0; blue, 0 }  ,draw opacity=1 ][fill={rgb, 255:red, 0; green, 0; blue, 0 }  ,fill opacity=1 ] (375.33,217) .. controls (375.33,215.53) and (376.53,214.33) .. (378,214.33) .. controls (379.47,214.33) and (380.67,215.53) .. (380.67,217) .. controls (380.67,218.47) and (379.47,219.67) .. (378,219.67) .. controls (376.53,219.67) and (375.33,218.47) .. (375.33,217) -- cycle ;
				\draw    (413.33,218) -- (381.33,217.06) ;
				\draw [shift={(379.33,217)}, rotate = 1.68] [color={rgb, 255:red, 0; green, 0; blue, 0 }  ][line width=0.75]    (10.93,-3.29) .. controls (6.95,-1.4) and (3.31,-0.3) .. (0,0) .. controls (3.31,0.3) and (6.95,1.4) .. (10.93,3.29)   ;
				\draw [color={rgb, 255:red, 208; green, 2; blue, 27 }  ,draw opacity=1 ]   (226,85) -- (250.33,65) ;
				\draw [color={rgb, 255:red, 208; green, 2; blue, 27 }  ,draw opacity=1 ]   (225,115) -- (276,64) ;
				\draw [color={rgb, 255:red, 208; green, 2; blue, 27 }  ,draw opacity=1 ]   (237.33,114) -- (275.33,81) ;
				\draw [color={rgb, 255:red, 208; green, 2; blue, 27 }  ,draw opacity=1 ]   (290.33,113) -- (327.33,77) ;
				\draw [color={rgb, 255:red, 208; green, 2; blue, 27 }  ,draw opacity=1 ]   (301.33,114) -- (326.33,91) ;
				\draw [color={rgb, 255:red, 208; green, 2; blue, 27 }  ,draw opacity=1 ]   (276,115) -- (327,64) ;
				\draw [color={rgb, 255:red, 208; green, 2; blue, 27 }  ,draw opacity=1 ]   (226.33,100) -- (266.33,64) ;
				\draw [color={rgb, 255:red, 208; green, 2; blue, 27 }  ,draw opacity=1 ]   (276.33,94) -- (310.33,65) ;
				\draw [color={rgb, 255:red, 208; green, 2; blue, 27 }  ,draw opacity=1 ]   (252,114) -- (276.33,94) ;
				\draw [color={rgb, 255:red, 208; green, 2; blue, 27 }  ,draw opacity=1 ]   (275,83) -- (298.33,64) ;
				\draw [color={rgb, 255:red, 74; green, 144; blue, 226 }  ,draw opacity=1 ]   (328,97) -- (349.33,116) ;
				\draw [color={rgb, 255:red, 74; green, 144; blue, 226 }  ,draw opacity=1 ]   (326.33,83) -- (362.33,114) ;
				\draw [color={rgb, 255:red, 74; green, 144; blue, 226 }  ,draw opacity=1 ]   (342,65) -- (378.33,100) ;
				\draw [color={rgb, 255:red, 74; green, 144; blue, 226 }  ,draw opacity=1 ]   (327,64) -- (378,115) ;
				\draw [color={rgb, 255:red, 74; green, 144; blue, 226 }  ,draw opacity=1 ]   (357,63) -- (378.33,82) ;
				\draw [color={rgb, 255:red, 208; green, 2; blue, 27 }  ,draw opacity=1 ]   (57,119) -- (249.37,79.4) ;
				\draw [shift={(251.33,79)}, rotate = 168.37] [color={rgb, 255:red, 208; green, 2; blue, 27 }  ,draw opacity=1 ][line width=0.75]    (10.93,-3.29) .. controls (6.95,-1.4) and (3.31,-0.3) .. (0,0) .. controls (3.31,0.3) and (6.95,1.4) .. (10.93,3.29)   ;
				\draw    (307.33,271) -- (306.37,223) ;
				\draw [shift={(306.33,221)}, rotate = 88.85] [color={rgb, 255:red, 0; green, 0; blue, 0 }  ][line width=0.75]    (10.93,-3.29) .. controls (6.95,-1.4) and (3.31,-0.3) .. (0,0) .. controls (3.31,0.3) and (6.95,1.4) .. (10.93,3.29)   ;
				\draw [color={rgb, 255:red, 74; green, 144; blue, 226 }  ,draw opacity=1 ]   (467.33,92) -- (362.33,92) ;
				\draw [shift={(360.33,92)}, rotate = 360] [color={rgb, 255:red, 74; green, 144; blue, 226 }  ,draw opacity=1 ][line width=0.75]    (10.93,-3.29) .. controls (6.95,-1.4) and (3.31,-0.3) .. (0,0) .. controls (3.31,0.3) and (6.95,1.4) .. (10.93,3.29)   ;
				
				\draw (183,37) node [anchor=north west][inner sep=0.75pt]    {$( 0,0)$};
				\draw (267,40) node [anchor=north west][inner sep=0.75pt]    {$s_{1}$};
				\draw (199,152) node [anchor=north west][inner sep=0.75pt]    {$t_{2}$};
				\draw (370,41) node [anchor=north west][inner sep=0.75pt]    {$s_{3}$};
				\draw (319,40) node [anchor=north west][inner sep=0.75pt]    {$s_{2}$};
				\draw (201,203) node [anchor=north west][inner sep=0.75pt]    {$t_{3}$};
				\draw (203,99) node [anchor=north west][inner sep=0.75pt]    {$t_{1}$};
				\draw (414,208) node [anchor=north west][inner sep=0.75pt]  [font=\small]  {$( s_{3} ,t_{3}) =( s_{3} ,t_{\sigma ( 3)})$};
				\draw (338,177) node [anchor=north west][inner sep=0.75pt]  [font=\normalsize]  {$z_{33}$};
				\draw (287,178) node [anchor=north west][inner sep=0.75pt]  [font=\normalsize]  {$z_{32}$};
				\draw (236,180) node [anchor=north west][inner sep=0.75pt]  [font=\normalsize]  {$z_{31}$};
				\draw (337,128) node [anchor=north west][inner sep=0.75pt]  [font=\normalsize]  {$z_{23}$};
				\draw (287,126) node [anchor=north west][inner sep=0.75pt]  [font=\normalsize]  {$z_{22}$};
				\draw (232,126) node [anchor=north west][inner sep=0.75pt]  [font=\normalsize]  {$z_{21}$};
				\draw (338,77) node [anchor=north west][inner sep=0.75pt]  [font=\normalsize]  {$z_{13}$};
				\draw (285,77) node [anchor=north west][inner sep=0.75pt]  [font=\normalsize]  {$z_{12}$};
				\draw (238.17,83) node [anchor=north west][inner sep=0.75pt]  [font=\normalsize]  {$z_{11}$};
				\draw (11,115) node [anchor=north west][inner sep=0.75pt]  [font=\large]  {$\textcolor[rgb]{0.82,0.01,0.11}{span}\textcolor[rgb]{0.82,0.01,0.11}{(}\textcolor[rgb]{0.82,0.01,0.11}{s}\textcolor[rgb]{0.82,0.01,0.11}{_{2}}\textcolor[rgb]{0.82,0.01,0.11}{,t}\textcolor[rgb]{0.82,0.01,0.11}{_{1}}\textcolor[rgb]{0.82,0.01,0.11}{)}$};
				\draw (256,270) node [anchor=north west][inner sep=0.75pt]  [font=\large,color={rgb, 255:red, 245; green, 166; blue, 35 }  ,opacity=1 ]  {$\textcolor[rgb]{0.96,0.65,0.14}{span( s_{3} ,t_{\sigma ( 3)})}$};
				\draw (472,71) node [anchor=north west][inner sep=0.75pt]   [align=left] {Substitution variable of \ };
				\draw (523,97) node [anchor=north west][inner sep=0.75pt]    {$z_{i,\sigma ^{-1}( i)} =z_{1,2}$};

			\end{tikzpicture}

		\end{center}
		
		\text{ }\\

		In what follows, we also call the
		pairs $\mathcal{O}_{i}=(z_{i,\sigma ^{-1}(i)},z_{i,\sigma
			^{-1}(i)+1}),i=1,\ldots,n$ \emph{orientation points}. Here by convention, we
		set $\mathcal{O}_{i}=z_{i,\sigma ^{-1}(i)}$, if $z_{i,\sigma ^{-1}(i)+1}$ is
		not a substitution variable. Hence, if $\sigma (1)=2,\sigma (2)=1,\sigma
		(3)=3$ in the previous example, the orientation points are given by

		\begin{center}

			\tikzset{every picture/.style={line width=0.75pt}} 
			
			\begin{tikzpicture}[x=0.42pt,y=0.42pt,yscale=-1,xscale=1]
				
				\draw  [draw opacity=0] (206.33,71) -- (359.33,71) -- (359.33,224) -- (206.33,224) -- cycle ; \draw   (257.33,71) -- (257.33,224)(308.33,71) -- (308.33,224) ; \draw   (206.33,122) -- (359.33,122)(206.33,173) -- (359.33,173) ; \draw   (206.33,71) -- (359.33,71) -- (359.33,224) -- (206.33,224) -- cycle ;
				\draw  [color={rgb, 255:red, 0; green, 0; blue, 0 }  ,draw opacity=1 ][fill={rgb, 255:red, 0; green, 0; blue, 0 }  ,fill opacity=1 ] (356.33,223) .. controls (356.33,221.53) and (357.53,220.33) .. (359,220.33) .. controls (360.47,220.33) and (361.67,221.53) .. (361.67,223) .. controls (361.67,224.47) and (360.47,225.67) .. (359,225.67) .. controls (357.53,225.67) and (356.33,224.47) .. (356.33,223) -- cycle ;
				\draw   (265,87) -- (304.33,87) -- (304.33,110) -- (265,110) -- cycle ;
				\draw   (263,135) -- (303.33,135) -- (303.33,159) -- (263,159) -- cycle ;
				\draw   (315,86) -- (357.33,86) -- (357.33,108) -- (315,108) -- cycle ;
				\draw   (313,189) -- (357.33,189) -- (357.33,212) -- (313,212) -- cycle ;
				\draw   (212,136) -- (254.33,136) -- (254.33,161) -- (212,161) -- cycle ;
				\draw  [draw opacity=0] (287.12,86.11) .. controls (286.39,85.44) and (286,84.74) .. (286,84) .. controls (286,80.13) and (296.82,77) .. (310.17,77) .. controls (323.51,77) and (334.33,80.13) .. (334.33,84) .. controls (334.33,85.05) and (333.54,86.05) .. (332.11,86.94) -- (310.17,84) -- cycle ; \draw  [color={rgb, 255:red, 208; green, 2; blue, 27 }  ,draw opacity=1 ] (287.12,86.11) .. controls (286.39,85.44) and (286,84.74) .. (286,84) .. controls (286,80.13) and (296.82,77) .. (310.17,77) .. controls (323.51,77) and (334.33,80.13) .. (334.33,84) .. controls (334.33,85.05) and (333.54,86.05) .. (332.11,86.94) ;  
				\draw  [draw opacity=0] (225.12,135.11) .. controls (224.39,134.44) and (224,133.74) .. (224,133) .. controls (224,129.13) and (234.82,126) .. (248.17,126) .. controls (261.51,126) and (272.33,129.13) .. (272.33,133) .. controls (272.33,134.05) and (271.54,135.05) .. (270.11,135.94) -- (248.17,133) -- cycle ; \draw  [color={rgb, 255:red, 208; green, 2; blue, 27 }  ,draw opacity=1 ] (225.12,135.11) .. controls (224.39,134.44) and (224,133.74) .. (224,133) .. controls (224,129.13) and (234.82,126) .. (248.17,126) .. controls (261.51,126) and (272.33,129.13) .. (272.33,133) .. controls (272.33,134.05) and (271.54,135.05) .. (270.11,135.94) ;  
				\draw [color={rgb, 255:red, 208; green, 2; blue, 27 }  ,draw opacity=1 ]   (502.33,70) .. controls (476.59,53.17) and (366.56,48.1) .. (315.85,77.11) ;
				\draw [shift={(314.33,78)}, rotate = 329.04] [color={rgb, 255:red, 208; green, 2; blue, 27 }  ,draw opacity=1 ][line width=0.75]    (10.93,-3.29) .. controls (6.95,-1.4) and (3.31,-0.3) .. (0,0) .. controls (3.31,0.3) and (6.95,1.4) .. (10.93,3.29)   ;
				\draw [color={rgb, 255:red, 208; green, 2; blue, 27 }  ,draw opacity=1 ]   (502.33,70) .. controls (505.3,97.72) and (310.29,106.82) .. (266.61,128.34) ;
				\draw [shift={(265.33,129)}, rotate = 331.78] [color={rgb, 255:red, 208; green, 2; blue, 27 }  ,draw opacity=1 ][line width=0.75]    (10.93,-3.29) .. controls (6.95,-1.4) and (3.31,-0.3) .. (0,0) .. controls (3.31,0.3) and (6.95,1.4) .. (10.93,3.29)   ;
				\draw [color={rgb, 255:red, 208; green, 2; blue, 27 }  ,draw opacity=1 ]   (502.33,70) .. controls (506.27,91.67) and (356.92,167.67) .. (328.23,187.15) ;
				\draw [shift={(327,188)}, rotate = 324.61] [color={rgb, 255:red, 208; green, 2; blue, 27 }  ,draw opacity=1 ][line width=0.75]    (10.93,-3.29) .. controls (6.95,-1.4) and (3.31,-0.3) .. (0,0) .. controls (3.31,0.3) and (6.95,1.4) .. (10.93,3.29)   ;
				
				\draw (164,43) node [anchor=north west][inner sep=0.75pt]    {$( 0,0)$};
				\draw (248,46) node [anchor=north west][inner sep=0.75pt]    {$s_{1}$};
				\draw (180,158) node [anchor=north west][inner sep=0.75pt]    {$t_{2}$};
				\draw (351,47) node [anchor=north west][inner sep=0.75pt]    {$s_{3}$};
				\draw (300,46) node [anchor=north west][inner sep=0.75pt]    {$s_{2}$};
				\draw (182,209) node [anchor=north west][inner sep=0.75pt]    {$t_{3}$};
				\draw (184,105) node [anchor=north west][inner sep=0.75pt]    {$t_{1}$};
				\draw (318,188) node [anchor=north west][inner sep=0.75pt]  [font=\normalsize]  {$z_{33}$};
				\draw (268,135) node [anchor=north west][inner sep=0.75pt]  [font=\normalsize]  {$z_{22}$};
				\draw (220,136) node [anchor=north west][inner sep=0.75pt]  [font=\normalsize]  {$z_{21}$};
				\draw (319,86) node [anchor=north west][inner sep=0.75pt]  [font=\normalsize]  {$z_{13}$};
				\draw (270,87) node [anchor=north west][inner sep=0.75pt]  [font=\normalsize]  {$z_{12}$};
				\draw (503,60) node [anchor=north west][inner sep=0.75pt]   [align=left] {orientation points};

			\end{tikzpicture}
		\end{center}
		Using the latter notation, let us first illustrate for the previous example,
		how we can obtain an estimate of the type (\ref{Estimate}). For this
		purpose, denote by $E_{i,j}$ the heat kernel given by 
		\begin{equation*}
			E_{i,j}(z)=\frac{1}{\sqrt{2\pi (s_{i}-s_{i-1})(t_{j}-t_{j-1})}}\exp (-\frac{%
				z^{2}}{2(s_{i}-s_{i-1})(t_{j}-t_{j-1})})
		\end{equation*}%
		and its derivative by $B_{i,j}$. So 
		\begin{align*}
			&\mathbb{E}\Big[ \prod\limits_{i=1}^{3}f_{i}^{\prime
			}(W_{s_{i},t_{\sigma (i)}})\Big] \\
			=&\int_{\mathbb{R}^{3\times 3}}(\prod\limits_{i=1}^{3}f_{i}^{\prime
			}(\sum_{l=1}^{\sigma
				(i)}\sum_{k=1}^{i}z_{l,k}))\prod%
			\limits_{l,k=1}^{3}E_{k,l}(z_{l,k})\diffns z_{1,1}\ldots \diffns z_{3,3} \\
			=&\int_{\mathbb{R}^{3\times 3}}f_{\sigma ^{-1}(1)}^{\prime
			}(\sum_{l=1}^{1}\sum_{k=1}^{\sigma ^{-1}(1)}z_{l,k})f_{\sigma
				^{-1}(2)}^{\prime }(\sum_{l=1}^{2}\sum_{k=1}^{\sigma
				^{-1}(2)}z_{l,k})f_{\sigma ^{-1}(3)}^{\prime
			}(\sum_{l=1}^{3}\sum_{k=1}^{\sigma ^{-1}(3)}z_{l,k})\\&\qquad\times \prod\limits_{l,k=1}^{3}E_{k,l}(z_{l,k})\diffns z_{1,1}\ldots \diffns z_{3,3}.
		\end{align*}
		
		We start with the first orientation point $(z_{1,2},z_{1,3})$. Then
		we can use the substitution variable $z_{1,3}$ to eliminate the variable $z_{1,2}$ in the other factors (here it is only $f_{\sigma
			^{-1}(3)}^{\prime }(\sum_{l=1}^{3}\sum_{k=1}^{\sigma ^{-1}(3)}z_{l,k})$).
		So%
		\begin{align*}
			&\mathbb{E}\Big[ \prod\limits_{i=1}^{3}f_{i}^{\prime
			}(W_{s_{i},t_{\sigma (i)}})\Big]  \\
			=&\int_{\mathbb{R}^{3\times 3}}f_{\sigma ^{-1}(1)}^{\prime
			}(\sum_{(l,k)\in \left\{ 1\right\} \times \left\{ 1,\ldots,\sigma
				^{-1}(1)\right\} }z_{l,k})f_{\sigma ^{-1}(2)}^{\prime }(\sum_{\substack{ %
					(l,k)\in \left\{ 1,2\right\} \times \left\{ 1,\ldots,\sigma ^{-1}(2)\right\} 
					\\ (l,k)\neq (1,\sigma ^{-1}(1))}}z_{l,k})\\&\qquad\times f_{\sigma ^{-1}(3)}^{\prime
			}(\sum_{\substack{ (l,k)\in \left\{ 1,2,3\right\} \times \left\{
					1,\ldots,\sigma ^{-1}(3)\right\}  \\ (l,k)\neq (1,\sigma ^{-1}(1))}}%
			z_{l,k}) 
			E_{2,1}(z_{1,2})E_{3,1}(z_{1,3}-z_{1,2})\\&\qquad\times\prod\limits_{\substack{ l,k=1  %
					(l,k)\neq (1,2),(1,3)}}^{3}E_{k,l}(z_{l,k})\diffns z_{1,1}\ldots \diffns z_{3,3}.
		\end{align*}%
		Then we can apply integration by parts with respect to the \emph{integration
			by parts variable }$z_{1,\sigma ^{-1}(1)}=z_{1,2}$ and get that%
		\begin{eqnarray*}
			\mathbb{E}\Big[ \prod\limits_{i=1}^{3}f_{i}^{\prime
			}(W_{s_{i},t_{\sigma (i)}})\Big] =-(\mathcal{I}_{1,1}+\mathcal{I}_{1,2})\text{,}
		\end{eqnarray*}%
		where%
		\begin{align*}
			\mathcal{I}_{1,1} =&\int_{\mathbb{R}^{3\times 3}}f_{\sigma
				^{-1}(1)}(\sum_{(l,k)\in \left\{ 1\right\} \times \left\{ 1,\ldots,\sigma
				^{-1}(1)\right\} }z_{l,k})f_{\sigma ^{-1}(2)}^{\prime }(\sum_{\substack{ %
					(l,k)\in \left\{ 1,2\right\} \times \left\{ 1,\ldots,\sigma ^{-1}(2)\right\} 
					\\ (l,k)\neq (1,\sigma ^{-1}(1))}}z_{l,k})\\&\times f_{\sigma ^{-1}(3)}^{\prime
			}(\sum_{\substack{ (l,k)\in \left\{ 1,2,3\right\} \times \left\{
					1,\ldots,\sigma ^{-1}(3)\right\}  \\ (l,k)\neq (1,\sigma ^{-1}(1))}}%
			z_{l,k})  \\
			&\times B_{2,1}(z_{1,2})E_{3,1}(z_{1,3}-z_{1,2})\prod\limits_{\substack{ l,k=1 \\ %
					(l,k)\neq (1,2),(1,3)}}^{3}E_{k,l}(z_{l,k})\diffns z_{1,1}\ldots \diffns z_{3,3}
		\end{align*}%
		and 
		\begin{align*}
			\mathcal{I}_{1,2} =&\int_{\mathbb{R}^{3\times 3}}f_{\sigma
				^{-1}(1)}(\sum_{(l,k)\in \left\{ 1\right\} \times \left\{ 1,\ldots,\sigma
				^{-1}(1)\right\} }z_{l,k})\times f_{\sigma ^{-1}(2)}^{\prime}(\sum_{\substack{ (l,k)\in \left\{
					1,2\right\} \times \left\{ 1,\ldots,\sigma ^{-1}(2)\right\}  \\ (l,k)\neq
					(1,\sigma ^{-1}(1))}}z_{l,k}))\\&\times f_{\sigma ^{-1}(3)}^{\prime}(\sum
			_{\substack{ (l,k)\in \left\{ 1,2,3\right\} \times \left\{ 1,\ldots,\sigma
					^{-1}(3)\right\}  \\ (l,k)\neq (1,\sigma ^{-1}(1))}}z_{l,k})  (-E_{2,1}(z_{1,2})B_{3,1}(z_{1,3}-z_{1,2}))\\
			&\times  \prod\limits_{\substack{ l,k=1 \\ (l,k)\neq (1,2),(1,3)}}%
			^{3}E_{k,l}(z_{l,k})\diffns z_{1,1}\ldots \diffns z_{3,3}.
		\end{align*}%
		Then, we go to the next orientation point $(z_{2,1},z_{2,2})$ and consider
		e.g. the term $\mathcal{I}_{1,1}$. We then choose a new substitution
		variable \emph{next to the right} (in the same row) outside the column of $%
		z_{1,2}$, that is $z_{2,3}$. Then $z_{2,1}$ is our next integration by parts
		variable and $z_{2,3}$ our substitution variable for $z_{2,1}$. Hence, we
		get that%
		\begin{equation*}
			\mathcal{I}_{1,1}=-(\mathcal{I}_{1,1,1}+\mathcal{I}_{1,1,2})
		\end{equation*}%
		where%
		\begin{align*}
			\mathcal{I}_{1,1,1} =&\int_{\mathbb{R}^{3\times 3}}(f_{\sigma
				^{-1}(1)}(\sum_{(l,k)\in \left\{ 1\right\} \times \left\{ 1,\ldots,\sigma
				^{-1}(1)\right\} }z_{l,k}))f_{\sigma ^{-1}(2)}(\sum_{\substack{ (l,k)\in
					\left\{ 1,2\right\} \times \left\{ 1,\ldots,\sigma ^{-1}(2)\right\}  \\ %
					(l,k)\neq (1,\sigma ^{-1}(1))}}z_{l,k}) \\
			&\times f_{\sigma ^{-1}(3)}^{\prime}(\sum_{\substack{ (l,k)\in \left\{
					1,2,3\right\} \times \left\{ 1,\ldots,\sigma ^{-1}(3)\right\}  \\ (l,k)\neq
					(1,\sigma ^{-1}(1)),(2,\sigma ^{-1}(2))}}z_{l,k})(B_{1,2}(z_{2,1})E_{3,2}(z_{2,3}-z_{2,1}))B_{2,1}(z_{1,2})  \\
			&\times E_{3,1}(z_{1,3}-z_{1,2})\times 
			\prod\limits_{\substack{ l,k=1 \\ (l,k)\neq (1,2),(1,3),(2,1),(2,3)}}%
			^{3}E_{k,l}(z_{l,k})\diffns z_{1,1}\ldots \diffns z_{3,3}
		\end{align*}%
		and 
		\begin{eqnarray*}
			\mathcal{I}_{1,1,2} &=&\int_{\mathbb{R}^{3\times 3}}(f_{\sigma
				^{-1}(1)}(\sum_{(l,k)\in \left\{ 1\right\} \times \left\{ 1,\dots,\sigma
				^{-1}(1)\right\} }z_{l,k}))f_{\sigma ^{-1}(2)}(\sum_{\substack{ (l,k)\in
					\left\{ 1,2\right\} \times \left\{ 1,\dots,\sigma ^{-1}(2)\right\}  \\ %
					(l,k)\neq (1,\sigma ^{-1}(1))}}z_{l,k})  \\
			&&\times f_{\sigma ^{-1}(3)}^{\prime}(\sum_{\substack{ (l,k)\in \left\{
					1,2,3\right\} \times \left\{ 1,\ldots,\sigma ^{-1}(3)\right\}  \\ (l,k)\neq
					(1,\sigma ^{-1}(1)),(2,\sigma ^{-1}(2))}}z_{l,k})(-E_{1,2}(z_{2,1})B_{3,2}(z_{2,3}-z_{2,1}))  \\
			&&\times B_{2,1}(z_{1,2})E_{3,1}(z_{1,3}-z_{1,2}) 
			\prod\limits_{\substack{ l,k=1 \\ (l,k)\neq (1,2),(1,3),(2,1),(2,3)}}%
			^{3}E_{k,l}(z_{l,k})\diffns z_{1,1}\dots \diffns z_{3,3}.
		\end{eqnarray*}%
		Let us now move from the orientation point $(z_{2,1},z_{2,2})$ to the next
		one, that is $z_{3,3}$, which does not have a substitution variable.
		Consider e.g. the term $\mathcal{I}_{1,1,2}$. Then we proceed as in the
		previous step: We observe here that $z_{3,3}$ is in the same column as the
		previous substitution variable $z_{2,3}$ (which corresponds to the factor $%
		B_{3,2}(z_{2,3}-z_{2,1})$ in $\mathcal{I}_{1,1,2}$). However, in this case
		we select a new integration by parts variable (instead of a new substitution
		variable as in the previous step), that is a variable \emph{next to the left}
		(in the same row), but outside of the columns of $z_{1,2}$ (which
		corresponds to $B_{1,2}$ in $\mathcal{I}_{1,1}$) and $z_{2,3}$
		(corresponding to $B_{3,2}$ in $\mathcal{I}_{1,1,2}$). Hence, the new
		integration by parts variable in the last step (without a substitution
		variable) is $z_{3,1}$. So we obtain that 
		\begin{eqnarray*}
			\mathcal{I}_{1,1,2} &=&-\int_{\mathbb{R}^{3\times 3}}f_{\sigma
				^{-1}(1)}(\sum_{(l,k)\in \left\{ 1\right\} \times \left\{ 1,\ldots,\sigma
				^{-1}(1)\right\} }z_{l,k})f_{\sigma ^{-1}(2)}(\sum_{\substack{ (l,k)\in
					\left\{ 1,2\right\} \times \left\{ 1,\ldots,\sigma ^{-1}(2)\right\}  \\ %
					(l,k)\neq (1,\sigma ^{-1}(1))}}z_{l,k})  \\
			&&\times f_{\sigma ^{-1}(3)}(\sum_{\substack{ (l,k)\in \left\{ 1,2,3\right\} \times
					\left\{ 1,\ldots,\sigma ^{-1}(3)\right\}  \\ (l,k)\neq (1,\sigma
					^{-1}(1)),(2,\sigma ^{-1}(2))}}z_{l,k})B_{1,3}(z_{3,1})(-E_{1,2}(z_{2,1})B_{3,2}(z_{2,3}-z_{2,1}))  \\
			&&\times B_{2,1}(z_{1,2})E_{3,1}(z_{1,3}-z_{1,2}) \prod\limits_{_{\substack{ l,k=1 \\ (l,k)\neq
						(1,2),(1,3),(2,1),(2,3),(3,1)}}}^{3}E_{k,l}(z_{l,k})\diffns z_{1,1}\ldots \diffns z_{3,3}.
		\end{eqnarray*}%
		So the ``algorithm", which leads to the (final) term $\mathcal{I}_{1,1,2}$,
		can be visualized as 
		\begin{equation*}
			\frame{$%
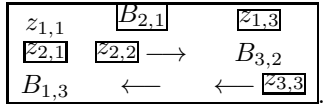
				\begin{array}{ccc}
					z_{1,1} & \frame{$B_{2,1}$} & \frame{$z_{1,3}$} \\ 
					\frame{$z_{2,1}$} & \frame{$z_{2,2}$}\longrightarrow  & B_{3,2} \\ 
					B_{1,3} & \longleftarrow  & \longleftarrow \frame{$z_{3,3}$}%
				\end{array}%
				$}.
		\end{equation*}%
		So $B_{2,1},B_{3,2}$ and $B_{1,3}$ in the diagram correspond to the factors
		in the integral $\mathcal{I}_{1,1,2}$ and we notice that these factors are
		not overlapping in the sense that they correspond to (open) rectangles which
		are disjoint.
		
		So we see that%
		\begin{equation*}
			\mathbb{E}\Big[ \prod\limits_{i=1}^{3}f_{i}^{\prime
			}(W_{s_{i},t_{\sigma (i)}})\Big] 
		\end{equation*}%
		can be written as a sum (of at most $2^{n}$) summands, which are given by
		integrals of products of factors, where exactly three non-overlapping
		factors appear.  
		
	So far, we have been discussing the case $n=3$ in \eqref{Estimate} for our algorithm, which selects certain rectangles in the plane. This specific case already encapsulates all the fundamental ideas needed to extend the algorithm to the general case for any $n$. Motivated by this example, the general case can be described as follows:
		
		\textbf{Step 1:} We start in the first row with an orientation point $\mathcal{O}_{1}$. There are two possibilities:
		\begin{enumerate}
			\item	If $%
			\mathcal{O}_{1}=z_{1,\sigma ^{-1}(1)}$, then we choose as integration by
			parts variable which leads to a factor $B_{\sigma ^{-1}(1),1}$ ``at this
			position". 
			\item If $\mathcal{O}_{1}=(z_{1,\sigma ^{-1}(1)},z_{1,\sigma
				^{-1}(1)+1})$, then $z_{1,\sigma ^{-1}(1)}$ is the integration by parts
			variable and $z_{1,\sigma ^{-1}(1)+1}$ the substitution variable for $%
			z_{1,\sigma ^{-1}(1)}$. Then $z_{1,\sigma ^{-1}(1)}$ gives rise to the
			factor $B_{\sigma ^{-1}(1),1}$ and $z_{1,\sigma ^{-1}(1)+1}$ to the factor $%
			B_{\sigma ^{-1}(1)+1,1}$. We can then move from $B_{\sigma ^{-1}(1),1}$ or $%
			B_{\sigma ^{-1}(1)+1,1}$ in the first row to the next orientation point $%
			\mathcal{O}_{2}$ in the second row. 
		\end{enumerate}
		
		\textbf{Step 2:} In this step we have 5 possibilities:
		\begin{enumerate}
			\item If we start in $B_{\sigma ^{-1}(1),1}$
			and if $\mathcal{O}_{2}=z_{2,\sigma ^{-1}(2)}$, then $z_{2,\sigma ^{-1}(2)}$
			is not in the column of $B_{\sigma ^{-1}(1),1}$ and we can select $%
			z_{2,\sigma ^{-1}(2)}$ as an integration by parts variable. 
			\item If we begin the
			path with $B_{\sigma ^{-1}(1)+1,1}$ and if $\mathcal{O}_{2}=z_{2,\sigma
				^{-1}(2)}$, then we distinguish two cases: 
			\begin{itemize}
				\item \textsl{Case 1:} $z_{2,\sigma
					^{-1}(2)}$ is not in the column of $B_{\sigma ^{-1}(1)+1,1}$. In this case,
				we choose $z_{2,\sigma ^{-1}(2)}$ as an integration by parts variable (which
				corresponds to the next factor $B_{\sigma ^{-1}(2),2}$).
				
				\item \textsl{Case 2:}	 $%
				z_{2,\sigma ^{-1}(2)}$ belongs to the column of $B_{\sigma ^{-1}(1)+1,1}$.
				Then, we select the closest variable $z_{2,\sigma ^{-1}(2)-k}$ to $%
				z_{2,\sigma ^{-1}(2)}$ on the left, which does not belong to the previous
				column (here $k=1$). In this way, we obtain the factor $B_{\sigma
					^{-1}(2)-k,2}$. 			
			\end{itemize} 
			\item Let us now start in $B_{\sigma ^{-1}(1),1}$ in connection
			with $\mathcal{O}_{2}=(z_{2,\sigma ^{-1}(2)},z_{2,\sigma ^{-1}(2)+1})$.
			Then, $z_{2,\sigma ^{-1}(2)}$ is not in the column of $B_{\sigma ^{-1}(1),1}$
			and we can choose this variable as an integration by parts variable, which
			gives rise to $B_{\sigma ^{-1}(2),2}$.
			\begin{itemize}
				\item \textsl{Case 1: } If $z_{2,\sigma ^{-1}(2)+1}$ is not
				in the column of $B_{\sigma ^{-1}(1),1}$, then we select $z_{2,\sigma
					^{-1}(2)+1}$ as a substitution variable, which gives rise to the factor $%
				B_{\sigma ^{-1}(2)+1,2}$ (in the path $(B_{\sigma ^{-1}(1),1},B_{\sigma
					^{-1}(2)+1,2},\ldots)$).
				\item \textsl{Case 2: } However, if $z_{2,\sigma ^{-1}(2)+1}$ is in the column
				of $B_{\sigma ^{-1}(1),1}$, then we take the closest variable $z_{2,\sigma
					^{-1}(2)+1+k}$ to $z_{2,\sigma ^{-1}(2)+1}$ on the right, which is not in
				the column of $B_{\sigma ^{-1}(1),1}$, as the new substitution variable ($%
				k\geq 1$). The latter leads to the factor $B_{\sigma ^{-1}(2)+1+k,2}$. In
				this case, we get a path $(B_{\sigma ^{-1}(1),1},B_{\sigma
					^{-1}(2)+1+k,2},\ldots)$.
			\end{itemize}
			\item Then we have to look at the case $B_{\sigma
				^{-1}(1)+1,1}$ in connection with $\mathcal{O}_{2}=\mathcal{(}z_{2,\sigma
				^{-1}(2)},z_{2,\sigma ^{-1}(2)+1})$. Here, 
			\begin{itemize}
				\item \textsl{Case 1: }if $z_{2,\sigma ^{-1}(2)}$ is not
				in the same column as $B_{\sigma ^{-1}(1)+1,1}$, then we choose this
				variable as integration by parts variable, which leads to the path $%
				(B_{\sigma ^{-1}(1)+1,1},B_{\sigma ^{-1}(2),2},\ldots)$.
				\item \textsl{Case 2: }On the other hand, if $%
				z_{2,\sigma ^{-1}(2)}$ belongs to the same column as $B_{\sigma ^{-1}(1)+1,1}
				$, then we make a shift to the left and choose the variable $z_{2,\sigma
					^{-1}(2)-k}$ closest to $z_{2,\sigma ^{-1}(2)}$, which does not belong to
				the column of $B_{\sigma ^{-1}(1)+1,1}$. This leads to path of the form $%
				(B_{\sigma ^{-1}(1)+1,1},B_{\sigma ^{-1}(2)-k,2},\ldots)$.
			\end{itemize}
			\item The remaining case
			with respect to the second row leads either to the path $(B_{\sigma
				^{-1}(1)+1,1},B_{\sigma ^{-1}(2)+1,2},\ldots)$ or $(B_{\sigma
				^{-1}(1)+1,1},B_{\sigma ^{-1}(2)+1+k,2},\ldots)$.
		\end{enumerate}
		
		Then we continue in this way
		until the $(i-1)-$th step. 
		
		\textbf{Step i:} Here we consider two possibilities:
		\begin{enumerate}
			\item	If we depart here from the specific path of
			non-overlapping factors $(B_{l_{1},1},\ldots,B_{l_{i-1},i-1})$ in the tree, then
			we can move to the next orientation point $\mathcal{O}_{i}$ in the $i$-th
			row. Let us say $\mathcal{O}_{i}=\mathcal{(}z_{i,\sigma
				^{-1}(i)},z_{2,\sigma ^{-1}(i)+1})$. Then, we choose an integration by parts variable as follows. 
			\begin{itemize}
				\item \textsl{Case 1:}	if $z_{i,\sigma ^{-1}(i)}$ is not
				in one of the columns of $(B_{l_{1},1},\ldots,B_{l_{i-1},i-1})$, then this gives
				the path $(B_{l_{1},1},\ldots,B_{l_{i-1},i-1},B_{\sigma ^{-1}(i),i})$ in the $i-$%
				th step.
				\item \textsl{Case 2: }Otherwise, we get the path of non-overlapping factors $%
				(B_{l_{1},1},\ldots,B_{l_{i-1},i-1},B_{\sigma ^{-1}(i),i},B_{\sigma
					^{-1}(i)-k,i})$, where $z_{i,\sigma ^{-1}(i)-k}$ (as the integration by
				parts variable) is the closest variable to $z_{i,\sigma ^{-1}(i)-k}$, which
				is not contained in one of the columns of $(B_{l_{1},1},\ldots,B_{l_{i-1},i-1})$%
				. See Lemma \ref{Shift} in the Appendix \ref{Appen1}, which shows that such a ``left-shift" is always possible.
			\end{itemize}
			\item	Further we select a substitution variable as described below.
			\begin{itemize}
				\item \textsl{Case 1: } If $z_{i,\sigma ^{-1}(i)+1}$ is not in one of the columns of $%
				(B_{l_{1},1},\ldots,B_{l_{i-1},i-1})$, then we select $z_{i,\sigma ^{-1}(i)+1}$
				as the substitution variable. The latter leads to the path (of
				non-overlapping factors) $(B_{l_{1},1},\ldots,B_{l_{i-1},i-1},B_{\sigma
					^{-1}(i)+1,i})$.
				\item \textsl{Case 2: }On \ the other hand, if $z_{i,\sigma ^{-1}(i)+1}$ belongs to
				one of the columns of $(B_{l_{1},1},\ldots,B_{l_{i-1},i-1})$, then we take a new
				substitution variable $z_{i,\sigma ^{-1}(i)+1+k}$, which is closest to $%
				z_{i,\sigma ^{-1}(i)+1}$ and which is not in one of the columns of $%
				(B_{l_{1},1},\ldots,B_{l_{i-1},i-1})$. Note that such a substitution variable $%
				z_{i,\sigma ^{-1}(i)+1+k}$ indeed exists (see Lemma \ref%
				{SubstitutionVariable} and Lemma \ref{Shift} in the Appendix \ref{Appen1}). Hence, we get in this case the path
				(of non-overlapping factors) $(B_{l_{1},1},\ldots,B_{l_{i-1},i-1},B_{\sigma
					^{-1}(i)+1+k,i})$.
			\end{itemize}
		\end{enumerate}
		
		Using this procedure until the $n-$th step, we find (as in the case $n=3)$
		that%
		\begin{equation*}
			\mathbb{E}\Big[ \prod\limits_{i=1}^{n}f_{i}^{\prime
			}(W_{s_{i},t_{\sigma (i)}})\Big] 
		\end{equation*}%
		can be written as a sum of at most $2^{n}$ summands, which are integrals of
		products of factors, where exactly $n$ non-overlapping factors occur. 		
		Then, by applying the estimate 
		\begin{equation*}
			\int_{\mathbb{R}}\left\vert B_{i,j}(z)\right\vert \diffns z\leq
			C_{0}((s_{i}-s_{i-1})(t_{j}-t_{j-1})^{-1/2}\text{ }
		\end{equation*}%
		for an absolute constant $C_{0}$, one can show (as we will see) the estimate \eqref{Estimate}.
		
		For $i=3$, we give two examples of application of the above algorithm in Appendix \ref{Appen2} using a binomial tree.
		
		\section{Proof of the main result}\label{mallcalgaus}

		In this section, we give the proof of the main result. More precisely, we consider the following SDE 
		\begin{align}\label{eqmain1}
			\mathrm{X}_{s,t}=b(s,t,X_{s,t})\mathrm{d} {s}\mathrm{d}t+ \mathrm{d}W_{s,t}, \,(s,t)\in \mathcal{T}^2, X_{s,0}=x=X_{0,t} \in \mathbb{R}^d,
		\end{align}
		in which the drift coefficient $b:\mathcal{T}^2\times \mathbb{R}^d\rightarrow \mathbb{R}^d$ is Borel measurable and bounded. 
		\begin{lemma} \label{lemmainres1}	
			Let $b: \mathcal{T}^2 \times \mathbb{R}^d \rightarrow \mathbb{R}^d$ be a smooth function with compact support. Then the corresponding strong solution $X$ of \eqref{eqmain1} satisfies 
			\begin{align}\label{Eqlemmainres1}
				\E \Big[ \Big| D_{r,u} X_{s,t} - D_{\bar r,\bar u} X_{s,t} \Big|^2 \Big] \leq C_d ( \|b\|_\infty  )  (|r -\bar r|+|u-\bar u|)^{\alpha}
			\end{align}
			for $0 \leq (\bar r,\bar u) \preceq (r,u) \preceq(s,t)\preceq (T,T)$, $\alpha = \alpha(s,t) > 0$  and
			\begin{align}\label{Eqlemmainres2}
				\sup_{0 \leq (r,u) \leq (T,T)}\ E \left[ | D_{r,u} X_{s,t} |^2 \right] \leq C_d (\|b\|_\infty),
			\end{align}
			where $C_{d,} : [0, \infty) \rightarrow [0, \infty)$ is an increasing, continuous function, $| \cdot |$ a matrix-norm on $\mathbb{R}^{d \times d}$.	
		\end{lemma}
		
		\begin{proof}
			Using the the chain-rule for the Malliavin derivatives (see \cite[Page 144]{Nua06}), it follows that
			\begin{align} \label{EqMalDer}
				D_{r,u }X_{s,t} = &  \mathcal{I}_d + \int_{r}^{s} \int_{u}^{t}b'(s_1,t_1,X_{s_1,t_1})D_{r,u}X_{s_1,t_1} \diffns s_1\diffns t_1\notag\\
				=&\mathcal{I}_d + \sum_{n=1}^{\infty} \int_{
					\begin{subarray}{c}
						r < s_n < \ldots < s_1 < s\\
						u<t_n<\ldots<t_1<t
					\end{subarray}
				} b'(s_1,t_1, X_{s_1,t_1})  \dots  b'(s_n,t_n, X_{s_n,t_n}) \diffns s_1 \diffns t_1\dots \diffns s_n \diffns t_n
			\end{align}
			$\mathbb{P}$-a.e. for all $(r,u) \preceq (s,t) $. Here $\mathcal{I}_d$ is the $d \times d$ identity matrix and $b'(s,t,x) = \left( \frac{ \partial}{ \partial x_i} b^{(j)}(s,t,x) \right)_{1 \leq i,j \leq d} $ is the spatial Jacobian derivative of $b$.
			Note that the above representation for the Malliavin derivative can be
			obtained by using Picard iteration.\ Alternatively, one can show that this
			representation solves the equation (\ref{EqMalDer}).\ Then uniqueness can be
			verified by applying a multiparameter version of Gr\"{o}nwall's Lemma (see
			e.g. \cite[Lemma 5.1.1]{Qi16}).    
			
			Further, we have that
			\begin{align*}
				&D_{r,u}X_{s,t}-D_{\bar r,\bar u}X_{s,t}\\=&\int_r^s\int_u^tb'(s_1,t_1,X_{s_1,t_1})(D_{r,u}X_{s_1,t_1}-D_{\bar r,\bar u}X_{s_1,t_1})\mathrm{d}s_1\mathrm{d}t_1\\&+\int_{\bar r}^r\int_{\bar u}^tb'(s_1,t_1,X_{s_1,t_1})D_{\bar r,\bar u}X_{s_1,t_1}\mathrm{d}s_1\mathrm{d}t_1+\int_{\bar r}^s\int_{\bar u}^ub'(s_1,t_1,X_{s_1,t_1})D_{\bar r,\bar u}X_{s_1,t_1}\mathrm{d}s_1\mathrm{d}t_1\\&-\int_{\bar r}^{r}\int_{\bar u}^ub'(s_1,t_1,X_{s_1,t_1})D_{\bar r,\bar u}X_{s_1,t_1}\mathrm{d}s_1\mathrm{d}t_1\\=&\int_r^s\int_u^tb'(s_1,t_1,X_{s_1,t_1})(D_{r,u}X_{s_1,t_1}-D_{\bar r,\bar u}X_{s_1,t_1})\mathrm{d}s_1\mathrm{d}t_1+(D_{\bar r,\bar u}X_{r,t}-\mathcal{I}_d)\\&+(D_{\bar r,\bar u}X_{s,u}-\mathcal{I}_d)-(D_{\bar r,\bar u}X_{r,u}-\mathcal{I}_d).
			\end{align*}
			Thus, by Picard iteration, we obtain 
			\begin{align*}
				&D_{r,u}X_{s,t}-D_{\bar r,\bar u}X_{s,t}\\=&(D_{\bar r,\bar u}X_{r,t}-\mathcal{I}_d)+(D_{\bar r,\bar u}X_{s,u}-\mathcal{I}_d) \\&+\sum\limits_{k=1}^{\infty}\,\int\limits_{
					\begin{subarray}{c}
						r<s_k<\ldots<s_1<s\\
						u<t_k<\ldots<t_1<t
					\end{subarray}
				}b'(s_1,t_1,X_{s_1,t_1})\ldots:b'(s_k,t_k,X_{s_k,t_k})(D_{\bar r,\bar u}X_{r,t_k}-\mathcal{I}_d)\mathrm{d}s_1\mathrm{d}t_1\ldots\mathrm{d}s_k\mathrm{d}t_k\\
				&+\sum\limits_{k=1}^{\infty}\,\int\limits_{
					\begin{subarray}{c}
						r<s_k<\ldots<s_1<s\\
						u<t_k<\ldots<t_1<t
					\end{subarray}
				}b'(s_1,t_1,X_{s_1,t_1})\ldots b'(s_k,t_k,X_{s_k,t_k})(D_{\bar r,\bar u}X_{s_k,u}-\mathcal{I}_d)\mathrm{d}s_1\mathrm{d}t_1\ldots\mathrm{d}s_k\mathrm{d}t_k\\&-(D_{\bar r,\bar u}X_{r,u}-\mathcal{I}_d)\Big(I_d+\sum\limits_{k=1}^{\infty}\,\int\limits_{
					\begin{subarray}{c}
						r<s_k<\ldots<s_1<s\\
						u<t_k<\ldots<t_1<t
					\end{subarray}
				}b'(s_1,t_1,X_{s_1,t_1})\ldots b'(s_k,t_k,X_{s_k,t_k})\mathrm{d}s_1\mathrm{d}t_1\ldots\mathrm{d}s_k\mathrm{d}t_k\Big).
			\end{align*}
			
			For $0\leq \bar r<r<s$, define
			\begin{align*}
				\nabla^{(k)}_{r,s}=&\{(s_1,\ldots,s_k)\in{\mathcal{T}}^{k}:\,r<s_k<\ldots<s_1<s\},
			\end{align*}
			\begin{align*}
				\Xi^{(k+n)}_{\bar r,r,s}=&\Big\{(s_1,\ldots,s_{k+n})\in \mathcal{T}^{k+n}:\,
				\bar r<s_{k+n}<\ldots<s_{k+1}<s_{k},\,r<s_{k}<\dots<s_{1}<s\Big\}
			\end{align*}
			and
			\begin{align*}
				\Delta^{(k+n)}_{\bar r,r,s}=&\Big\{(s_1,t_1,\ldots,s_{k+n},t_{k+n})\in\mathcal{T}^{k+n}:\,\bar r<s_{k+n}<\ldots<s_{k+1}<r<s_{k}<\ldots<s_{1}<s\text{ } \Big\}.
			\end{align*}
			
			Hence
			
			\begin{align*}
				&D_{r,u}X_{s,t}-D_{\bar r,\bar u}X_{s,t}\\=&\sum\limits_{k=1}^{\infty}\,\int_{
					\nabla^{(k)}_{\bar r,r}
				}\int_{
					\nabla^{(k)}_{\bar u,t}
				}\prod\limits_{j=1}^kb'(s_j,t_j,X_{s_j,t_j}) \,\mathrm{d}s_1\mathrm{d}t_1\ldots\mathrm{d}s_k\mathrm{d}t_k\\&+\sum\limits_{k=1}^{\infty}\,\int_{
					\nabla^{(k)}_{\bar r,s}
				}\int_{
					\nabla^{(k)}_{\bar u,u}
				}\prod\limits_{j=1}^kb'(s_j,t_j,X_{s_j,t_j}) \,\mathrm{d}s_1\mathrm{d}t_1\ldots\mathrm{d}s_k\mathrm{d}t_k \\&+\sum\limits_{k=1}^{\infty}\sum\limits_{n=1}^{\infty}\,\int_{\Xi^{(k+n)}_{\bar r,r,s}}\int_{
					\Delta^{(k+n)}_{\bar u,u,t}
				}\prod\limits_{j=1}^{k+n}b'(s_j,t_j,X_{s_j,t_j})\,  \mathrm{d}s_{1}\mathrm{d}t_{1}\ldots\mathrm{d}s_{k+n}\mathrm{d}t_{k+n}\\
				&+\sum\limits_{k=1}^{\infty}\sum\limits_{n=1}^{\infty}\,\int_{\Delta^{(k+n)}_{\bar r,r,s}}\int_{
					\Xi^{(k+n)}_{\bar u,u,t}
				}\prod\limits_{j=1}^{k+n}b'(s_j,t_j,X_{s_j,t_j})\,  \mathrm{d}s_{1}\mathrm{d}t_{1}\ldots\mathrm{d}s_{k+n}\mathrm{d}t_{k+n}\\
				&-\Big(\sum\limits_{k=1}^{\infty}\,\int_{
					\nabla^{(k)}_{\bar r,r}
				}\int_{
					\nabla^{(k)}_{\bar u,u}
				}\prod\limits_{j=1}^kb'(s_j,t_j,X_{s_j,t_j}) \,\mathrm{d}s_1\mathrm{d}t_1\ldots\mathrm{d}s_k\mathrm{d}t_k\Big)\\&\quad\times\Big(I_d+\sum\limits_{k=1}^{\infty}\,\int_{
					\nabla^{(k)}_{r,s}
				}\int_{
					\nabla^{(k)}_{u,t}
				}\prod\limits_{j=1}^kb'(s_j,t_j,X_{s_j,t_j}) \,\mathrm{d}s_1\mathrm{d}t_1\ldots\mathrm{d}s_k\mathrm{d}t_k\Big).
			\end{align*}
			Squaring both sides of the equality above gives
			\begin{align*}
				&|D_{r,u}X_{s,t}-D_{\bar r,\bar u}X_{s,t}|^2\\ \leq&5\Big\{\Big(\sum\limits_{k=1}^{\infty}\,\int_{
					\nabla^{(k)}_{\bar r,r}
				}\int_{
					\nabla^{(k)}_{\bar u,t}
				}\prod\limits_{j=1}^kb'(s_j,t_j,X_{s_j,t_j}) \,\mathrm{d}s_1\mathrm{d}t_1\ldots\mathrm{d}s_k\mathrm{d}t_k\Big)^2\\&+\Big(\sum\limits_{k=1}^{\infty}\,\int_{
					\nabla^{(k)}_{\bar r,s}
				}\int_{
					\nabla^{(k)}_{\bar u,u}
				}\prod\limits_{j=1}^kb'(s_j,t_j,X_{s_j,t_j}) \,\mathrm{d}s_1\mathrm{d}t_1\ldots\mathrm{d}s_k\mathrm{d}t_k\Big)^2 \\&+\Big(\sum\limits_{k=1}^{\infty}\sum\limits_{n=1}^{\infty}\,\int_{\Xi^{(k+n)}_{\bar r,r,s}}\int_{
					\Delta^{(k+n)}_{\bar u,u,t}
				}\prod\limits_{j=1}^{k+n}b'(s_j,t_j,X_{s_j,t_j})\,  \mathrm{d}s_{1}\mathrm{d}t_{1}\ldots\mathrm{d}s_{k+n}\mathrm{d}t_{k+n}\Big)^2\\
				&+\Big(\sum\limits_{k=1}^{\infty}\sum\limits_{n=1}^{\infty}\,\int_{\Delta^{(k+n)}_{\bar r,r,s}}\int_{
					\Xi^{(k+n)}_{\bar u,u,t}
				}\prod\limits_{j=1}^{k+n}b'(s_j,t_j,X_{s_j,t_j})\,  \mathrm{d}s_{1}\mathrm{d}t_{1}\ldots\mathrm{d}s_{k+n}\mathrm{d}t_{k+n}\Big)^2\\
				&+\Big[\Big(\sum\limits_{k=1}^{\infty}\,\int_{
					\nabla^{(k)}_{\bar r,r}
				}\int_{
					\nabla^{(k)}_{\bar u,u}
				}\prod\limits_{j=1}^kb'(s_j,t_j,X_{s_j,t_j}) \,\mathrm{d}s_1\mathrm{d}t_1\ldots\mathrm{d}s_k\mathrm{d}t_k\Big)\\&\quad\times\Big(I_d+\sum\limits_{k=1}^{\infty}\,\int_{
					\nabla^{(k)}_{r,s}
				}\int_{
					\nabla^{(k)}_{u,t}
				}\prod\limits_{j=1}^kb'(s_j,t_j,X_{s_j,t_j}) \,\mathrm{d}s_1\mathrm{d}t_1\ldots\mathrm{d}s_k\mathrm{d}t_k\Big)\Big]^2\Big\}.
			\end{align*}
			Taking the expectation on both sides and using the Cameron-Martin-Girsanov theorem we have
			\begin{align*}
				&\E[|D_{r,u}X_{s,t}-D_{\bar r,\bar u}X_{s,t}|^2]\\ \leq&5\E\Big[\mathcal{E}\Big(\int_{\mathcal{T}} b(s_1,t_1,W_{s_1,t_1})\cdot\mathrm{d}W_{s_1,t_1}\Big)\\&\times\Big\{\Big(\sum\limits_{k=1}^{\infty}\,\int_{
					\nabla^{(k)}_{\bar r,r}
				}\int_{
					\nabla^{(k)}_{\bar u,t}
				}\prod\limits_{j=1}^kb'(s_j,t_j,W_{s_j,t_j}) \,\mathrm{d}s_1\mathrm{d}t_1\ldots\mathrm{d}s_k\mathrm{d}t_k\Big)^2\\&\quad+\Big(\sum\limits_{k=1}^{\infty}\,\int_{
					\nabla^{(k)}_{\bar r,s}
				}\int_{
					\nabla^{(k)}_{\bar u,u}
				}\prod\limits_{j=1}^kb'(s_j,t_j,W_{s_j,t_j}) \,\mathrm{d}s_1\mathrm{d}t_1\ldots\mathrm{d}s_k\mathrm{d}t_k\Big)^2 \\&+\Big(\sum\limits_{k=1}^{\infty}\sum\limits_{n=1}^{\infty}\,\int_{\Xi^{(k+n)}_{\bar r,r,s}}\int_{
					\Delta^{(k+n)}_{\bar u,u,t}
				}\prod\limits_{j=1}^{k+n}b'(s_j,t_j,W_{s_j,t_j})\,  \mathrm{d}s_{1}\mathrm{d}t_{1}\ldots\mathrm{d}s_{k+n}\mathrm{d}t_{k+n}\Big)^2\\
				&+\Big(\sum\limits_{k=1}^{\infty}\sum\limits_{n=1}^{\infty}\,\int_{\Delta^{(k+n)}_{\bar r,r,s}}\int_{
					\Xi^{(k+n)}_{\bar u,u,t}
				}\prod\limits_{j=1}^{k+n}b'(s_j,t_j,W_{s_j,t_j})\,  \mathrm{d}s_{1}\mathrm{d}t_{1}\ldots\mathrm{d}s_{k+n}\mathrm{d}t_{k+n}\Big)^2\\
				&+\Big[\Big(\sum\limits_{k=1}^{\infty}\,\int_{
					\nabla^{(k)}_{\bar r,r}
				}\int_{
					\nabla^{(k)}_{\bar u,u}
				}\prod\limits_{j=1}^kb'(s_j,t_j,W_{s_j,t_j}) \,\mathrm{d}s_1\mathrm{d}t_1\ldots\mathrm{d}s_k\mathrm{d}t_k\Big)\\&\quad\times\Big(I_d+\sum\limits_{k=1}^{\infty}\,\int_{
					\nabla^{(k)}_{r,s}
				}\int_{
					\nabla^{(k)}_{u,t}
				}\prod\limits_{j=1}^kb'(s_j,t_j,W_{s_j,t_j}) \,\mathrm{d}s_1\mathrm{d}t_1\ldots\mathrm{d}s_k\mathrm{d}t_k\Big)\Big]^2\Big\}\Big].
			\end{align*}
			It follows from Cauchy-Schwarz Inequality that
			\begin{align*}
				&\E[|D_{r,u}X_{s,t}-D_{\bar r,\bar u}X_{s,t}|^2]\\ \leq&5\E\Big[\mathcal{E}\Big(\int_{\mathcal{T}} b(s_1,t_1,W_{s_1,t_1})\cdot\mathrm{d}W_{s_1,t_1}\Big)^2\Big]^{1/2}\\&\times\Big\Vert\Big(\sum\limits_{k=1}^{\infty}\,\int_{
					\nabla^{(k)}_{\bar r,r}
				}\int_{
					\nabla^{(k)}_{\bar u,t}
				}\prod\limits_{j=1}^kb'(s_j,t_j,W_{s_j,t_j}) \,\mathrm{d}s_1\mathrm{d}t_1\ldots\mathrm{d}s_k\mathrm{d}t_k\Big)^2\\&\quad+\Big(\sum\limits_{k=1}^{\infty}\,\int_{
					\nabla^{(k)}_{\bar r,s}
				}\int_{
					\nabla^{(k)}_{\bar u,u}
				}\prod\limits_{j=1}^kb'(s_j,t_j,W_{s_j,t_j}) \,\mathrm{d}s_1\mathrm{d}t_1\ldots\mathrm{d}s_k\mathrm{d}t_k\Big)^2 \\&+\Big(\sum\limits_{k=1}^{\infty}\sum\limits_{n=1}^{\infty}\,\int_{\Xi^{(k+n)}_{\bar r,r,s}}\int_{
					\Delta^{(k+n)}_{\bar u,u,t}
				}\prod\limits_{j=1}^{k+n}b'(s_j,t_j,W_{s_j,t_j})\,  \mathrm{d}s_{1}\mathrm{d}t_{1}\ldots\mathrm{d}s_{k+n}\mathrm{d}t_{k+n}\Big)^2\\
				&+\Big(\sum\limits_{k=1}^{\infty}\sum\limits_{n=1}^{\infty}\,\int_{\Delta^{(k+n)}_{\bar r,r,s}}\int_{
					\Xi^{(k+n)}_{\bar u,u,t}
				}\prod\limits_{j=1}^{k+n}b'(s_j,t_j,W_{s_j,t_j})\,  \mathrm{d}s_{1}\mathrm{d}t_{1}\ldots\mathrm{d}s_{k+n}\mathrm{d}t_{k+n}\Big)^2\\
				&+\Big[\Big(\sum\limits_{k=1}^{\infty}\,\int_{
					\nabla^{(k)}_{\bar r,r}
				}\int_{
					\nabla^{(k)}_{\bar u,u}
				}\prod\limits_{j=1}^kb'(s_j,t_j,W_{s_j,t_j}) \,\mathrm{d}s_1\mathrm{d}t_1\ldots\mathrm{d}s_k\mathrm{d}t_k\Big)\\&\quad\times\Big(I_d+\sum\limits_{k=1}^{\infty}\,\int_{
					\nabla^{(k)}_{r,s}
				}\int_{
					\nabla^{(k)}_{u,t}
				}\prod\limits_{j=1}^kb'(s_j,t_j,W_{s_j,t_j}) \,\mathrm{d}s_1\mathrm{d}t_1\ldots\mathrm{d}s_k\mathrm{d}t_k\Big)\Big]^2\Big\Vert_{L^2(\Omega)}.
			\end{align*}
			We deduce from triangular inequality, H\"older inequality and the monotone convergence theorem that
			\begin{align*}
				&\E[|D_{r,u}X_{s,t}-D_{\bar r,\bar u}X_{s,t}|^2]\\ \leq&
				C\Big\{
				\Big(\sum\limits_{k=1}^{\infty}\,\Big\Vert\int_{
					\nabla^{(k)}_{\bar r,r}
				}\int_{
					\nabla^{(k)}_{\bar u,t}
				}\prod\limits_{j=1}^kb'(s_j,t_j,W_{s_j,t_j}) \,\mathrm{d}s_1\mathrm{d}t_1\ldots\mathrm{d}s_k\mathrm{d}t_k\Big\Vert_{L^4(\Omega)}\Big)^2\\&+\Big(\sum\limits_{k=1}^{\infty}\,\Big\Vert\int_{
					\nabla^{(k)}_{\bar r,s}
				}\int_{
					\nabla^{(k)}_{\bar u,u}
				}\prod\limits_{j=1}^kb'(s_j,t_j,W_{s_j,t_j}) \,\mathrm{d}s_1\mathrm{d}t_1\ldots\mathrm{d}s_k\mathrm{d}t_k\Big\Vert_{L^4(\Omega)}\Big)^2 \\&+\Big(\sum\limits_{k,n=1}^{\infty} \Big\Vert\int_{\Xi^{(k+n)}_{\bar r,r,s}}\int_{
					\Delta^{(k+n)}_{\bar u,u,t}
				}\prod\limits_{j=1}^{k+n}b'(s_j,t_j,W_{s_j,t_j})\,  \mathrm{d}s_{1}\mathrm{d}t_{1}\ldots\mathrm{d}s_{k+n}\mathrm{d}t_{k+n}\Big\Vert_{L^4(\Omega)}\Big)^2\\
				&+\Big(\sum\limits_{k,n=1}^{\infty} \Big\Vert\int_{\Delta^{(k+n)}_{\bar r,r,s}}\int_{
					\Xi^{(k+n)}_{\bar u,u,t}
				}\prod\limits_{j=1}^{k+n}b'(s_j,t_j,W_{s_j,t_j})\,  \mathrm{d}s_{1}\mathrm{d}t_{1}\ldots\mathrm{d}s_{k+n}\mathrm{d}t_{k+n}\Big\Vert_{L^4(\Omega)}\Big)^2\\
				&+\Big(\sum\limits_{k=1}^{\infty}\Big\Vert\int_{
					\nabla^{(k)}_{\bar r,r}
				}\int_{
					\nabla^{(k)}_{\bar u,u}
				}\prod\limits_{j=1}^kb'(s_j,t_j,W_{s_j,t_j}) \,\mathrm{d}s_1\mathrm{d}t_1\ldots\mathrm{d}s_k\mathrm{d}t_k\Big\Vert_{L^8(\Omega)}\Big)^2\\&\quad\times\Big(1+\sum\limits_{k=1}^{\infty}\Big\Vert\int_{
					\nabla^{(k)}_{r,s}
				}\int_{
					\nabla^{(k)}_{u,t}
				}\prod\limits_{j=1}^kb'(s_j,t_j,W_{s_j,t_j}) \,\mathrm{d}s_1\mathrm{d}t_1\ldots\mathrm{d}s_k\mathrm{d}t_k\Big\Vert_{L^8(\Omega)}\Big)^2
				\Big\}\\
				=&C\{I_1+I_2+I_3+I_4+I_5\times I_6\},
			\end{align*}
			where
			\begin{align*}
				C=5\E\Big[\mathcal{E}\Big(\int_{\mathcal{T}}b_{s_1,t_1}(W_{s_1,t_1})\cdot\mathrm{d}W_{s_1,t_1}\Big)^2\Big]^{1/2}.
			\end{align*}
			Now observe that
			\begin{align*}
				&\Big(\int_{
					\nabla^{(k)}_{\bar r,r}
				}\int_{
					\nabla^{(k)}_{\bar u,t}
				} \prod\limits_{v=1}^kb'(s_v,t_v,W_{s_v,t_v})\mathrm{d}s_1\mathrm{d}t_1\ldots\mathrm{d}s_{k}\mathrm{d}t_{k}\Big)_{i,j}\\=&
				\sum\limits_{l_1,\ldots,l_{k-1}=1}^d\int_{
					\nabla^{(k)}_{\bar r,r}
				}\int_{
					\nabla^{(k)}_{\bar u,t}
				}\frac{\partial b^{(l_1)}_{s_1,t_1}}{\partial x_{i}}(W_{s_1,t_1})\frac{\partial b^{(l_2)}_{s_2,t_2}}{\partial x_{l_1}}(W_{s_2,t_2})\ldots\frac{\partial b^{(j)}_{s_{k},t_{k}}}{\partial x_{l_{k-1}}}(W_{s_{k},t_{k}})\,\mathrm{d}s_1\mathrm{d}t_1\ldots\mathrm{d}s_{k}\mathrm{d}t_{k}.
			\end{align*}
			As a consequence, 
			\begin{align*}
				I_1\leq&\Big(\sum\limits_{k=1}^{\infty}\sum\limits_{i,j=1}^d\sum\limits_{l_1,\ldots,l_{k-1}=1}^d\Big\|\int_{
					\nabla^{(k)}_{\bar r,r}
				}\int_{
					\nabla^{(k)}_{\bar u,t}
				}\frac{\partial b^{(l_1)}_{s_1,t_1}}{\partial x_{i}}(W_{s_1,t_1})\frac{\partial b^{(l_2)}_{s_2,t_2}}{\partial x_{l_1}}(W_{s_2,t_2})\times\ldots\\&\qquad\qquad\qquad\qquad\qquad\ldots\times\frac{\partial b^{(j)}_{s_{k},t_{k}}}{\partial x_{l_{k-1}}}(W_{s_{k},t_{k}})\,\mathrm{d}s_1\mathrm{d}t_1\ldots\mathrm{d}s_{k}\mathrm{d}t_{k}\Big\|_{L^4(\Omega)}\Big)^2.
			\end{align*}
			Similarly, we have
			\begin{align*}
				I_2\leq&\Big(\sum\limits_{k=1}^{\infty}\sum\limits_{i,j=1}^d\sum\limits_{l_1,\ldots,l_{k-1}=1}^d\Big\|\int_{
					\nabla^{(k)}_{\bar r,s}
				}\int_{
					\nabla^{(k)}_{\bar u,u}
				}\frac{\partial b^{(l_1)}_{s_1,t_1}}{\partial x_{i}}(W_{s_1,t_1})\frac{\partial b^{(l_2)}_{s_2,t_2}}{\partial x_{l_1}}(W_{s_2,t_2})\times\ldots\\&\qquad\qquad\qquad\qquad\qquad\qquad\ldots\times\frac{\partial b^{(j)}_{s_{k},t_{k}}}{\partial x_{l_{k-1}}}(W_{s_{k},t_{k}})\,\mathrm{d}s_1\mathrm{d}t_1\ldots\mathrm{d}s_{k}\mathrm{d}t_{k}\Big\|_{L^4(\Omega)}\Big)^2,
			\end{align*}
			\begin{align*}
				I_3\leq\Big(\sum\limits_{k,n=1}^{\infty}\sum\limits_{i,j=1}^d\sum\limits_{l_1,\ldots,l_{k+n-1}=1}^d&\Big\Vert\int_{\Delta^{(k+n)}_{\bar r,r,s}}\int_{
					\Xi^{(k+n)}_{\bar u,u,t}
				}\frac{\partial b^{(l_1)}_{s_1,t_1}}{\partial x_{i}}(W_{s_1,t_1})\frac{\partial b^{(l_2)}_{s_2,t_2}}{\partial x_{l_1}}(W_{s_2,t_2})\ldots \\&\ldots\frac{\partial b^{(j)}_{s_{k+n},t_{k+n}}}{\partial x_{l_{k+n-1}}}(W_{s_{k+n},t_{k+n}})\,\mathrm{d}s_1\mathrm{d}t_1\ldots\mathrm{d}s_{k+n}\mathrm{d}t_{k+n}\Big\Vert_{L^4(\Omega)}\Big)^2,
			\end{align*}
			\begin{align*}
				I_4\leq\Big(\sum\limits_{k,n=1}^{\infty}\sum\limits_{i,j=1}^d\sum\limits_{l_1,\ldots,l_{k+n-1}=1}^d&\Big\Vert\int_{\Xi^{(k+n)}_{\bar r,r,s}}\int_{
					\Delta^{(k+n)}_{\bar u,u,t}
				}\frac{\partial b^{(l_1)}_{s_1,t_1}}{\partial x_{i}}(W_{s_1,t_1})\frac{\partial b^{(l_2)}_{s_2,t_2}}{\partial x_{l_1}}(W_{s_2,t_2})\ldots \\&\ldots\frac{\partial b^{(j)}_{s_{k+n},t_{k+n}}}{\partial x_{l_{k+n-1}}}(W_{s_{k+n},t_{k+n}})\,\mathrm{d}s_1\mathrm{d}t_1\ldots\mathrm{d}s_{k+n}\mathrm{d}t_{k+n}\Big\Vert_{L^4(\Omega)}\Big)^2,
			\end{align*}
			\begin{align*}
				I_5\leq&\Big(\sum\limits_{k=1}^{\infty}\sum\limits_{i,j=1}^d\sum\limits_{l_1,\ldots,l_{k-1}=1}^d\Big\|\int_{
					\nabla^{(k)}_{\bar r,r}
				}\int_{
					\nabla^{(k)}_{\bar u,u}
				}\frac{\partial b^{(l_1)}_{s_1,t_1}}{\partial x_{i}}(W_{s_1,t_1})\frac{\partial b^{(l_2)}_{s_2,t_2}}{\partial x_{l_1}}(W_{s_2,t_2})\ldots\frac{\partial b^{(j)}_{s_{k},t_{k}}}{\partial x_{l_{k-1}}}(W_{s_{k},t_{k}})\\&\qquad\qquad\qquad\qquad\qquad\qquad\times\mathrm{d}s_1\mathrm{d}t_1\ldots\mathrm{d}s_{k}\mathrm{d}t_{k}\Big\|_{L^8(\Omega)}\Big)^2
			\end{align*}
			and
			\begin{align*}
				I_6\leq&\Big(1+\sum\limits_{k=1}^{\infty}\sum\limits_{i,j=1}^d\sum\limits_{l_1,\ldots,l_{k-1}=1}^d\Big\|\int_{
					\nabla^{(k)}_{r,s}
				}\int_{
					\nabla^{(k)}_{u,t}
				}\frac{\partial b^{(l_1)}_{s_1,t_1}}{\partial x_{i}}(W_{s_1,t_1})\frac{\partial b^{(l_2)}_{s_2,t_2}}{\partial x_{l_1}}(W_{s_2,t_2})\ldots\\&\qquad\qquad\qquad\qquad\qquad\qquad\ldots\frac{\partial b^{(j)}_{s_{k},t_{k}}}{\partial x_{l_{k-1}}}(W_{s_{k},t_{k}})\,\mathrm{d}s_1\mathrm{d}t_1\ldots\mathrm{d}s_{k}\mathrm{d}t_{k}\Big\|_{L^8(\Omega)}\Big)^2.
			\end{align*}
			Moreover, as $\{\nabla^{(4k,\sigma)}_{r,s}\}_{\sigma\in\widehat{\mathcal{P}}_{4k}}$ is a partition of $(\nabla^{(k)}_{r,s})^4$ for any $0\leq r<s$ (see Lemma \ref{lemma:Shuffle}), one has
			\begin{align*}
				&\Big(\int_{
					\nabla^{(k)}_{\bar r,r}
				}\int_{
					\nabla^{(k)}_{\bar u,t}
				}\frac{\partial b^{(l_1)}_{s_1,t_1}}{\partial x_{i}}(W_{s_1,t_1})\frac{\partial b^{(l_2)}_{s_2,t_2}}{\partial x_{l_1}}(W_{s_2,t_2})\ldots\frac{\partial b^{(j)}_{s_{k},t_{k}}}{\partial x_{l_{k-1}}}(W_{s_{k},t_{k}})\,\mathrm{d}s_1\mathrm{d}t_1\ldots\mathrm{d}s_{k}\mathrm{d}t_{k}\Big)^4\\=&\sum\limits_{\sigma,\gamma}\int_{
					\nabla^{(4k,\sigma)}_{\bar r,r}
				}\int_{
					\nabla^{(4k,\gamma)}_{\bar u,t}
				}g_{1}(s_1,t_1)\ldots g_{4k}(s_{4k},t_{4k})\,\mathrm{d}s_1\mathrm{d}t_1\ldots\mathrm{d}s_{4k}\mathrm{d}t_{4k},
			\end{align*}
			where $g_{\ell}\in\Big\{\frac{\partial b^{(j)}_{\cdot,\cdot}}{\partial x_{i}}(W_{\cdot,\cdot}),\,1\leq i,j\leq d\Big\}$ for all $\ell\in\{1,\ldots,4k\}$. Then using \eqref{eq:MMNPZmD} (in Corollary \ref{corol:DavieVarSheet}),
			\begin{align*}
				&\E\Big[\Big(\int_{
					\nabla^{(k)}_{\bar r,r}
				}\int_{
					\nabla^{(k)}_{\bar u,t}
				}\frac{\partial b^{(l_1)}_{s_1,t_1}}{\partial x_{i}}(W_{s_1,t_1})\frac{\partial b^{(l_2)}_{s_2,t_2}}{\partial x_{l_1}}(W_{s_2,t_2})\ldots\frac{\partial b^{(j)}_{s_{k},t_{k}}}{\partial x_{l_{k-1}}}(W_{s_{k},t_{k}})\,\mathrm{d}s_1\mathrm{d}t_1\ldots\mathrm{d}s_{k}\mathrm{d}t_{k}\Big)^4\Big]\\=&\E\Big[\int_{
					(	\nabla^{(k)}_{\bar r,r})^4
				}\int_{
					(\nabla^{(k)}_{\bar u,t})^4
				}g_{1}(s_1,t_1)\ldots g_{4k}(s_{4k},t_{4k})\,\mathrm{d}s_1\mathrm{d}t_1\ldots\mathrm{d}s_{4k}\mathrm{d}t_{4k}\Big] 
				\\ =&\sum\limits_{\sigma,\gamma}\E\Big[\int_{
					\nabla^{(4k,\sigma)}_{\bar r,r}
				}\int_{
					\nabla^{(4k,\gamma)}_{\bar u,t}
				}g_{1}(s_1,t_1)\ldots g_{4k}(s_{4k},t_{4k})\,\mathrm{d}s_1\mathrm{d}t_1\ldots\mathrm{d}s_{4k}\mathrm{d}t_{4k}\Big]
				\\
				\leq&\sum\limits_{\sigma,\gamma}\int_{
					\nabla^{(4k,\sigma)}_{\bar r,r}
				}\int_{
					\nabla^{(4k,\gamma)}_{\bar u,t}
				}\vert\E[g_{1}(s_1,t_1)\ldots g_{4k}(s_{4k},t_{4k})]\vert\,\mathrm{d}s_1\mathrm{d}t_1\ldots\mathrm{d}s_{4k}\mathrm{d}t_{4k} \\ \leq&\frac{4^{9k}C_1^{4k}\|b\|^{4k}_{\infty}(r-\bar r)^{2k}(t-\bar u)^{2k}}{\Gamma\left(\frac{4k+1}{2}\right)^2},
			\end{align*}
			since the number of elements in $\widehat{\mathcal{P}}_{4k}$ does not exceed $2^{9k}$. As a consequence,
			\begin{align*}
				I_1\leq\Big(\sum\limits_{k=1}^{\infty}\frac{2^{5k}d^{2k}C_1^{k}\Vert b\Vert^{k}_{\infty}(r-\bar r)^{(k-1)/2}(t-\bar u)^{k/2}}{\Gamma\left(\frac{4k+1}{2}\right)^{1/2}}\Big)^2|r-\bar r|=C^{(1)}_d(\|b\|_{\infty})|r-\bar r|.
			\end{align*}
			Similarly, we have
			\begin{align*}
				I_2\leq\Big(\sum\limits_{k=1}^{\infty}\frac{2^{5k}d^{2k}C_1^{k}\Vert b\Vert^{k}_{\infty}(s-\bar r)^{k/2}(u-\bar u)^{(k-1)/2}}{\Gamma\left(\frac{4k+1}{2}\right)^{1/2}}\Big)^2|u-\bar u|= C^{(2)}_d(\Vert b\Vert_{\infty})|u-\bar u|,
			\end{align*}
			\begin{align*}
				I_5\leq&\Big(\sum\limits_{k=1}^{\infty}\frac{2^{5k}d^{2k}C_1^{k}\Vert b\Vert^{k}_{\infty}(r-\bar r)^{(k-1)/2}(u-\bar u)^{(k-1)/2}}{\Gamma\left(\frac{8k+1}{2}\right)^{1/4}}\Big)^2\vert r-\bar r\vert\times\vert u-\bar u\vert\\=& C^{(5)}_d(\Vert b\Vert_{\infty})\vert r-\bar r\vert\times\vert u-\bar u\vert
			\end{align*}
			and
			\begin{align*}
				I_6\leq&\Big(1+\sum\limits_{k=1}^{\infty}\frac{2^{5k}d^{2k}C_1^{k}\Vert b\Vert^{k}_{\infty}(s-r)^{k/2}(t-u)^{k/2}}{\Gamma\left(\frac{8k+1}{2}\right)^{1/4}}\Big)^2= C^{(6)}_d(\Vert b\Vert_{\infty})
			\end{align*}
			On the other hand $\Xi^{(k+n)}_{\bar u,u,t}
			\subset\nabla_{\bar u,t}^{(k+n)}$ gives
			\begin{align*}
				&\E\Big[\Big(\int_{\Delta^{(k+n)}_{\bar r,r,s}}\int_{
					\Xi^{(k+n)}_{\bar u,u,t}
				}\frac{\partial b^{(l_1)}_{s_1,t_1}}{\partial x_{i}}(W_{s_1,t_1})\frac{\partial b^{(l_2)}_{s_2,t_2}}{\partial x_{l_1}}(W_{s_2,t_2})\ldots  \frac{\partial b^{(j)}_{s_{k+n},t_{k+n}}}{\partial x_{l_{k+n-1}}}(W_{s_{k+n},t_{k+n}})\\&\qquad\qquad\qquad\qquad\qquad\times\mathrm{d}s_1\mathrm{d}t_1\ldots\mathrm{d}s_{k+n}\mathrm{d}t_{k+n}\Big)^4\Big]\\ =&
				\E\Big[\int_{(\Delta^{(k+n)}_{\bar r,r,s})^4}\int_{
					(\Xi^{(k+n)}_{\bar u,u,t}
					)^4}g_{1}(s_1,t_1)\ldots g_{4(k+n)}(s_{4(k+n)},t_{4(k+n)})\,\mathrm{d}s_1\mathrm{d}t_1\ldots\mathrm{d}s_{4(k+n)}\mathrm{d}t_{4(k+n)}\Big]\\ \leq&
				\int_{(\Delta^{(k+n)}_{\bar r,r,s})^4}\int_{
					(\Xi^{(k+n)}_{\bar u,u,t}
					)^4}|\E[g_{1}(s_1,t_1)\ldots g_{4(k+n)}(s_{4(k+n)},t_{4(k+n)})]|\,\mathrm{d}s_1\mathrm{d}t_1\ldots\mathrm{d}s_{4(k+n)}\mathrm{d}t_{4(k+n)}\\
				\leq&\int_{(\Delta^{(k+n)}_{\bar r,r,s})^4}\int_{
					(\nabla^{(k+n)}_{\bar u,t}
					)^4}|\E[g_{1}(s_1,t_1)\ldots g_{4(k+n)}(s_{4(k+n)},t_{4(k+n)})]|\mathrm{d}s_1\mathrm{d}t_1\ldots\mathrm{d}s_{4(k+n)}\mathrm{d}t_{4(k+n)},	
			\end{align*}
			where $g_{\ell}\in\Big\{\frac{\partial b^{(j)}_{\cdot,\cdot}}{\partial x_{i}}(W_{\cdot,\cdot}),\,1\leq i,j\leq d\Big\}$ for all $\ell\in\{1,\ldots,4(k+n)\}$. Hence, as
			$\{\Delta^{(4(k+n),\pi,\rho)}_{\bar r,r,s}\}_{(\pi,\rho)\in\widehat{\mathcal{P}}^{\ast}_{4k}\times\widehat{\mathcal{P}}^{\ast\ast}_{4n}}$ is a partition of $(\Delta^{(k+n)}_{\bar r,r,s})^4$ (see Lemma \ref{lemme:Shuffle2}) and $\{\nabla^{(4(k+n),\sigma)}_{\bar u ,t}\}_{\sigma\in\widehat{\mathcal{P}}_{4(k+n)}}$ is a partition of $(\nabla^{(k+n)}_{\bar u,t})^4$ , we deduce from \eqref{eq:MMNPZmD2} (in Corollary \ref{corol:DavieVarSheet}), that
			\begin{align*}
				&\E\Big[\Big(\int_{\Delta^{(k+n)}_{\bar r,r,s}}\int_{
					\Xi^{(k+n)}_{\bar u,u,t}
				}\frac{\partial b^{(l_1)}_{s_1,t_1}}{\partial x_{i}}(W_{s_1,t_1})\frac{\partial b^{(l_2)}_{s_2,t_2}}{\partial x_{l_1}}(W_{s_2,t_2})\ldots  \frac{\partial b^{(j)}_{s_{k+n},t_{k+n}}}{\partial x_{l_{k+n-1}}}(W_{s_{k+n},t_{k+n}})\\&\qquad\qquad\qquad\qquad\qquad\times\mathrm{d}s_1\mathrm{d}t_1\ldots\mathrm{d}s_{k+n}\mathrm{d}t_{k+n}\Big)^4\Big] \\
				\leq&\int_{(\Delta^{(k+n)}_{\bar r,r,s})^4}\int_{
					(\nabla^{(k+n)}_{\bar u,t}
					)^4}|\E[g_{1}(s_1,t_1)\ldots g_{4(k+n)}(s_{4(k+n)},t_{4(k+n)})]|\,\mathrm{d}s_1\mathrm{d}t_1\ldots\mathrm{d}s_{4(k+n)}\mathrm{d}t_{4(k+n)}\\
				=&\sum\limits_{\pi,\rho,\sigma}\int_{\Delta^{(4(k+n),\pi,\rho)}_{\bar r,r,s}}\int_{
					\nabla^{(4(k+n),\sigma)}_{\bar u,t}
				}|\E[g_{1}(s_1,t_1)\ldots g_{4(k+n)}(s_{4(k+n)},t_{4(k+n)})]|\,\mathrm{d}s_1\mathrm{d}t_1\ldots\mathrm{d}s_{4(k+n)}\mathrm{d}t_{4(k+n)}\\
				\leq&\frac{4^{9(k+n)}C_1^{4(k+n)}\|b\|^{4(k+n)}_{\infty}(r-\bar r)^{2n}(s-r)^{2k}(t-\bar u)^{2(k+n)}}{\Gamma\left(\frac{4n+1}{2}\right)\Gamma\left(\frac{4k+1}{2}\right)\Gamma\left(\frac{4(k+n)+1}{2}\right)},
			\end{align*}
			since the number of elements in $\widehat{\mathcal{P}}^{\ast}_{4k}$ (respectively $\widehat{\mathcal{P}}^{\ast}_{4n}$ and $\widehat{\mathcal{P}}_{4(k+n)}$) does not exceed $2^{9k}$ (respectively $2^{9n}$ and $2^{9(k+n)}$). Thus,
			\begin{align*}
				I_3\leq&\Big(\sum\limits_{k,n=1}^{\infty}\frac{(2d)^{5(k+n)}C_1^{k+n}\|b\|^{k+n}_{\infty}(r-\bar r)^{(n-1)/2}(s-r)^{k/2}(t-\bar u)^{(k+n)/2}}{\Gamma\Big(\frac{4n+1}{2}\Big)^{1/4}\Gamma\Big(\frac{4k+1}{2}\Big)^{1/4}\Gamma\Big(\frac{4(k+n)+1}{2}\Big)^{1/4}}\Big)^2|r-\bar r|\\=&C^{(3)}_d(\|b\|_{\infty})|r-\bar r|.
			\end{align*}
			Similarly, one may show that
			\begin{align*}
				I_4\leq&\Big(\sum\limits_{k,n=1}^{\infty}\frac{(2d)^{5(k+n)}C_1^{k+n}\|b\|^{k+n}_{\infty}(s-\bar r)^{(k+n)/2}(t-u)^{k/2}(u-\bar u)^{(n-1)/2}}{\Gamma\Big(\frac{4n+1}{2}\Big)^{1/4}\Gamma\Big(\frac{4k+1}{2}\Big)^{1/4}\Gamma\Big(\frac{4(k+n)+1}{2}\Big)^{1/4}}\Big)^2|u-\bar u|\\=&C^{(4)}_d(\Vert b\Vert_{\infty})|u-\bar u|.
			\end{align*}

			This ends the proof of \eqref{Eqlemmainres1}. The proof of \eqref{Eqlemmainres2} follows the same lines.

		\end{proof}
		
		\bigskip 
		
		In order to prove (local) Sobolev differentiability of the solution with
		respect to the initial value in Theorem \ref{Mainresult}, we also need the
		following estimate:
		
		\begin{lemma}\label{FlowEstimate}
			Assume that the vector field $b$ in the SDE \eqref{SDEWienerSheet} belongs
			to $C_{c}^{\infty }(\mathcal{T}^2\times \mathbb{R}^{d};\mathbb{R}%
			^{d})$. Then for all $p\geq 2$ there exists a, increasing continuous
			function $C_{d,p}:\left[ 0,\infty \right) \longrightarrow \left[ 0,\infty
			\right) $ such that%
			\begin{equation*}
				\sup_{\substack{ x\in \mathbb{R}^{d} \\ 0\leq s,t\leq T}}\E\Big[ \Big\| 
				\frac{\partial }{\partial x}X_{s,t}^{x}\Big\| ^{p}\Big] \leq
				C_{d,p}\left( \left\Vert b\right\Vert _{\infty }\right) \text{.}
			\end{equation*}
		\end{lemma}
		
		\begin{proof}
			The proof is very similar to that of Lemma \ref{lemmainres1}. More precisely, using
			the chain rule, we find that%
			\begin{equation*}
				\frac{\partial }{\partial x}X_{s,t}^{x}=\mathcal{I}_{d\times
					d}+\int_{0}^{s}\int_{0}^{t}b'(r,l,X_{r,l}^{x})\frac{\partial 
				}{\partial x}X_{r,l}^{x}\mathrm{d}l\mathrm{d}r.
			\end{equation*}%
			Then, using Picard iteration, we find the representation%
			\begin{equation*}
				\frac{\partial }{\partial x}X_{s,t}^{x}=\mathcal{I}_{d\times d}+\sum_{n\geq
					1}\int_{
					\begin{subarray}{}
						0<r_n<\cdots<r_1<s\\
						0<l_n<\cdots<l_1<t
				\end{subarray} }b'(r_{n},l_{n},X_{r_{n},l_{n}}^{x})\ldots b'(r_{1},l_{1},X_{r_{1},l_{1}}^{x})\mathrm{d}l_{n}\mathrm{d}r_{n}\cdots\mathrm{d}l_{1}\mathrm{d}r_{1}
			\end{equation*}%
			a.e. for all $s,t$. If we replace the Malliavin derivative $%
			D_{u,v}X_{s,t}^{x,n}$ by the derivative $\frac{\partial }{\partial x}%
			X_{s,t}^{x}$ in the proof of Lemma \ref{lemmainres1}, then we obtain the result for $%
			p=2$. The proof for $p>2$ rests on the repeated application of Lemmas  \ref{lemma:Shuffle} and \ref{lemme:Shuffle2}.   
		\end{proof}

		\begin{lemma}\label{Weak}
			Let $b\in L^{\infty }(\mathcal{T}^2\times \mathbb{R}^{d};\mathbb{R}%
			^{d})$. Assume a sequence of functions $\left\{ b_{n}\right\} _{n\geq
				1}\subset C_{c}^{\infty }(\mathcal{T}^2\times \mathbb{R}^{d};%
			\mathbb{R}^{d})$ such that%
			\[
			b_{n}(t,x)\underset{n\longrightarrow \infty }{\longrightarrow }b(t,x)\text{ }%
			(t,x)\text{-a.e.}
			\]%
			and%
			\[
			\sup_{n\geq 1}\left\Vert b_{n}\right\Vert _{\infty }<\infty \text{.}
			\]%
			Let $X_{\cdot }^{x,n}$ be the strong solution of the SDE associated with the
			vector field $b_{n}$ for $n\geq 1$. Suppose that $X_{\cdot }^{x}$ is a weak
			solution of the SDE on the same probability space. Then for all bounded
			continuous functions $\varphi :\mathbb{R}^{d}\longrightarrow \mathbb{R}$ we
			have that%
			\[
			\varphi (X_{t_{1},t_{2}}^{x,n})\underset{n\longrightarrow \infty }{%
				\longrightarrow }\E\left[ \varphi (X_{t_{1},t_{2}}^{x})\right. \left\vert 
			\mathcal{F}_{t_{1},t_{2}}\right] 
			\]%
			weakly in $L^{2}(\Omega ;\mathcal{F}_{t_{1},t_{2}})$.
		\end{lemma}
		
		\begin{proof}
			Recall that the Doleans-Dade exponential for square integrable adapted
			processes $Y_{\cdot }$ on the plane is defined as 
			\begin{align*}
				&\mathcal{E}(\int_{\left[ 0,T\right] \times \left[ 0,u\right] }\left\langle
				Y_{s_{1},s_{2}},\mathrm{d}W_{s_{1},s_{2}}\right\rangle ) \\
				=&\exp (\int_{\left[ 0,T\right] \times \left[ 0,u\right] }\left\langle
				Y_{s_{1},s_{2}},\mathrm{d}W_{s_{1},s_{2}}\right\rangle -\frac{1}{2}%
				\int_{0}^{1}\int_{0}^{u}\left\Vert Y_{s_{1},s_{2}}\right\Vert
				^{2}\mathrm{d}s_{1}\mathrm{d}s_{2}).
			\end{align*}%
			See Appendix. Since 
			\begin{equation*}
				\sup_{n\geq 1}\E\Big[ \Big\|\int_{\mathcal{T}^2 }b_{n}(s_{1},s_{2},x+W_{s_{1},s_{2}})\mathrm{d}W_{s_{1},s_{2}}\Big\| ^{p}%
				\Big] <\infty 
			\end{equation*}%
			for all $p\geq 1$, it follows from the It\^{o} isometry with respect to the
			multiparameter Wiener processes and by assumption that%
			\begin{equation*}
				\int_{\left[ 0,T\right] \times \left[ 0,u\right]
				}b_{n}(s_{1},s_{2},x+W_{s_{1},s_{2}})\mathrm{d}W_{s_{1},s_{2}}\underset{%
					n\longrightarrow \infty }{\longrightarrow }\int_{\left[ 0,T\right] \times %
					\left[ 0,u\right] }b(s_{1},s_{2},x+W_{s_{1},s_{2}})\mathrm{d}W_{s_{1},s_{2}}
			\end{equation*}%
			in $L^{p}(\Omega )$ for all $p\geq 1$.
			
			We mention that the set%
			\begin{align*}
				\Pi _{t_{1},t_{2}} :=&\Big\{ \exp \Big(\sum_{j=1}^{d}\int_{\left[ 0,u_{1}\right] \times \left[
					0,u_{2}\right] }f_{i}(s_{1},s_{2})\mathrm{d}W_{s_{1},s_{2}}^{j}\Big):0\leq u_{i}\leq
				t_{i},i=1,2,\\&\qquad\qquad\qquad\qquad\qquad\qquad f_{j}\in L^{\infty }(\left[ 0,t_{1}\right] \times \left[ 0,t_{2}%
				\right] ),j=1,\ldots,d\Big\} 
			\end{align*}%
			is a total subspace of $L^{p}(\Omega ;\mathcal{F}_{t_{1},t_{2}})$. So it is
			sufficient to show that%
			\begin{equation*}
				\E\left[ \varphi (X_{t_{1},t_{2}}^{x,n})\xi \right] \underset{%
					n\longrightarrow \infty }{\longrightarrow }\E\left[ \E\left[ \varphi
				(X_{t_{1},t_{2}}^{x})\right. \left\vert \mathcal{F}_{t_{1},t_{2}}\right] \xi %
				\right] 
			\end{equation*}%
			for all $\xi \in \Pi _{t_{1},t_{2}}$. Define the Girsanov change of measures 
			\begin{equation*}
				\mathrm{d}\Q_{n}=\mathcal{E}(\int_{\mathcal{T}^2
				}\left\langle
				b_{n}(s_{1},s_{2},X_{s_{1},s_{2}}^{x,n}),\mathrm{d}W_{s_{1},s_{2}}\right\rangle
				)\mathrm{d}\Pb,\,\,\,n\geq 1
			\end{equation*}%
			See Appendix. Then Girsanov`s theorem for multiparameter Wiener processes
			implies that%
			\begin{equation*}
				W_{\cdot }^{\ast ,n}:=W_{\cdot }-\int_{0}^{\cdot }\int_{0}^{\cdot
				}b_{n}(s_{1},s_{2},X_{s_{1},s_{2}}^{x,n})\mathrm{d}s_{1}\mathrm{d}s_{2}
			\end{equation*}%
			is a $\mathbb{Q}_{n}$-Wiener process on the plane and that 
			\begin{align*}
				&\E\Big[ \varphi (X_{t_{1},t_{2}}^{x,n})\exp \Big\{\sum_{j=1}^{d}\int_{\left[
					0,u_{1}\right] \times \left[ 0,u_{2}\right]
				}f_{j}(s_{1},s_{2})\mathrm{d}W_{s_{1},s_{2}}^{j}\Big\}\Big]  \\
				=&\E_{\Q_{n}}\Big[ \varphi (x+W_{t_{1},t_{2}}^{\ast ,n})\exp
				\Big\{\sum_{j=1}^{d}\int_{\left[ 0,u_{1}\right] \times \left[ 0,u_{2}\right]
				}f_{j}(s_{1},s_{2})\mathrm{d}W_{s_{1},s_{2}}^{\ast
					,n,j}\\&+\sum_{j=1}^{d}\int_{0}^{u_{1}}%
				\int_{0}^{u_{2}}f_{j}(s_{1},s_{2})b_{n}^{j}(s_{1},s_{2},x+W_{s_{1},s_{2}}^{%
					\ast ,n})\mathrm{d}s_{1}\mathrm{d}s_{2}\Big\}  \mathcal{E}\Big(\int_{\mathcal{T}^2
				}\left\langle b_{n}(s_{1},s_{2},x+W_{s_{1},s_{2}}^{\ast
					,n}),\mathrm{d}W_{s_{1},s_{2}}^{\ast ,n}\right\rangle \Big)\Big]  \\
				=&\E\Big[ \varphi (x+W_{t_{1},t_{2}})\exp \Big\{\sum_{j=1}^{d}\int_{\left[
					0,u_{1}\right] \times \left[ 0,u_{2}\right]
				}f_{j}(s_{1},s_{2})\mathrm{d}W_{s_{1},s_{2}}^{,j}\\&+\sum_{j=1}^{d}\int_{0}^{u_{1}}%
				\int_{0}^{u_{2}}f_{j}(s_{1},s_{2})b_{n}^{j}(s_{1},s_{2},x+W_{s_{1},s_{2}})\mathrm{d}s_{1}\mathrm{d}s_{2}\Big\} 
				\mathcal{E}\Big(\int_{\mathcal{T}^2
				}\left\langle b_{n}(s_{1},s_{2},x+W_{s_{1},s_{2}}^{
				}),\mathrm{d}W_{s_{1},s_{2}}^{}\right\rangle \Big)\Big] .
			\end{align*}%
			We have that%
			\begin{align*}
				&\E\Big[ \E\Big[ \varphi (X_{t_{1},t_{2}}^{x})\Big\vert \mathcal{F}%
				_{t_{1},t_{2}}\Big] \exp (\sum_{j=1}^{d}\int_{\left[ 0,u_{1}\right] \times %
					\left[ 0,u_{2}\right] }f_{j}(s_{1},s_{2})\mathrm{d}W_{s_{1},s_{2}}^{j})\Big]  \\
				=&E\Big[ \E\Big[ \varphi (X_{t_{1},t_{2}}^{x})\exp (\sum_{j=1}^{d}\int_{%
					\left[ 0,u_{1}\right] \times \left[ 0,u_{2}\right]
				}f_{j}(s_{1},s_{2})\mathrm{d}W_{s_{1},s_{2}}^{j})\Big\vert \mathcal{F}%
				_{t_{1},t_{2}}\Big] \Big]  \\
				=&\E\Big[ \varphi (X_{t_{1},t_{2}}^{x})\exp (\sum_{j=1}^{d}\int_{\left[
					0,u_{1}\right] \times \left[ 0,u_{2}\right]
				}f_{j}(s_{1},s_{2})\mathrm{d}W_{s_{1},s_{2}}^{j})\Big] .
			\end{align*}%
			Since $\left\vert e^{x}-e^{y}\right\vert \leq e^{x+y}\left\vert
			x-y\right\vert $ for all $x,y$, our assumptions on $\varphi ,f,b_{n},n\geq 1$
			and H\"{o}lder's inequality entail that%
			\begin{align*}
				&\Big|\E\Big[ \varphi (X_{t_{1},t_{2}}^{x,n})\exp
				(\sum_{j=1}^{d}\int_{\left[ 0,u_{1}\right] \times \left[ 0,u_{2}\right]
				}f_{j}(s_{1},s_{2})\mathrm{d}W_{s_{1},s_{2}}^{j})\Big] \\&\qquad\qquad-\E\Big[ \varphi
				(X_{t_{1},t_{2}}^{x})\exp \Big(\sum_{j=1}^{d}\int_{\left[ 0,u_{1}\right] \times %
					\left[ 0,u_{2}\right] }f_{j}(s_{1},s_{2})\mathrm{d}W_{s_{1},s_{2}}^{j}\Big)\Big]
				\Big|  \\
				\leq &\Big| \E\Big[ \varphi (x+W_{t_{1},t_{2}})\exp
				\Big(\sum_{j=1}^{d}\int_{\left[ 0,u_{1}\right] \times \left[ 0,u_{2}\right]
				}f_{j}(s_{1},s_{2})\mathrm{d}W_{s_{1},s_{2}}^{,j}\\&\quad+\sum_{j=1}^{d}\int_{0}^{u_{1}}%
				\int_{0}^{u_{2}}f_{j}(s_{1},s_{2})b_{n}^{j}(s_{1},s_{2},x+W_{s_{1},s_{2}})\mathrm{d}s_{1}\mathrm{d}s_{2}\Big)
				\mathcal{E}(\int_{\mathcal{T}^2
				}\left\langle
				b_{n}(s_{1},s_{2},x+W_{s_{1},s_{2}}),\mathrm{d}W_{s_{1},s_{2}}\right\rangle )\Big] 
				\\
				&-\E\Big[ \varphi (x+W_{t_{1},t_{2}})\exp \Big(\sum_{j=1}^{d}\int_{\left[
					0,u_{1}\right] \times \left[ 0,u_{2}\right]
				}f_{j}(s_{1},s_{2})\mathrm{d}W_{s_{1},s_{2}}^{,j}\\&\quad+\sum_{j=1}^{d}\int_{0}^{u_{1}}%
				\int_{0}^{u_{2}}f_{j}(s_{1},s_{2})b^{j}(s_{1},s_{2},x+W_{s_{1},s_{2}})\mathrm{d}s_{1}\mathrm{d}s_{2}\Big)
				\mathcal{E}(\int_{\mathcal{T}^2 }\left\langle
				b(s_{1},s_{2},x+W_{s_{1},s_{2}}),\mathrm{d}W_{s_{1},s_{2}}\right\rangle )\Big]
				\Big|  \\
				\leq &KI_{1}^{n}I_{2}^{n},
			\end{align*}%
			where 
			\begin{align*}
				K:=\E\Big[ \varphi (x+W_{t_{1},t_{2}})^2\exp \Big(2\sum_{j=1}^{d}\int_{\left[
					0,u_{1}\right] \times \left[ 0,u_{2}\right]
				}f_{j}(s_{1},s_{2})\mathrm{d}W_{s_{1},s_{2}}^{,j}\Big)\Big]^{1/2}
			\end{align*}
			
			\begin{align*}
				I_{1}^{n} :=&\E\Big[ \exp \Big(4\sum_{j=1}^{d}\int_{\left[ 0,u_{1}\right] \times \left[
					0,u_{2}\right]
				}f_{j}(s_{1},s_{2})(b_{n}^{j}(s_{1},s_{2},x+W_{s_{1},s_{2}})+b^{j}(s_{1},s_{2},x+W_{s_{1},s_{2}}))\mathrm{d}W_{s_{1},s_{2}}^{j}\Big)%
				\\&\times\exp\Big(4\int_{\mathcal{T}^2 }\left\langle
				b_n(s_{1},s_{2},x+W_{s_{1},s_{2}})+b(s_{1},s_{2},x+W_{s_{1},s_{2}}),\mathrm{d}W_{s_{1},s_{2}}\right\rangle\Big)\\
				&\times\exp\Big(-2\int_{\mathcal{T}^2}\left\{\left\Vert
				b_n(s_{1},s_{2},x+W_{s_{1},s_{2}})\right\Vert ^{2}+\left\Vert
				b(s_{1},s_{2},x+W_{s_{1},s_{2}})\right\Vert ^{2}\right\}\mathrm{d}s_{1}\mathrm{d}s_{2}\Big)\Big] ^{1/4},
			\end{align*}%
			\begin{align*}
				I_{2}^{n} :=&\E\Big[ \Big|
				\sum_{j=1}^{d}\int_{0}^{u_{1}}%
				\int_{0}^{u_{2}}f_{j}(s_{1},s_{2})b^{j}(s_{1},s_{2},x+W_{s_{1},s_{2}})\mathrm{d}s_{1}\mathrm{d}s_{2}
				\\
				&+\sum_{j=1}^{d}\int_{\mathcal{T}^2}
				b^{j}(s_{1},s_{2},x+W_{s_{1},s_{2}})\mathrm{d}W_{s_{1},s_{2}}^{j}-\frac{1}{2}%
				\int_{\mathcal{T}^2}\left\Vert
				b(s_{1},s_{2},x+W_{s_{1},s_{2}})\right\Vert ^{2}\mathrm{d}s_{1}\mathrm{d}s_{2} \\
				&-\sum_{j=1}^{d}\int_{0}^{u_{1}}%
				\int_{0}^{u_{2}}f_{j}(s_{1},s_{2})b_{n}^{j}(s_{1},s_{2},x+W_{s_{1},s_{2}})\mathrm{d}s_{1}\mathrm{d}s_{2}-\sum_{j=1}^{d}\int_{
					\mathcal{T}^2
				}b_{n}^{j}(s_{1},s_{2},x+W_{s_{1},s_{2}})\mathrm{d}W_{s_{1},s_{2}}^{j} \\
				& +\frac{1}{2}\int_{\mathcal{T}^2}\left\Vert
				b_{n}(s_{1},s_{2},x+W_{s_{1},s_{2}})\right\Vert ^{2}\mathrm{d}s_{1}\mathrm{d}s_{2}\Big|^{4}%
				\Big] ^{1/4}
			\end{align*}%
			and where $C<\infty $ is a constant. It follows from the martingale property
			of the Doleans-Dade exponential and the bounded assumption of the functions
			(see Appendix) that $\ \sup_{n\geq 1}I_{1}^{n}<\infty $. Further, using also
			dominated convergence, we have that $I_{2}^{n}\underset{n\longrightarrow
				\infty }{\longrightarrow }0$, which gives the result.

		\end{proof}

		Finally, we give the proof of our main result:
		
		\begin{proof}[Proof of Theorem \protect\ref{Mainresult}] We use the previous results to prove \textbf{Step 2} and \textbf{Step 3}.

			\textbf{Step 2:} Let $b_{n}:\mathcal{T}^2 \times 
			\mathbb{R}^{d}\longrightarrow \mathbb{R}^{d},n\geq 1$ a sequence of
			compactly supported smooth vector fields, which approximates $b\in L^{\infty
			}(\mathcal{T}^2\times \mathbb{R}^{d};\mathbb{R}^{d})$ in the SDE (%
			\ref{SDEWienerSheet}) in the sense of Lemma \ref{Weak}. Denote by $X_{\cdot,\cdot
			}^{x,n},n\geq 1$ the unique global strong solutions to the SDE \eqref{SDEWienerSheet} associated with the vector fields $b_{n},n\geq 1$. Let $%
			(s,t)\in \mathcal{T}^2$. Then, it follows from Lemma \ref{lemmainres1} that
			there exists an increasing continuous function $C_{d}:\left[ 0,\infty
			\right) \longrightarrow \left[ 0,\infty \right) $ such that for all $n$ 
			\begin{equation*}
				\sup_{\substack{ 0\leq r\leq s \\ 0\leq u\leq t}}\E\Big[ \Big\|
				D_{r,u}X_{s,t}^{x,n}\Big\| ^{2}\Big] \leq C_{d}\Big( \sup_{n\geq
					1}\left\Vert b_{n}\right\Vert _{\infty }\Big) 
			\end{equation*}%
			as well as%
			\begin{equation*}
				\E\left[ \left\Vert D_{r,u}X_{s,t}^{x,n}-D_{\bar{r},\bar{u}%
				}X_{s,t}^{x,n}\right\Vert ^{2}\right] \leq C_{d}\Big( \sup_{n\geq
					1}\left\Vert b_{n}\right\Vert _{\infty }\Big) (\left\vert r-\bar{r}%
				\right\vert +\left\vert u-\bar{u}\right\vert )^{\alpha }
			\end{equation*}%
			for all $0\leq r,\bar{r}\leq s,0\leq u,\bar{u}\leq t$ and for some $\alpha >0
			$. The latter in connection with the compactness criterion of Theorem \ref%
			{CompactnessCriterion} entails that there exists a subsequence $%
			n_{k}=n_{k}(s,t),k\geq 1$ and some $Y_{s,t}^{x}\in L^{2}(\Omega ;\mathbb{R}%
			^{d})$ such that%
			\begin{equation*}
				X_{s,t}^{x,n_{k}}\underset{k\longrightarrow \infty }{\longrightarrow }%
				Y_{s,t}^{x}\text{ in }L^{2}(\Omega ;\mathbb{R}^{d})\text{.}
			\end{equation*}%
			\textbf{Step 3:}	Let $\varphi :\mathbb{R}^{d}\longrightarrow \mathbb{R}$ be a bounded
			continuous function. Then, by selecting an appropriate subsequence with
			respect to $n_{k},k\geq 1$,  we obtain from Lemma \ref{Weak} that%
			\begin{equation*}
				\varphi (Y_{s,t}^{x})=\E\left[ \varphi (X_{s,t}^{x})\right. \left\vert 
				\mathcal{F}_{s,t}\right] \text{ }a.e.,
			\end{equation*}%
			where $X_{\cdot }^{x}$ is a unique weak solution to the SDE (\ref%
			{SDEWienerSheet}) and $\mathcal{F}=\left\{ \mathcal{F}_{s,t}\right\} _{0\leq
				s,t\leq }$ is the (completed) filtration generated by the driving noise.
			Using an approximation argument (with respect to $\varphi $) combined with
			dominated convergence, we find that%
			\begin{equation*}
				Y_{s,t}^{x}=\E\left[ X_{s,t}^{x}\right. \left\vert \mathcal{F}_{s,t}\right] 
				\text{ }a.e.
			\end{equation*}%
			So we see (for the whole sequence) that%
			\begin{equation*}
				X_{s,t}^{x,n}\underset{n\longrightarrow \infty }{\longrightarrow }\text{ }\E%
				\left[ X_{s,t}^{x}\right. \left\vert \mathcal{F}_{s,t}\right] \text{ in }%
				L^{2}(\Omega ;\mathbb{R}^{d})\text{.}
			\end{equation*}%
			Hence, for an arbitrary bounded continuous function $\varphi :\mathbb{R}%
			^{d}\longrightarrow \mathbb{R}$ we get the following "transformation
			property":%
			\begin{equation*}
				\varphi (\E\left[ X_{s,t}^{x}\right. \left\vert \mathcal{F}_{s,t}\right] )=\E%
				\left[ \varphi (X_{s,t}^{x})\right. \left\vert \mathcal{F}_{s,t}\right] 
				\text{ }a.e.
			\end{equation*}%
			Using once again an approximation argument, we find that%
			\begin{equation*}
				\left( \E\left[ \pi _{i}(X_{s,t}^{x})\right. \left\vert \mathcal{F}_{s,t}%
				\right] \right) ^{2}=\E\left[ \left( \pi _{i}(X_{s,t}^{x})\right) ^{2}\right.
				\left\vert \mathcal{F}_{s,t}\right] ,\text{ }a.e.,i=1,\ldots,d
			\end{equation*}%
			for the projections $\pi _{i}:\mathbb{R}^{d}\longrightarrow \mathbb{R}$, $%
			i=1,\ldots,d$. So%
			\begin{equation*}
				\E\left[ X_{s,t}^{x}\right. \left\vert \mathcal{F}_{s,t}\right] =X_{s,t}^{x}%
				\text{ }a.e.
			\end{equation*}
		Indeed, using the above relation, and the tower property, repeatedly, it holds
				\begin{align}
					&	\E	[\big(\E\left[ \pi _{i}( X_{s,t}^{x})\right. \left\vert \mathcal{F}_{s,t}\right] -\pi _{i}(X_{s,t}^{x})\big)^2]\notag\\
					= &	\E	[(\E\left[ \pi _{i}(X_{s,t}^{x})\right. \left\vert \mathcal{F}_{s,t}\right])^2]+\E	[(\pi _{i}((X_{s,t}^{x}))^2]-2\E	[\E\left[\pi _{i}( X_{s,t}^{x})\right. \left\vert \mathcal{F}_{s,t}\right]\pi _{i}(X_{s,t}^{x})]\notag \\
					=&	\E	[\E\left[(\pi _{i}( X_{s,t}^{x}))^2\right. \left\vert \mathcal{F}_{s,t}\right]]+\E	[(\pi _{i}(X_{s,t}^{x}))^2]-2\E	[\E\left[ \pi _{i}(X_{s,t}^{x})\right. \left\vert \mathcal{F}_{s,t}\right]\pi _{i}(X_{s,t}^{x})] \notag \\
					=&2\E	[(\pi _{i}(X_{s,t}^{x}))^2]-2\E	\big[\E	\big[\big\{\E\left[ \pi _{i}(X_{s,t}^{x})\right. \left\vert \mathcal{F}_{s,t}\right]\pi _{i}(X_{s,t}^{x})\big\}\big\vert\mathcal{F}_{s,t}\big]\big] \notag\\
					=&2\E	[(\pi _{i}(X_{s,t}^{x}))^2]-2\E	\big[\E\left[ \pi _{i}(X_{s,t}^{x})\right. \left\vert \mathcal{F}_{s,t}\right]\E	\big[\pi _{i}(X_{s,t}^{x})\big\vert\mathcal{F}_{s,t}\big]\big]=0.\notag 
			\end{align}	

			Therefore, the weak solution $X_{\cdot }^{x}$ must be a strong one. Further, we obtain strong uniqueness of solutions by using Girsanov`s theorem in connection with the $\mathcal{S}-$ transform (or alternatively the Wiener transform) in the case of a Wiener sheet. See for example a very similar argument in \cite{MMNPZ13} in the Wiener noise case. It also follows from \cite[Lemma 1.2.3]{Nua06} that for all $(s,t)\in \left[ 0,1%
			\right] ^{2}$ $X_{s,t}^{x}$ is Malliavin differentiable.
			
			The proof of (local) Sobolev differentiability of the solution with respect
			to the initial value follows from the estimate in Lemma \ref{FlowEstimate}.
			See e.g. \cite{MMNPZ13}.       
		\end{proof}
		\begin{remark}
			In the proof of Theorem \ref{Mainresult} we constructed a unique strong solution $X_{\cdot }$
			with respect to a certain probability space $(\Omega ,\mathcal{F},\mathbb{P})
			$. However, strong solutions can be obtained on any other probability space
			by using the following argument: We know that $X_{\cdot }=F(\cdot ,W_{\cdot
			})$ a.e. for an adapted measurable functional $F$, where $W_{\cdot }$ is the
			Wiener sheet on $(\Omega ,\mathcal{F},\mathbb{P})$. Consider now another
			probability space $(\Omega ^{\ast },\mathcal{F}^{\ast },\mathbb{P}^{\ast })$
			with Wiener sheet $W_{\cdot }^{\ast }$. We then see that $X_{\cdot }^{\ast
			}:=F(\cdot ,W_{\cdot }^{\ast })$ is a (unique) strong solution with respect
			to $(\Omega ^{\ast },\mathcal{F}^{\ast },\mathbb{P}^{\ast })$.   
		\end{remark}
		
		\section{Proof of preliminary results}\label{proofprere}
		In the general case the ideas developed in Subsection \ref{algo1} can be more formalized, as the
		proof of Proposition \ref{prop:DavieVarSheet} (based on a slightly different algorithm)
		shows:
		
		\begin{proof}[Proof of Proposition \ref{prop:DavieVarSheet}]
			We suppose without loss of generality that $s^{}_{1}<\ldots<s^{}_{n}$,    $t^{}_{1}<\ldots<t^{}_{n}$ and we show that for any permutation $\sigma$ on $\{1,\ldots,n\}$,
			\begin{align}\label{EqDavieVarSheetdD2}
				\Big|\E\Big[\prod\limits_{i=1}^n\frac{\partial b^{(l_{i-1})}}{\partial x_{l_i}}(s^{}_i,t^{}_{\sigma(i)},W_{s^{}_i,t^{}_{\sigma(i)}})\Big]\Big|\leq C_0^n\Vert b\Vert_{\infty}^n\prod\limits_{i=1}^n(s_{i}-s_{i-1})^{-1/2}(t_{i}-t_{i-1})^{-1/2},
			\end{align}
			Indeed, if we define $\hat{b}^{(l_{i-1})}:=b^{(l_{\gamma(i)-1})}_{}$, $\hat{x}_{l_i}:=x_{l_{\gamma(i)}}$, $\hat{s}_i:=s_{\gamma(i)}$ and $\check{t}_i:=t_{\theta(i)}$,  then \eqref{EqDavieVarSheetdD} yields
			\begin{align*}
				&\Big|\E\Big[\prod\limits_{i=1}^n\frac{\partial \hat{b}^{(l_{i-1})}}{\partial \hat{x}_{l_i}}(\hat{s}^{}_i,\check{t}^{}_{\theta^{-1}\circ\gamma(i)},W_{\hat{s}^{}_i,\check{t}^{}_{\theta^{-1}\circ\gamma(i)}})\Big]\Big|\\=&\Big|\E\Big[\prod\limits_{i=1}^n\frac{\partial {b}^{(l_{\gamma(i)-1})}}{\partial {x}_{l_{\gamma(i)}}}({s}^{}_{\gamma(i)},{t}^{}_{\gamma(i)},W_{{s}^{}_{\gamma(i)},{t}^{}_{\gamma(i)}})\Big]\Big|	=\Big|\E\Big[\prod\limits_{i=1}^n\frac{\partial {b}^{(l_{i-1})}}{\partial {x}_{l_{i}}}({s}^{}_{i},{t}^{}_{i},W_{{s}^{}_{i},{t}^{}_{i}})\Big]\Big|\\
				\leq& C_0^n\Vert b\Vert^n_{\infty}\prod\limits_{i=1}^n(s_{\gamma(i)}-s_{\gamma(i-1)})^{-1/2}(t_{\theta(i)}-t_{\theta(i-1)})^{-1/2}=C_0^n\Vert b\Vert^n_{\infty}\prod\limits_{i=1}^n(\hat{s}_{i}-\hat{s}_{i-1})^{-1/2}(\check{t}_{i}-\check{t}_{i-1})^{-1/2}.
			\end{align*}
			Let  $Z_{i,j}=(Z_{i,j}^{(1)},\,\ldots,\,Z^{(d)}_{i,j})$, $(i,j)\in\{1,\cdots,n\}^2$ be the independent Gaussian random vectors given by
			\begin{align*}
				Z_{i,j}=W^{}_{s_i,t_j}-W^{}_{s_{i-1},t_j}-W^{}_{s_i,t_{j-1}}+W^{}_{s_{i-1},t_{j-1}},
			\end{align*}	 
			where $s_0=0=t_0$. For any permutation $\sigma$ on $I_n=\{1,\ldots,n\}$, we have
			\begin{align*}
				&\Big|\E\Big[\prod\limits_{i=1}^n\frac{\partial b^{(l_{i-1})}}{\partial x_{l_i}}(s^{}_i,t^{}_{\sigma(i)},W_{s^{}_i,t^{}_{\sigma(i)}})\Big]\Big|\\=&\Big|\E\Big[\prod\limits_{i=1}^n\frac{\partial \tilde{b}_i}{\partial x_{l_i}}\Big(\sum_{k=1}^i\sum_{\ell=1}^{\sigma(i)}Z_{k,\ell}\Big)\Big]\Big|= \int_{\R^{dn^2}}\prod\limits_{i=1}^n\frac{\partial \tilde{b}_i}{\partial x_{l_i}}\Big(\sum_{k=1}^i\sum_{\ell=1}^{\sigma(i)}z_{k,\ell}\Big) \prod\limits_{j=1}^{n}E_{i,j}(z_{i,j})\,\mathrm{d}z_{i,j},
			\end{align*}
			where, for every $(i,j)$, $\tilde{b}_i(z)=b^{(l_{i-1})}(s_i,t_{\sigma(i)},z)$ and
			\begin{align*}
				E_{i,j}(z)=\frac{1}{(2\pi(s_i-s)(t_j-t_{j-1}))^{d/2}}\exp\Big(-\frac{\Vert z\Vert^2}{2(s_i-s_{i-1})(t_j-t_{j-1})}\Big).
			\end{align*}
			Define also $J_{\sigma}=\{i\in I_n:\,\text{there is }k\in I_n\text{ s.t. }(i,\sigma(i))\prec(k,\sigma(k))\}$. We distinguish cases $J_{\sigma}=\emptyset$ and $J_{\sigma}\neq\emptyset$. If $J_{\sigma}=\emptyset$  (i.e. $\sigma$ is non-increasing), then each of the  variables $ z_{i,\sigma(i)}$, $i=1,\,\ldots,\,n$ appears only in one factor $\frac{\partial \tilde{b}}{\partial x_{l_i}}\Big(\sum_{k=1}^i\sum_{\ell=1}^{\sigma(i)}z_{k,\ell}\Big)$ and there is no need to eliminate this variable in other factors. As a consequence, the orientation points are $\mathcal{O}_i=z_{i,\sigma(i)}$, $i=1,\,\ldots,\,n$. Hence, we apply integration by parts (with respect to $z^{(l_i)}_{i,\sigma(i)},\,i=1,\ldots,n$) to shift the derivatives onto Gaussian densities and we obtain
			\begin{align*}
				&\Big|\E\Big[\prod\limits_{i=1}^n\frac{\partial \tilde{b}_i}{\partial x_{l_i}}\Big(\sum_{k=1}^i\sum_{\ell=1}^{\sigma(i)}Z_{k,\ell}\Big)\Big]\Big|\\=&\Big|\int_{\R^{dn^2-n}}\prod\limits_{i=1}^n\Big[\int_{\R}\frac{\partial \tilde{b}_i}{\partial x_{l_i}}\Big(\sum_{k=1}^i\sum_{\ell=1}^{\sigma(i)}z_{k,\ell}\Big) E_{i,\sigma(i)}(z_{i,\sigma(i)})\,\mathrm{d}z^{(l_i)}_{i,\sigma(i)}\Big]\prod\limits_{l\neq l_i}\mathrm{d}z^{(l)}_{i,\sigma(i)}\\&\qquad\qquad\qquad\qquad\qquad\qquad\qquad\qquad\qquad\times\prod\limits_{j\in I_n\setminus\{\sigma(i)\}}^{}E_{i,j}(z_{i,j})\,\mathrm{d}z_{i,j}\Big|\\
				=&\Big|(-1)^n\int_{\R^{dn^2}}\prod\limits_{i=1}^n\tilde{b}_i^{}\Big(\sum_{k=1}^i\sum_{\ell=1}^{\sigma(i)}z_{k,\ell}\Big) B^{(l_i)}_{i,\sigma(i)}(z_{i,\sigma(i)})\,\mathrm{d}z_{i,\sigma(i)}\prod\limits_{j\in I_n\setminus\{\sigma(i)\}}^{}E_{i,j}(z_{i,j})\,\mathrm{d}z_{i,j}\Big|\\
				\leq&\prod\limits_{i=1}^{n}\Vert \tilde{b}_i\Vert_{\infty}\int_{\R^{dn^2-n}}\prod\limits_{i=1}^n\Big(\int_{\R}|B^{(l_i)}_{i,\sigma(i)}(z_{i,\sigma(i)})|\,\mathrm{d}z_{i,\sigma(i)}\Big)\prod\limits_{j\in I_n\setminus\{\sigma(i)\}}E_{i,j}(z_{i,j})\,\mathrm{d}z_{i,j}\\
				\leq&2^{n/2}\Vert b\Vert^n_{\infty}\prod\limits_{i=1}^n(s_i-s_{i-1})^{-1/2}(t_{\sigma(i)}-t_{\sigma(i)-1})^{-1/2},
			\end{align*}
			where, for any $i\in I_n$, $B^{(l_i)}_{i,\sigma(i)}(z)=\frac{\partial E_{i,\sigma(i)}}{\partial z^{(l_i)}}(z)$.	
			Suppose now that $J_{\sigma}\neq\emptyset$. We need some additional notations.
			Let   $\Lambda_{i,\sigma}=\{1,\ldots,i\}\times\{1,\ldots,\sigma(i)\}$,  $J_{\sigma}=\{i_1,\cdots,i_q\}$, $q\leq n-1$, with $i_1<\cdots<i_q$,    $J_{r,\sigma}=\{i_1,\cdots,i_r\}$, $J_{0,\sigma}=\{i_1\}=J_{1,\sigma}$, $J^{-}_{r,\sigma}=\{k\in J_{r,\sigma}:\,\sigma(k)<\sigma(i_{r+1})\}$, $J^{+}_{r,\sigma}=\{k\in J_{r,\sigma}:\,\sigma(k)>\sigma(i_{r+1})\}$, $\sigma(J^{-}_{r,\sigma}\setminus K)=\{\sigma(k):\,k\in J^{-}_{r,\sigma}\setminus K\}$ and $\sigma^{+}(K\cap J^{+}_{r,\sigma})=\{\sigma(k)+1:\,k\in K\cap J^{+}_{r,\sigma}\}$. Observe that $J^{-}_{r,\sigma}\cap J^{+}_{r,\sigma}=\emptyset$, $J^{-}_{r,\sigma}\cup J^{+}_{r,\sigma}=J_{r,\sigma}$ and $J_{q,\sigma}=J_{\sigma}$. Define $(\gamma_{i_r})_{1\leq r\leq q}$, $(\tau_{i_r})_{1\leq r\leq q}$, $(D_{r,\sigma})_{1\leq r\leq q}$ and $(D^{(\tau)}_{r,\sigma})_{1\leq r\leq q}$ as follows. $\gamma_{i_1}=\gamma_{i_1}(\emptyset)=\sigma(i_1)$, $\tau_{i_1}=\tau_{i_1}(\emptyset)=\sigma(i_1)+1$ and, for any $r\in\{1,\ldots,q-1\}$ and any $K\subset J_{r,\sigma}$, 
			\begin{align*}
				&\gamma_{i_{r+1}}=\gamma_{i_{r+1}}(K)\\:=&\max\Big\{k\in \sigma(J^{-}_{r,\sigma}\setminus K)\cup\{\sigma(i_{r+1})\}:\,k\neq\tau_{i_\ell}(K\cap J_{\ell-1,\sigma}),\,\forall\,i_{\ell}\in J^{-}_{r,\sigma}\setminus K\Big\},
			\end{align*}
			\begin{align*}
				&\tau_{i_{r+1}}=\tau_{i_{r+1}}(K)\\:=&\min\Big\{k\in \sigma^{+}(K\cap J^{+}_{r,\sigma})\cup\{\sigma(i_{r+1})+1\}:\,k\neq \gamma_{i_m}(K\cap J_{m-1,\sigma}),\,\forall\,i_m\in K\cap J^{+}_{r,\sigma}\Big\},
			\end{align*}
			\begin{align*}
				D_{r,\sigma}=\{(i_1,\gamma_{i_1}),\ldots,(i_r,\gamma_{i_r})\}\text{ and }	D^{(\tau)}_{r,\sigma}=\{(i_1,\gamma_{i_1}),\ldots,(i_r,\gamma_{i_r}),(i_1,\tau_{i_1}),\ldots,(i_r,\tau_{i_r})\}.
			\end{align*}
			The orientation points are $\mathcal{O}_i=(z_{i,\sigma(i)},z_{i,\sigma(i)+1})$ for $i\in J_{\sigma}$ and $\mathcal{O}_i=z_{i,\sigma(i)}$ for $i\in I_n\setminus J_{\sigma}$. When $i\in J_{\sigma}$, the variable $z_{i,\sigma(i)}$ appears at least in two factors $\frac{\partial \tilde{b}}{\partial x_{l_j}}\Big(\sum_{k=1}^j\sum_{\ell=1}^{\sigma(j)}z_{k,\ell}\Big)$. Hence, we need a substitution variable (that is $z_{i,\tau_i}$ for $i\in J_{\sigma}$) to eliminate (via linear transformations) the  variable $z_{i,\gamma(i)}$ in the other factors before applying integration by parts with respect to the variable $z^{(l_i)}_{i,\gamma(i)}$. 
			This allows us to show by induction that for every $r\in\{1,\ldots,q\}$,
			\begin{align}
				\label{EqDavieVarSheet1}
				\mathcal{J}:=\E\Big[\prod\limits_{i=1}^n\frac{\partial \tilde{b}_i}{\partial x_{l_i}}\Big(\sum_{k=1}^i\sum_{\ell=1}^{\sigma(i)}Z_{k,\ell}\Big)\Big]=\sum\limits_{K\subset J_{r,\sigma}}(-1)^{\# K}\mathcal{J}_{K},
			\end{align}
			where $\# K$ is the number of elements in $K$, and
			\begin{align}
				\label{EqDavieVarSheet2}
				\mathcal{J}_{K}=&\int_{\R^{dn^2-2dr}}\Big(\int_{\R^{2dr}}\prod\limits_{i\in J_{r,\sigma}}\tilde{b}_i\Big(z_{i,\gamma_i}+ \sum\limits_{(k,\ell)\in\Lambda_{i,\sigma}\setminus D^{}_{r,\sigma}}z_{k,\ell}\Big)\prod\limits_{i\in K}B^{(l_i)}_{i,\gamma_i}(z_{i,\gamma_i})E_{i,\tau_i}(z_{i,\tau_i}-z_{i,\gamma_i})\\&\quad\times\prod\limits_{i\in J_{r,\sigma}\setminus K}E_{i,\gamma_i}(z_{i,\gamma_i})B^{(l_i)}_{i,\tau_i}(z_{i,\tau_i}-z_{i,\gamma_i})\prod\limits_{i\in J_{r,\sigma}}\mathrm{d}z^{}_{i,\gamma_i}\mathrm{d}z^{}_{i,\tau_i} \Big)\nonumber\\&\quad\times 
				\prod\limits_{i\in I_n\setminus J_{r,\sigma}}\frac{\partial \tilde{b}_i}{\partial x_{l_i}}\Big(\sum\limits_{(k,\ell)\in\Lambda_{i,\sigma}\setminus D_{r,\sigma}}z_{k,\ell}\Big)\prod\limits_{(k,\ell)\in I_n^2\setminus D^{(\tau)}_{r,\sigma}}E_{k,\ell}(z_{k,\ell})\mathrm{d}z_{k,\ell}.\nonumber
			\end{align}		
			Suppose first that $r=1$. Using the linear transformation $y_{i_1,\tau_{i_1}}=z_{i_1,\tau_{i_1}}+z_{i_1,\gamma_{i_1}}$, $y_{i,j}=z_{i,j}$ for all $(i,j)\in I_n^2\setminus\{(i_1,\tau_{i_1})\}$, we obtain
			\begin{align*}
				\mathcal{J}&=\E\Big[\prod\limits_{i=1}^n\frac{\partial \tilde{b}_i}{\partial x_{l_i}}\Big(\sum_{k=1}^i\sum_{\ell=1}^{\sigma(i)}Z_{k,\ell}\Big)\Big] =\int_{\R^{dn^2}}\prod\limits_{i\in I_n}\frac{\partial \tilde{b}_i}{\partial x_{l_i}}\Big(\sum_{(k,\ell)\in\Lambda_{i,\sigma}}z_{k,\ell}\Big) \prod\limits_{(k,\ell)\in I^2_n}E_{k,\ell}(z_{k,\ell})\,\mathrm{d}z_{k,\ell} \\
				&=\int_{\R^{dn^2-2}}\Big(\int_{\R^2}\frac{\partial \tilde{b}_{i_1}}{\partial x_{l_{i_1}}}\Big(\sum_{(k,\ell)\in\Lambda_{i_1,\sigma}}y_{k,\ell}\Big) E_{i_1,\gamma_{i_1}}(y_{i_1,\gamma_{i_1}}) E_{i_1,\tau_{i_1}}(y_{i_1,\tau_{i_1}}-y_{i_1,\gamma_{i_1}})\mathrm{d}y^{(l_{i_1})}_{i_1,\gamma_{i_1}}\mathrm{d}y^{(l_{i_1})}_{i_1,\tau_{i_1}}\Big)\\
				&\quad\times\prod\limits_{l\neq l_{i_1}}\mathrm{d}y^{(l)}_{i_1,\gamma_{i_1}}\mathrm{d}y^{(l)}_{i_1,\tau_{i_1}}\prod\limits_{i\in I_n\setminus\{i_1\}} \frac{\partial \tilde{b}_i}{\partial x_{l_i}}\Big(\sum_{(k,\ell)\in\Lambda_{i,\sigma}\setminus D_{1,\sigma}}
				y_{k,\ell}\Big) \prod\limits_{(k,\ell)\in I^2_n\setminus D^{(\tau)}_{1,\sigma}}
				E_{k,\ell}(y_{k,\ell})\,\mathrm{d}y_{k,\ell}.
			\end{align*}
			It follows from integration by parts with respect to $y^{(l_{i_1})}_{i_1,\gamma_{i_1}}$ that $\mathcal{J}=\mathcal{J}_{\{i_1\}}+\mathcal{J}_{\emptyset}$, where
			\begin{align*}
				\mathcal{J}_{\{i_1\}}&=-\int_{\R^{dn^2-2d}}\Big(\int_{\R^{2d}}\tilde{b}_{{i_1}}^{}\Big(\sum_{(k,\ell)\in\Lambda_{i_1,\sigma}}y_{k,\ell}\Big) B^{(l_{i_1})}_{i_1,\gamma_{i_1}}(y_{i_1,\gamma_{i_1}}) E_{i_1,\tau_{i_1}}(y_{i_1,\tau_{i_1}}-y_{i_1,\gamma_{i_1}})\mathrm{d}y_{i_1,\gamma_{i_1}}\mathrm{d}y_{i_1,\tau_{i_1}}\Big)\\
				&\quad\times\prod\limits_{i\in I_n\setminus\{i_1\}} \frac{\partial \tilde{b}_i}{\partial x_{l_i}}\Big(\sum_{(k,\ell)\in\Lambda_{i,\sigma}\setminus D_{1,\sigma}}
				y_{k,\ell}\Big) \prod\limits_{(k,\ell)\in I^2_n\setminus D^{(\tau)}_{1,\sigma}}
				E_{k,\ell}(y_{k,\ell})\,\mathrm{d}y_{k,\ell}
			\end{align*}
			and
			\begin{align*}
				\mathcal{J}_{\emptyset}&=\int_{\R^{dn^2-2d}}\Big(\int_{\R^{2d}}\tilde{b}_{{i_1}}^{}\Big(\sum_{(k,\ell)\in\Lambda_{i_1,\sigma}}y_{k,\ell}\Big) E_{i_1,\gamma_{i_1}}(y_{i_1,\gamma_{i_1}}) B^{(l_{i_1})}_{i_1,\tau_{i_1}}(y_{i_1,\tau_{i_1}}-y_{i_1,\gamma_{i_1}})\mathrm{d}y_{i_1,\gamma_{i_1}}\mathrm{d}y_{i_1,\tau_{i_1}}\Big)\\
				&\quad\times\prod\limits_{i\in I_n\setminus\{i_1\}} \frac{\partial \tilde{b}_i}{\partial x_{l_i}}\Big(\sum_{(k,\ell)\in\Lambda_{i,\sigma}\setminus D_{1,\sigma}}
				y_{k,\ell}\Big) \prod\limits_{(k,\ell)\in I^2_n\setminus D^{(\tau)}_{1,\sigma}}
				E_{k,\ell}(y_{k,\ell})\,\mathrm{d}y_{k,\ell},
			\end{align*}
			and therefore \eqref{EqDavieVarSheet1}-\eqref{EqDavieVarSheet2} hold for $r=1$ with $J_{1,\sigma}=\{i_1\}$.
			Suppose now that \eqref{EqDavieVarSheet1}-\eqref{EqDavieVarSheet2} hold for some $r\in\{1,\ldots,q-1\}$ and fix $K\in J_{r,\sigma}$. Then, using the linear transformation $y_{i_{r+1},\tau_{i_{r+1}}}=z_{i_{r+1}, \tau_{i_{r+1}}}+z_{i_{r+1},\gamma_{i_{r+1}}}$ and $y_{i,j}=z_{i,j}$ for $(i,j)\in I_n^2\setminus\{(i_{r+1},\tau_{i_{r+1}})\}$, we have
			\begin{align*}
				&\mathcal{J}_{K}\\=&\int_{\R^{dn^2-2d(r+1)}}\Big(\int_{\R^2}\frac{\partial \tilde{b}_{{i_{r+1}}}^{}}{\partial x_{l_{i_{r+1}}}}\Big(y_{i_{r+1},\gamma_{i_{r+1}}}+\sum\limits_{(k,\ell)\in\Lambda_{i,\sigma}\setminus D_{r+1,\sigma}}y_{k,\ell}\Big)E_{i_{r+1},\gamma_{i_{r+1}}}(y_{i_{r+1},\gamma_{i_{r+1}}})\\&
				\quad\times E_{i_{r+1},\tau_{i_{r+1}}}(y_{i_{r+1},\tau_{i_{r+1}}}-y_{i_{r+1},\gamma_{i_{r+1}}})\,\mathrm{d}y^{(l_{i_{r+1}})}_{i_{r+1},\gamma_{i_{r+1}}}\mathrm{d}y^{(l_{i_{r+1}})}_{i_{r+1},\tau_{i_{r+1}}}\Big)\prod\limits_{\quad l\neq l_{i_{r+1}}}\mathrm{d}y^{(l)}_{i_{r+1},\gamma_{i_{r+1}}}\mathrm{d}y^{(l)}_{i_{r+1},\tau_{i_{r+1}}}\\
				&\quad\times\Big(\int_{\R^{2dr}}\prod\limits_{i\in J_{r,\sigma}}\tilde{b}_i\Big(y_{i,\gamma_i}+ \sum\limits_{(k,\ell)\in\Lambda_{i,\sigma}\setminus D^{}_{r+1,\sigma}}y_{k,\ell}\Big)\prod\limits_{i\in K}B^{(l_i)}_{i,\gamma_i}(y_{i,\gamma_i})E_{i,\tau_i}(y_{i,\tau_i}-y_{i,\gamma_i})\\&\quad\times\prod\limits_{i\in J_{r,\sigma}\setminus K}E_{i,\gamma_i}(y_{i,\gamma_i})B^{(l_i)}_{i,\tau_i}(y_{i,\tau_i}-y_{i,\gamma_i})\prod\limits_{i\in J_{r,\sigma}}\mathrm{d}y_{i,\gamma_i}\mathrm{d}y_{i, \tau_i}\Big)\nonumber\\
				&\quad\times 
				\prod\limits_{i\in I_n\setminus J_{r+1,\sigma}}\frac{\partial \tilde{b}_i^{}}{\partial x_{l_i}}\Big(\sum\limits_{(k,\ell)\in\Lambda_{i,\sigma}\setminus D_{r+1,\sigma}}y_{k,\ell}\Big)\prod\limits_{(k,\ell)\in I_n^2\setminus D^{(\tau)}_{r+1,\sigma}}E_{k,\ell}(y_{k,\ell})\mathrm{d}y_{k,\ell}.
			\end{align*}
			Applying integration by parts with respect to $y^{(l_{i_{r+1}})}_{i_{r+1},\sigma(i_{r+1})}$, we obtain $\mathcal{J}_K=\mathcal{J}_{K\cup\{i_{r+1}\}}+\mathcal{J}_K$ with
			\begin{align*}
				\mathcal{J}_{K\cup\{i_{r+1}\}}=&-\int_{\R^{dn^2-2d(r+1)}}\Big(\int_{\R^{2d}}\tilde{b}_{{i_{r+1}}}^{}\Big(y_{i_{r+1},\gamma_{i_{r+1}}}+\sum\limits_{(k,\ell)\in\Lambda_{i,\sigma}\setminus D_{r+1,\sigma}}y_{k,\ell}\Big)B^{(l_{i_{r+1}})}_{i_{r+1},\gamma_{i_{r+1}}}(y_{i_{r+1},\gamma_{i_{r+1}}})\\&\qquad\times E_{i_{r+1}, \tau_{i_{r+1}}}(y_{i_{r+1},\tau_{i_{r+1}}}-y_{i_{r+1},\gamma_{i_{r+1}}})\,dy_{i_{r+1},\gamma_{i_{r+1}}}dy_{i_{r+1},\tau_{i_{r+1}}}\Big)\\
				&\quad\times\Big(\int_{\R^{2dr}}\prod\limits_{i\in J_{r,\sigma}}\tilde{b}_i\Big(y_{i,\gamma_i}+ \sum\limits_{(k,\ell)\in\Lambda_{i,\sigma}\setminus D^{}_{r+1,\sigma}}y_{k,\ell}\Big)\prod\limits_{i\in K}B^{(l_i)}_{i,\gamma_i}(y_{i,\gamma_i})E_{i,\tau_i}(y_{i,\tau_i}-y_{i,\gamma_i})\\&\quad\times\prod\limits_{i\in J_{r,\sigma}\setminus K}E_{i,\gamma_i}(y_{i,\gamma_i})B^{(l_i)}_{i,\tau_i}(y_{i,\tau_i}-y_{i,\gamma_i})\prod\limits_{i\in J_{r,\sigma}}\mathrm{d}y_{i,\gamma_i}\mathrm{d}y_{i,\tau_i}\Big)\nonumber\\
				&\quad\times 
				\prod\limits_{i\in I_n\setminus J_{r+1,\sigma}}\frac{\partial \tilde{b}_i^{}}{\partial x_{l_i}}\Big(\sum\limits_{(k,\ell)\in\Lambda_{i,\sigma}\setminus D_{r+1,\sigma}}y_{k,\ell}\Big)\prod\limits_{(k,\ell)\in I_n^2\setminus D^{(\tau)}_{r+1,\sigma}}E_{k,\ell}(y_{k,\ell})\mathrm{d}y_{k,\ell}
			\end{align*}
			and
			\begin{align*}
				\mathcal{J}_{K}=&\int_{\R^{dn^2-2d(r+1)}}\Big(\int_{\R^{2d}}\tilde{b}_{{i_{r+1}}}^{}\Big(y_{i_{r+1},\gamma_{i_{r+1}}}+\sum\limits_{(k,\ell)\in\Lambda_{i,\sigma}\setminus D_{r+1,\sigma}}y_{k,\ell}\Big)E_{i_{r+1},\gamma_{i_{r+1}}}(y_{i_{r+1},\gamma_{i_{r+1}}})\\&\qquad\times B^{(l_{i_{r+1}})}_{i_{r+1},\tau_{i_{r+1}}}(y_{i_{r+1},\tau_{i_{r+1}}}-y_{i_{r+1},\gamma_{i_{r+1}}})\,dy_{i_{r+1},\gamma_{i_{r+1}}}dy_{i_{r+1},\tau_{i_{r+1}}}\Big)\\&\quad\times\Big(\int_{\R^{2dr}}\prod\limits_{i\in J_{r,\sigma}}\tilde{b}_i\Big(y_{i,\gamma_i}+ \sum\limits_{(k,\ell)\in\Lambda_{i,\sigma}\setminus D^{}_{r+1,\sigma}}y_{k,\ell}\Big)\prod\limits_{i\in K}B^{(l_i)}_{i,\sigma(i)}(y_{i,\gamma_i})E_{i,\tau_i}(y_{i,\tau_i}-y_{i,\gamma_i})\\&\quad\times\prod\limits_{i\in J_{r,\sigma}\setminus K}E_{i,\gamma_i}(y_{i,\gamma_i})B^{(l_i)}_{i,\tau_i}(y_{i,\tau_i}-y_{i,\gamma_i})\prod\limits_{i\in J_{r,\sigma}}\mathrm{d}y_{i,\gamma_i}\mathrm{d}y_{i,\tau_i}\Big)\nonumber\\
				&
				\quad\times 
				\prod\limits_{i\in I_n\setminus J_{r+1,\sigma}}\frac{\partial \tilde{b}_i^{}}{\partial x_{l_i}}\Big(\sum\limits_{(k,\ell)\in\Lambda_{i,\sigma}\setminus D_{r+1,\sigma}}y_{k,\ell}\Big)\prod\limits_{(k,\ell)\in I_n^2\setminus D^{(\tau)}_{r+1,\sigma}}E_{k,\ell}(y_{k,\ell})\mathrm{d}y_{k,\ell}.
			\end{align*}
			the above can be rewriten as
			\begin{align*}
				&\mathcal{J}_{K\cup\{i_{r+1}\}}\\
				=&-\int_{\R^{dn^2-2d(r+1)}}\Big(\int_{\R^{2d(r+1)}}\prod\limits_{i\in J_{r+1,\sigma}}\tilde{b}_i\Big(y_{i,\gamma_i}+ \sum\limits_{(k,\ell)\in\Lambda_{i,\sigma}\setminus D^{}_{r+1,\sigma}}y_{k,\ell}\Big)\\&\quad\times\prod\limits_{i\in K\cup\{i_{r+1}\}}B^{(l_i)}_{i,\gamma_i}(y_{i,\gamma_i})E_{i,\tau_i}(y_{i,\tau_i}-y_{i,\gamma_i})\\
				&\quad\times\prod\limits_{i\in J_{r+1,\sigma}\setminus (K\cup\{i_{r+1}\})}E_{i,\gamma_i}(y_{i,\gamma_i})B^{(l_i)}_{i,\tau_i}(y_{i,\tau_i}-y_{i,\gamma_i})\prod\limits_{i\in J_{r+1,\sigma}}\mathrm{d}y_{i,\gamma_i}\mathrm{d}y_{i,\tau_i}\Big)\nonumber\\
				&\quad\times 
				\prod\limits_{i\in I_n\setminus J_{r+1,\sigma}}\frac{\partial \tilde{b}_i^{}}{\partial x_{l_i}}\Big(\sum\limits_{(k,\ell)\in\Lambda_{i,\sigma}\setminus D_{r+1,\sigma}}y_{k,\ell}\Big)\prod\limits_{(k,\ell)\in I_n^2\setminus D^{(\tau)}_{r+1,\sigma}}E_{k,\ell}(y_{k,\ell})\mathrm{d}y_{k,\ell}
			\end{align*}
			and
			\begin{align*}
				\mathcal{J}_{K}=&\int_{\R^{dn^2-2d(r+1)}} \Big(\int_{\R^{2d(r+1)}}\prod\limits_{i\in J_{r+1,\sigma}}\tilde{b}_i\Big(y_{i,\gamma_i}+ \sum\limits_{(k,\ell)\in\Lambda_{i,\sigma}\setminus D^{}_{r+1,\sigma}}y_{k,\ell}\Big)\\&\quad\times\prod\limits_{i\in K}B^{(l_i)}_{i,\gamma_i}(y_{i,\gamma_i})E_{i,\tau_i}(y_{i,\tau_i}-y_{i,\gamma_i})\\&\quad\times\prod\limits_{i\in J_{r+1,\sigma}\setminus K}E_{i,\gamma_i}(y_{i,\gamma_i})B^{(l_i)}_{i,\tau_i}(y_{i,\tau_i}-y_{i,\gamma_i})\prod\limits_{i\in J_{r+1,\sigma}}\mathrm{d}y_{i,\gamma_i}\mathrm{d}y_{i,\tau_i}\Big)\nonumber\\
				&\quad\times 
				\prod\limits_{i\in I_n\setminus J_{r+1,\sigma}}\frac{\partial \tilde{b}_i^{}}{\partial x_{l_i}}\Big(\sum\limits_{(k,\ell)\in\Lambda_{i,\sigma}\setminus D_{r+1,\sigma}}y_{k,\ell}\Big)\prod\limits_{(k,\ell)\in I_n^2\setminus D^{(\tau)}_{r+1,\sigma}}E_{k,\ell}(y_{k,\ell})\mathrm{d}y_{k,\ell}.
			\end{align*}
			This completes the proof of \eqref{EqDavieVarSheet1} and \eqref{EqDavieVarSheet2}. Thus for $r=q$, we may write
			\begin{align}
				\label{EqDavieVarShFin1}
				\mathcal{J}:=\E\Big[\prod\limits_{i=1}^n\frac{\partial \tilde{b}_i}{\partial x_{l_i}}\Big(\sum_{k=1}^i\sum_{\ell=1}^{\sigma(i)}Z_{k,\ell}\Big)\Big]=\sum\limits_{K\subset J_{\sigma}}(-1)^{\# K}\mathcal{J}_{K},
			\end{align}
			where $\# K$ is the number of elements in $K$, and
			\begin{align}
				\label{EqDavieVarShFin2}
				\mathcal{J}_{K}=&\int_{\R^{dn^2-2dq}}\Big(\int_{\R^{2dq}}\prod\limits_{i\in J_{\sigma}}\tilde{b}_i\Big(z_{i,\gamma_i}+ \sum\limits_{(k,\ell)\in\Lambda_{i,\sigma}\setminus D^{}_{q,\sigma}}z_{k,\ell}\Big)\prod\limits_{i\in K}B^{(l_i)}_{i,\gamma_i}(z_{i,\gamma_i})E_{i,\tau_i}(z_{i,\tau_i}-z_{i,\gamma_i})\\&\quad\times\prod\limits_{i\in J_{\sigma}\setminus K}E_{i,\gamma_i}(z_{i,\gamma_i})B^{(l_i)}_{i,\tau_i}(z_{i,\tau_i}-z_{i,\gamma_i})\prod\limits_{i\in J_{\sigma}}\mathrm{d}z_{i,\gamma_i}\mathrm{d}z_{i,\tau_i}\Big)\nonumber\\
				&\quad\times 
				\prod\limits_{i\in I_n\setminus J_{\sigma}}\frac{\partial \tilde{b}_i^{}}{\partial x_{l_i}}\Big(\sum\limits_{(k,\ell)\in\Lambda_{i,\sigma}\setminus D_{q,\sigma}}z_{k,\ell}\Big)\prod\limits_{(k,\ell)\in I_n^2\setminus D^{(\tau)}_{r,\sigma}}E_{k,\ell}(z_{k,\ell})\mathrm{d}z_{k,\ell}.\nonumber
			\end{align}	
			Now for every $i\in I_n\setminus J_{\sigma}$ and $K\subset J_{\sigma}$, the variable $z_{i,\sigma(i)}$ appears in only one factor $\frac{\partial \tilde{b}}{\partial x_{l_i}}\Big(\sum_{k=1}^i\sum_{\ell=1}^{\sigma(i)}z_{k,\ell}\Big)$  and there is no need to eliminate this variable in other factors. As a consequence, there is no substitution variable. The integration by parts variable is $z^{(l_i)}_{i,\gamma_i}$, where  
			
			\begin{align*}
				\gamma_i=\gamma_i(K):=&\max\Big\{k\in \sigma(J^{-}_{i,\sigma}\setminus K)\cup\{\sigma(i)\}:\, k\neq\tau_{i_{\ell}}(K\cap J_{\ell-1}),\,\forall\,i_{\ell}\in J^{-}_{i,\sigma}\setminus K\Big\},
			\end{align*}
			where $	J^{-}_{i,\sigma}=\{\ell\in J_{\sigma}:\,\sigma(\ell)<\sigma(i)\}$. We deduce from \eqref{EqDavieVarShFin2} that
			\begin{align*}
				\mathcal{J}_{K}=&\int_{\R^{dn^2-dn-dq}}\Big(\int_{\R^{2dq}}\prod\limits_{i\in J_{\sigma}}\tilde{b}_i\Big(z_{i,\gamma_i}+ \sum\limits_{(k,\ell)\in\Lambda_{i,\sigma}\setminus D^{}_{q,\sigma}}z_{k,\ell}\Big)\prod\limits_{i\in K}B^{(l_i)}_{i,\gamma_i}(z_{i,\gamma_i})E_{i,\tau_i}(z_{i,\tau_i}-z_{i,\gamma_i})\\&\quad\times\prod\limits_{i\in J_{\sigma}\setminus K}E_{i,\gamma_i}(z_{i,\gamma_i})B^{(l_i)}_{i,\tau_i}(z_{i,\tau_i}-z_{i,\gamma_i})\prod\limits_{i\in J_{\sigma}}\mathrm{d}z_{i,\gamma_i}\mathrm{d}z_{i,\tau_i}\Big)\nonumber\\
				&\quad\times \prod\limits_{i\in I_n\setminus J_{\sigma}}
				\Big(\int_{\R}\frac{\partial \tilde{b}_i^{}}{\partial x_{l_i}}\Big(z_{i,\gamma_i}+\sum\limits_{(k,\ell)\in\Lambda_{i,\sigma}\setminus (D_{q,\sigma}\cup\{(i,\gamma_i)\})}z_{k,\ell}\Big)E_{i,\gamma_i}(z_{i,\gamma_i})\,\mathrm{d}z^{(l_i)}_{i,\gamma_i}\Big)\prod\limits_{l\neq l_i}\mathrm{d}z^{(l)}_{i,\gamma_i}\\
				&\quad\times\prod\limits_{(k,\ell)\in I_n^2\setminus (D^{(\tau)}_{r,\sigma}\cup\{(i,\gamma_i):\,i\in I_n\setminus J_{\sigma}\})}E_{k,\ell}(z_{k,\ell})\mathrm{d}z_{k,\ell}.	 
			\end{align*}
			Applying integration by parts with respect to the variables $z^{(l_i)}_{i,\gamma_i}$, $i\in I_n\setminus J_{\sigma}$ we get
			\begin{align*}
				\mathcal{J}_{K}=&(-1)^{n-q}\int_{\R^{dn^2-dn-dq}}\Big(\int_{\R^{2dq}}\prod\limits_{i\in J_{\sigma}}\tilde{b}_i\Big(z_{i,\gamma_i}+ \sum\limits_{(k,\ell)\in\Lambda_{i,\sigma}\setminus D^{}_{q,\sigma}}z_{k,\ell}\Big)\prod\limits_{i\in K}B^{(l_i)}_{i,\gamma_i}(z_{i,\gamma_i})E_{i,\tau_i}(z_{i,\tau_i}-z_{i,\gamma_i})\\&\quad\times\prod\limits_{i\in J_{\sigma}\setminus K}E_{i,\gamma_i}(z_{i,\gamma_i})B^{(l_i)}_{i,\tau_i}(z_{i,\tau_i}-z_{i,\gamma_i})\prod\limits_{i\in J_{\sigma}}\mathrm{d}z_{i,\gamma_i}\mathrm{d}z_{i,\tau_i}\Big)\nonumber\\
				&\quad\times \prod\limits_{i\in I_n\setminus J_{\sigma}}
				\Big(\int_{\R^d}\tilde{b}_i^{}\Big(z_{i,\gamma_i}+\sum\limits_{(k,\ell)\in\Lambda_{i,\sigma}\setminus (D_{q,\sigma}\cup\{(i,\gamma_i)\})}z_{k,\ell}\Big)B^{(l_i)}_{i,\gamma_i}(z_{i,\gamma_i})\,\mathrm{d}z_{i,\gamma_i}\Big)\\
				&\quad\times\prod\limits_{(k,\ell)\in I_n^2\setminus (D^{(\tau)}_{r,\sigma}\cup\{(i,\gamma_i):\,i\in I_n\setminus J_{\sigma}\})}E_{k,\ell}(z_{k,\ell})\mathrm{d}z_{k,\ell}.
			\end{align*}
			Then for any $K\subset J_{\sigma}$,
			\begin{align*}
				|\mathcal{J}_{K}|\leq& \prod\limits_{i=1}^n\Vert \tilde{b}_i\Vert_{\infty}\int_{\R^{dn^2-dn-dq}}\Big(\int_{\R^{2dq}} \prod\limits_{i\in K}|B^{(l_i)}_{i,\gamma_i}(z_{i,\gamma_i})|E_{i,\tau_i}(z_{i,\tau_i}-z_{i,\gamma_i})\\
				&\qquad\quad\times\prod\limits_{i\in J_{\sigma}\setminus K}E_{i,\gamma_i}(z_{i,\gamma_i})|B^{(l_i)}_{i,\tau_i}(z_{i,\tau_i}-z_{i,\gamma_i})|\prod\limits_{i\in J_{\sigma}}\mathrm{d}z_{i,\gamma_i}\,\mathrm{d}z_{i,\tau_i}\Big)\nonumber\\
				&\qquad\quad\times \prod\limits_{i\in I_n\setminus J_{\sigma}}
				\Big(\int_{\R^d} |B^{(l_i)}_{i,\gamma_i}(z_{i,\gamma_i})|\,\mathrm{d}z_{i,\gamma_i}\Big) \prod\limits_{(k,\ell)\in I_n^2\setminus (D^{(\tau)}_{r,\sigma}\cup\{(i,\gamma_i):\,i\in I_n\setminus J_{\sigma}\})}E_{k,\ell}(z_{k,\ell})\mathrm{d}z_{k,\ell}\\
				\leq &2^{n/2}\prod\limits_{i=1}^n\Vert \tilde{b}_i\Vert_{\infty}\prod\limits_{i\in K}(s_i-s_{i-1})^{-1/2}(t_{\gamma_i}-t_{\gamma_i-1})^{-1/2}\prod\limits_{i\in J_{\sigma}\setminus K}(s_i-s_{i-1})^{-1/2}(t_{\tau_i}-t_{\tau_i-1})^{-1/2}\\&\times\prod\limits_{i\in I_n\setminus J_{\sigma}}(s_i-s_{i-1})^{-1/2}(t_{\gamma_i}-t_{\gamma_i-1})^{-1/2}\\=&2^{n/2}\Vert b\Vert^n_{\infty}\prod\limits_{i\in I_n} (s_i-s_{i-1})^{-1/2}(t_i-t_{i-1})^{-1/2}.
			\end{align*}
			Thus,
			\begin{align*}
				|\mathcal{J}|\leq\sum\limits_{K\subset J_{\sigma}}|\mathcal{J}_{K}|\leq(2\sqrt{2})^n\Vert b\Vert^n_{\infty}\prod\limits_{i\in I_n}(s_i-s_{i-1})^{-1/2}(t_i-t_{i-1})^{-1/2}.
			\end{align*}
			The proof is completed.		 		
		\end{proof}

		\begin{proof}[Proof of Corollary \ref{corol:DavieVarSheet}]
			This follows immediately from Proposition \ref{prop:DavieVarSheet} and from the relationship between multivariate Beta function and Gamma function (see e.g. \cite[Lemma 4.3]{Re14}). Precisely, since $0\leq\frac{(s-s_{\gamma(1)})(t-t_{\theta(1)})}{(s-r)(t-u)}\leq1$,  one has
			\begin{align*}
				&\int_{\begin{subarray}{c}
						r<s_{\gamma(n)}<\cdots<s_{\gamma(1)}<s\\
						u<t_{\theta(n)}<\cdots<t_{\theta(1)}<t	
					\end{subarray}
				}\,\Big|\E\Big[\prod\limits_{i=1}^n \frac{\partial b^{(l_{i-1})}}{\partial x_{l_i}}\left(s_i,t_{i},W_{s_i,t_{i}}\right)\Big]\Big|\mathrm{d}s_1\mathrm{d}t_1\ldots \mathrm{d}s_n\mathrm{d}t_n\\
				\leq& C_0^n\|b\|^n_{\infty}\int_{\begin{subarray}{c}
						r=s_{\gamma(n+1)}<s_{\gamma(n)}<\cdots<s_{\gamma(1)}<s\\
						u=t_{\theta(n+1)}<t_{\theta(n)}<\cdots<t_{\theta(1)}<t	
					\end{subarray}
				}\,\prod\limits_{i=1}^n(s_{\gamma(i)}-s_{\gamma(i+1)})^{-1/2}(t_{\theta(i)}-t_{\theta(i+1)})^{-1/2}\\&\qquad\qquad\qquad\qquad\qquad\qquad\qquad\qquad\times\mathrm{d}s_{\gamma(1)}\mathrm{d}t_{\theta(1)}\ldots \mathrm{d}s_{\gamma(n)}\mathrm{d}t_{\theta(n)}\\ \leq& C_0^n\|b\|^n_{\infty}(s-r)^{1/2}\Big(\int_{r=s_{\gamma(n+1)}<s_{\gamma(n)}<\cdots<s_{\gamma(1)}<s}(s-s_{\gamma(1)})^{-1/2}\prod\limits_{i=1}^n(s_{\gamma(i)}-s_{\gamma(i+1)})^{-1/2}\\&\qquad\qquad\qquad\qquad\qquad\qquad\qquad\qquad\times\mathrm{d}s_{\gamma(1)}\ldots \mathrm{d}s_{\gamma(n)}\Big)(t-u)^{1/2}\\&\qquad\times\Big(\int_{u=t_{\theta(n+1)}<t_{\theta(n)}<\cdots<t_{\theta(1)}<t}(t-t_{\theta(1)})^{-1/2}\prod\limits_{i=1}^n(t_{\theta(i)}-t_{\theta(i+1)})^{-1/2}\,\mathrm{d}t_{\theta(1)}\ldots \mathrm{d}t_{\theta(n)}\Big)\\=&\frac{C_0^n\|b\|^n_{\infty}\Gamma(1/2)^{2n+2}(s-r)^{n/2}(t-u)^{n/2}}{\Gamma\Big(\frac{n+1}{2}\Big)^2}.
			\end{align*}
			The proof of \eqref{eq:MMNPZmD} is completed. The proof of \eqref{eq:MMNPZmD2} (and \eqref{eq:MMNPZmD3}) follows similarly. Indeed, we have
			\begin{align*}
				&\int_{\begin{subarray}{c}
						\bar{r}<s_{\sigma(k+n)}<\cdots<s_{\sigma(k+1)}<r<s_{\sigma(k)}<\cdots<s_{\sigma(1)}<s\\
						\bar{u}<t_{\pi(k+n)}<\cdots<t_{\pi(k+1)}<t_{\pi(k)}<\cdots<t_{\pi(1)}<t	
					\end{subarray}
				}\,\Big|\E\Big[\prod\limits_{i=1}^{k+n} \frac{\partial b^{(l_{i-1})}}{\partial x_{l_i}}\left(s_i,t_{i},W_{s_i,t_{i}}\right)\Big]\Big|\mathrm{d}s_1\mathrm{d}t_1\ldots \mathrm{d}s_{k+n}\mathrm{d}t_{k+n}\\ \leq&C_0^{k+n}\|b\|^{k+n}_{\infty}\int_{\begin{subarray}{c}
						\bar{r}=s_{\sigma(k+n+1)}<s_{\sigma(k+n)}<\cdots<s_{\sigma(k+1)}<r<s_{\sigma(k)}<\cdots<s_{\sigma(1)}<s\\
						\bar{u}=t_{\pi(k+n+1)}<t_{\pi(k+n)}<\cdots<t_{\pi(k+1)}<t_{\pi(k)}<\cdots<t_{\pi(1)}<t	
					\end{subarray}
				}\,\prod\limits_{i=1}^{k+n}(s_{\sigma(i)}-s_{\sigma(i+1)})^{-1/2}\\&\qquad\qquad\qquad\qquad\times\prod\limits_{i=1}^{k+n}(t_{\pi(i)}-t_{\pi(i+1)})^{-1/2}\mathrm{d}s_{\sigma(1)}\mathrm{d}t_{\pi(1)}\ldots\mathrm{d}s_{\sigma(k+n)}\mathrm{d}t_{\pi(k+n)}\\=&C_0^{k+n}\|b\|^{k+n}_{\infty}\Big(\int_{\bar{r}=s_{\sigma(k+n+1)}<s_{\sigma(k+n)}<\cdots<s_{\sigma(k+1)}<r<s_{\sigma(k)}<\cdots<s_{\sigma(1)}<s}\prod\limits_{i=1}^{k+n}(s_{\sigma(i)}-s_{\sigma(i+1)})^{-1/2}\\&\qquad\qquad\qquad\qquad\qquad\qquad\times\mathrm{d}s_{\sigma(1)}\ldots\mathrm{d}s_{\sigma(k+n)}\Big)\\&\qquad\times\Big(\int_{\bar{u}=t_{\pi(k+n+1)}<t_{\pi(k+n)}<\cdots<t_{\pi(k+1)}<t_{\pi(k)}<\cdots<t_{\pi(1)}<t}\prod\limits_{i=1}^{k+n}(t_{\pi(i)}-t_{\pi(i+1)})^{-1/2}\,\mathrm{d}t_{\pi(1)}\ldots\mathrm{d}t_{\pi(k+n)}\Big).
			\end{align*}
			Since $0\leq\frac{s-s_{\sigma(1)}}{s-r}<1$,   $0\leq\frac{(r-s_{\sigma(k+1)})(s_{\sigma(k)}-r)}{(r-\bar{r})(s_{\sigma(k)}-s_{\sigma(k+1)})}\leq1$ and $0\leq\frac{t-t_{\pi(1)}}{t-\bar{u}}\leq1$, we have
			\begin{align*}
				&\int_{\bar{r}=s_{\sigma(k+n+1)}<s_{\sigma(k+n)}<\cdots<s_{\sigma(k+1)}<r<s_{\sigma(k)}<\cdots<s_{\sigma(1)}<s}\prod\limits_{i=1}^{k+n}(s_{\sigma(i)}-s_{\sigma(i+1)})^{-1/2}\,\mathrm{d}s_{\sigma(1)}\ldots\mathrm{d}s_{\sigma(k+n)}\\ \leq&(r-\bar{r})^{1/2}\Big(\int_{\bar{r}=s_{\sigma(k+n+1)}<s_{\sigma(k+n)}<\ldots<s_{\sigma(k+1)}<r}(r-s_{\sigma(k+1)})^{-1/2}\prod\limits_{i=k+1}^{k+n}(s_{\sigma(i)}-s_{\sigma(i+1)})^{-1/2}\\&\qquad\qquad\qquad\qquad\qquad\qquad\qquad\qquad\qquad\qquad\qquad\qquad\qquad\times\mathrm{d}s_{\sigma(k+1)}\ldots\mathrm{d}s_{\sigma(k+n)}\Big)\\
				&\qquad\times(s-r)^{1/2}\Big(\int_{r=s_{\sigma(k+1)}<s_{\sigma(k)}<\ldots<s_{\sigma(1)}<s}(s-s_{\sigma(1)})^{-1/2}\prod\limits_{i=1}^k(s_{\sigma(i)}-s_{\sigma(i+1)})^{-1/2}\,\mathrm{d}s_{\sigma(1)}\ldots\mathrm{d}s_{\sigma(k)}\Big)
			\end{align*}
			and
			\begin{align*}
				&\int_{\bar{u}=t_{\pi(k+n+1)}<t_{\pi(k+n)}<\cdots<t_{\pi(k+1)}<t_{\pi(k)}<\cdots<t_{\pi(1)}<t}\prod\limits_{i=1}^{k+n}(t_{\pi(i)}-t_{\pi(i+1)})^{-1/2}\,\mathrm{d}t_{\pi(1)}\ldots\mathrm{d}t_{\pi(k+n)}
				\\ \leq&(t-\bar{u})^{1/2}\int_{\bar{u}=t_{\pi(k+n+1)}<t_{\pi(k+n)}<\cdots<t_{\pi(k+1)}<t_{\pi(k)}<\cdots<t_{\pi(1)}<t}(t-t_{\pi(1)})^{-1/2}\prod\limits_{i=1}^{k+n}(t_{\pi(i)}-t_{\pi(i+1)})^{-1/2}\\&\qquad\qquad\qquad\qquad\qquad\qquad\qquad\qquad\qquad\qquad\qquad\qquad\qquad\qquad\times\mathrm{d}t_{\pi(1)}\ldots\mathrm{d}t_{\pi(k+n)}.
			\end{align*}
			Then, using \cite[Lemma 4.3]{Re14}, one has
			\begin{align*}
				&\int_{\bar{r}=s_{\sigma(k+n+1)}<s_{\sigma(k+n)}<\cdots<s_{\sigma(k+1)}<r<s_{\sigma(k)}<\cdots<s_{\sigma(1)}<s}\prod\limits_{i=1}^{k+n}(s_{\sigma(i)}-s_{\sigma(i+1)})^{-1/2}\,\mathrm{d}s_{\sigma(1)}\ldots\mathrm{d}s_{\sigma(k+n)}\\\leq&\frac{\Gamma(1/2)^{k+n+2}(r-\bar{r})^{n/2}(s-r)^{k/2}}{\Gamma\left(\frac{n+1}{2}\right)\Gamma\left(\frac{k+1}{2}\right)}
			\end{align*}
			and
			\begin{align*}
				&	\int_{\bar{u}=t_{\pi(k+n+1)}<t_{\pi(k+n)}<\cdots<t_{\pi(k+1)}<t_{\pi(k)}<\cdots<t_{\pi(1)}<t}\prod\limits_{i=1}^{k+n}(t_{\pi(i)}-t_{\pi(i+1)})^{-1/2}\,\mathrm{d}t_{\pi(1)}\ldots\mathrm{d}t_{\pi(k+n)}
				\\	\leq&\frac{\Gamma(1/2)^{k+n+1}(t-\bar{u})^{(k+n)/2}}{\Gamma\left(\frac{k+n+1}{2}\right)}.
			\end{align*}
			This completes the proof of \eqref{eq:MMNPZmD2}. The proof of \eqref{eq:MMNPZmD3} follows the same lines as that of \eqref{eq:MMNPZmD2}.
		\end{proof}

		\begin{proof}[Proof of Lemma \ref{lemma:Shuffle}]
			It suffices to prove that  $\{\nabla_{r,s}^{(mk,\sigma)}\}_{\sigma\in\widehat{\mathcal{P}}_{km}}$  is a partition of $(\nabla^{(k)}_{r,s})^k$. We first show that $\nabla_{r,s}^{(mk,\sigma)}\subset(\nabla_{r,s}^{(k)})^m$. Let $\sigma\in\widehat{\mathcal{P}}_{km}$ and  $(s_1,\ldots,s_{mk})\in\nabla_{r,s}^{(mk,\sigma)}$. Since $\sigma((i+1)k)<\ldots<\sigma(1+ik)$, then, by definition of $\nabla_{r,s}^{(mk,\sigma)}$, we have 
			\begin{align*}
				r<s_{\sigma^{-1}\circ\sigma((i+1)k)}<\ldots<s_{\sigma^{-1}\circ\sigma(1+ik)}<s,\,\forall\,i\in\{0,1,\ldots,m-1\},
			\end{align*}
			which can be rewritten as
			\begin{align*}
				r<s_{ (i+1)k}<\ldots<s_{ 1+ik}<s,\,\forall\,i\in\{0,1,\ldots,m-1\}.
			\end{align*}
			This means that $(s_{1+ik},\ldots,s_{ (1+i)k})\in\nabla^{(k)}_{r,s}$ for every $i$ and then $(s_1,\ldots,s_{mk})\in(\nabla^{(k)}_{r,s})^m$. Moreover, $\{\nabla_{r,s}^{(mk,\sigma)}\}_{\sigma}$ and $\{ \nabla_{u,t}^{(mk,\gamma)}\}_{\gamma}$ are clearly families of disjoint nonempty sets. Hence, it remains to show that
			\begin{align}\label{Partition1}
				(\nabla_{r,s}^{(k)})^m\subset\bigcup\limits_{\sigma\in\widehat{\mathcal{P}}_{km}}\nabla_{r,s}^{(mk,\sigma)}.	
			\end{align}	
			Let $(s_1,\ldots,s_{mk})\in(\nabla_{r,s}^{(k)})^m$.   Then
			\begin{align}\label{Partition11}
				s_{(i+1)k}<\ldots<s_{1+ik}\,\text{ for all }i\in\{0,1,\ldots,m-1\}
			\end{align}
			and there exists $\sigma\in\mathcal{P}_{mk}$ such that
			\begin{align}\label{Partition12}
				s_{\sigma^{-1}(mk)}<\ldots<s_{\sigma^{-1}(1)}.
			\end{align}
			Since \eqref{Partition11} is equivalent to
			\begin{align*}
				s_{\sigma^{-1}\circ\sigma((i+1)k)}<\ldots<s_{\sigma^{-1}\circ\sigma(1+ik)}\quad\text{for all }i\in\{0,1,\ldots,m-1\},
			\end{align*}
			then it follows from \eqref{Partition12} that
			\begin{align*}
				\sigma((i+1)k)<\ldots<\sigma(1+ik)\, \text{ for all }i\in\{0,1,\ldots,m-1\},
			\end{align*}
			and therefore $\sigma\in\widehat{\mathcal{P}}_{km}$.
			This ends the proof.
		\end{proof}
		\begin{proof}[Proof of Lemma \ref{lemme:Shuffle2}]
			It suffices to show that $\{\Delta_{\bar r,r,s}^{(m,k,\ell,\pi,\rho)}\}_{\pi,\rho}$ is a partition of $(\Delta^{(k+\ell)}_{\bar r,r,s})^m$.   We first observe that 
			$\{\Delta_{\bar r,r,s}^{(m,k,\ell,\pi,\rho)}\}_{\pi,\rho}$ is a partition of $(\Delta^{(k+\ell)}_{\bar r,r,s})^m$ is a family of disjoint nonempty sets. Let us show that for every $(\pi,\rho)$, $\Delta_{\bar r,r,s}^{(m,k,\ell,\pi,\rho)}\subset(\Delta^{(k+\ell)}_{\bar r,r,s})^m$. \\Let $\pi\in\widehat{\mathcal{P}}^{\ast}_{mk}$, $\rho\in\widehat{\mathcal{P}}^{\ast\ast}_{m\ell}$ and $(s_1,\ldots,s_{m(k+\ell)})\in\Delta_{\bar r,r,s}^{(m,k,\ell,\pi,\rho)}$. Since
			\begin{align*}
				\rho(\zeta_{i,\ell})<\ldots<\rho(\zeta_{i,1})\text{ and  }\pi(\xi_{i,k})<\ldots<\pi(\xi_{i,1})\text{ for all  }i=1,\ldots,m,
			\end{align*}
			then, by definition of $\Delta_{\bar r,r,s}^{(m,k,\ell,\pi,\rho)}$,
			\begin{align*}
				&s_{\rho^{-1}\circ\rho(\zeta_{i,\ell})}<\ldots<s_{\rho^{-1}\circ\rho(\zeta_{i,1})}\text{ and  }s_{\pi^{-1}\circ\pi(\xi_{i,k})}<\ldots<s_{\pi^{-1}\circ\pi(\xi_{i,1})}  \text{ for all }i,
			\end{align*}
			which rewrites
			\begin{align*}
				s_{ \zeta_{i,\ell}}<\ldots<s_{\zeta_{i,1}},\text{ and }s_{ \xi_{i,k}}<\ldots<s_{\xi_{i,1}} \,\text{ for all }i.
			\end{align*}
			This means that $(s_{\xi_{i,1}},\ldots,s_{\xi_{i,k}},s_{\zeta_{i,1}},\ldots,s_{\zeta_{i,\ell}})\in\Delta^{(k+\ell)}_{\bar r,r,s}$ for all $i$, which means that $(s_1,\ldots,s_{m(k+\ell)})\in(\Delta^{(k+\ell)}_{\bar r,r,s})^m$. It remains to show that
			\begin{align}\label{Partition2a}
				(\Delta^{(k+\ell)}_{\bar r,r,s})^m\subset\bigcup\limits_{(\pi,\rho)\in\widehat{\mathcal{P}}^{\ast}_{mk}\times\widehat{\mathcal{P}}^{\ast\ast}_{m\ell} }\Delta_{\bar r,r,s}^{(m,k,\ell,\pi,\rho)}.
			\end{align}
			Let $(s_1,\ldots,s_{m(k+\ell)})\in(\Delta^{(k+\ell)}_{\bar r,r,s})^m$. Then 
			\begin{align}\label{Partition2a1}
				s_{ \zeta_{i,\ell}}<\ldots<s_{\zeta_{i,1}}\text{ and }s_{ \xi_{i,k}}<\ldots<s_{\xi_{i,1}} \,\text{ for all }i
			\end{align}
			and
			there exist $\pi\in\mathcal{P}^{\ast}_{mk}$ and $\rho\in\mathcal{P}^{\ast}_{m\ell}$  
			such that
			\begin{align}\label{Partition2a2}
				&s_{\rho^{-1}(\zeta_{i,\ell})}<\ldots<s_{\rho^{-1}(\zeta_{i,1})},\text{ and }s_{\pi^{-1}(\xi_{i,k})}<\ldots<s_{\pi^{-1}(\xi_{i,1})} \,\text{ for all }i
			\end{align}
			Notice that \eqref{Partition2a1} can be rewritten as
			\begin{align*}
				&s_{\rho^{-1}\circ\rho(\zeta_{i,\ell})}<\ldots<s_{\rho^{-1}\circ\rho(\zeta_{i,1})},\text{ and  }s_{\pi^{-1}\circ\pi(\xi_{i,k})}<\ldots<s_{\pi^{-1}\circ\pi(\xi_{i,1})}  \,\text{ for all }i,
			\end{align*}
			and thus \eqref{Partition2a2} yields
			\begin{align*}
				\rho(\zeta_{i,\ell})<\ldots<\rho(\zeta_{i,1})\text{ and  }\pi(\xi_{i,m})<\ldots<\pi(\xi_{i,1})\,\text{ for all }i,
			\end{align*}
			which means that $(\pi,\rho)\in\widehat{\mathcal{P}}^{\ast}_{mk}\times\widehat{\mathcal{P}}^{\ast\ast}_{m\ell}$. This completes the proof.
		\end{proof}

		\section*{Declarations}
		%
		%
		\begin{itemize}
			\item \textbf{Funding:} The project on which this publication is based has been carried out with funding provided by the Alexander von Humboldt Foundation, under the programme financed by the German Federal Ministry of Education and Research entitled German Research Chair No 01DG15010.
			\item \textbf{Conflict of interest/Competing interests:} The authors declare no conflict of interest.
			\item \textbf{Code availability:}  This manuscript has no associated code. 
		\end{itemize}
		%
		%
		%
		%
		%
		%

			\appendix
			
			\section{Some auxiliary results}\label{Appen}
			

			\begin{theorem}\label{CompactnessCriterion}
				Denote by $\mathcal{F}_{t_{1},\ldots,t_{d}}$ the $\sigma $-algebra generated by
				a $\mathbb{R}^{m}$-valued multi-parameter Wiener process $%
				W_{s_{1},\ldots,s_{d}},0\leq s_{i}\leq t_{i},i=1,\ldots,d$ for $%
				(t_{1},\ldots,t_{d})\in \left[ 0,\infty \right) ^{d}$. Let $X_{n},n\geq 1$ be a
				sequence of $\mathcal{F}_{t_{1},\ldots,t_{d}},\mathcal{B}(\mathbb{R}^{m})$-%
				measurable random variables and let $D_{z_{1},\ldots,z_{d}}$ be the Malliavin
				derivative associated with $W_{s_{1},\ldots,s_{d}},s_{i}\in \left[ 0,\infty
				\right) ,i=1,\ldots,d$. Assume that%
				\begin{equation}
					\sup_{n\geq 1}\left\Vert D_{\cdot }X_{n}\right\Vert _{L^{2}(\Omega \times 
						\left[ 0,t_{1}\right] \times \ldots\times \left[ 0,t_{d}\right] ;\mathbb{R}%
						^{m\times d})}<\infty \text{.}  \label{(i)}
				\end{equation}%
				as well as%
				\begin{equation}
					\sup_{n\geq 1}\int_{\left[ 0,t_{1}\right] \times \ldots\times \left[ 0,t_{d}%
						\right] }\int_{\left[ 0,t_{1}\right] \times \ldots\times \left[ 0,t_{d}\right] }%
					\frac{\left\Vert D_{x}X_{n}-D_{y}X_{n}\right\Vert _{L^{2}(\Omega ;\mathbb{R}%
							^{m\times d})}^{2}}{\left\Vert x-y\right\Vert ^{d+2\beta }}\mathrm{d}x\mathrm{d}y<\infty 
					\label{(ii)}
				\end{equation}%
				for some $\beta \in (0,\frac{1}{2})$ ($\left\Vert \cdot \right\Vert $ norm
				on $\mathbb{R}^{d}$). Then $X_{n},n\geq 1$ is relatively compact in $%
				L^{2}(\Omega ;\mathbb{R}^{m})$.
			\end{theorem}

			\begin{proof}
				The proof is analogous to that of \cite[Proposition 2]{DPMN92} (see also \cite{BMP22}).
			\end{proof}
			
			\bigskip 
			
			In what follows, we also need a version of Girsanov's theorem, which can be
			e.g. found in \cite{DM15} and which requires for its statement some
			notation and definitions: Denote by $(\mathcal{G}_{u}^{\ast })_{0\leq u\leq
				T}$ the filtration given by%
			\begin{equation*}
				\mathcal{G}_{u}^{\ast }=\sigma (W_{s,t}:0\leq s\leq T,0\leq t\leq u),
			\end{equation*}%
			where $(W_{s,t},0\leq s,t\leq T)$ is a $\mathbb{R}^{d}$-valued Wiener sheet.
			Let $\mathcal{G}:=(\mathcal{G}_{u})_{0\leq u\leq T}$ be the completed and
			right-continuous filtration with respect to $\mathcal{G}_{u}^{\ast },0\leq
			u\leq T$. Consider a measurable random field $Z:\Omega \times \mathcal{T}^2 \longrightarrow \mathbb{R}^{d}$, which is
			adapted to $\mathcal{G}$, that is for all $0\leq s,t\leq 1$ $Z_{s,t}$ is $%
			\mathcal{G}_{t}$-measurable. Assume that%
			\begin{equation}
				\E\Big[ \int_{\mathcal{T}^2 }\left\Vert
				Z_{r}\right\Vert ^{2}\mathrm{d}r\Big] <\infty   \label{SquareInt}
			\end{equation}%
			and define for $0\leq u\leq 1$ the Doleans-Dade type of exponential 
			\begin{equation*}
				M_{u}=\mathcal{E}(\int_{\left[ 0,T\right] \times \left[ 0,u\right]
				}\left\langle Z_{r},\mathrm{d}W_{r}\right\rangle ):=\exp (\int_{\left[ 0,T\right]
					\times \left[ 0,u\right] }\left\langle Z_{r},\mathrm{d}W_{r}\right\rangle -\frac{1}{2}%
				\int_{\left[ 0,T\right] \times \left[ 0,u\right] }\left\Vert
				Z_{r}\right\Vert ^{2}\mathrm{d}r)\text{,}
			\end{equation*}%
			where $\left\langle Z_{r},\mathrm{d}W_{r}\right\rangle
			:=\sum_{i=1}^{d}Z_{r}^{(i)}\mathrm{d}W_{r}^{(i)}$.
			
			\begin{theorem}[Cameron-Martin-Girsanov]\label{Girsanov}
				Let $Z:\Omega \times \mathcal{T}^2
				\longrightarrow \mathbb{R}^{d}$ be a measurable $\mathcal{G}-$adapted random
				field satisfying \eqref{SquareInt} and assume that $M_{u},0\leq u\leq T$ is
				a $\mathcal{G}-$adapted on the probability space $(\Omega ,\mathcal{F},P)$.
				Then the random field $W_{s,t}^{\ast },0\leq s,t\leq T$ defined by%
				\begin{equation*}
					W_{s,t}^{\ast }:=W_{s,t}-\int_{\left[ 0,s\right] \times \left[ 0,t\right]
					}Z_{r}\mathrm{d}r
				\end{equation*}%
				is a $\Q$-Wiener sheet with $\frac{\mathrm{d}\Q}{\mathrm{d}\Pb}:=M_{T}$.   
			\end{theorem}

			\bigskip 
			Using Girsanov`s theorem for random fields we can construct just as in the
			oneparameter case \emph{weak solutions} to SDEs driven by a Wiener sheet. In
			order to introduce this concept of solution we need the following property
			with respect to filtrations $\mathcal{F}=\left\{ \mathcal{F}_{s,t}\right\}
			_{0\leq s,t\leq T}$: For all $0\leq s,t\leq T$ the $\sigma -$algebras   
			\begin{equation}
				\vee _{0\leq u\leq T}\mathcal{F}_{u,t}\text{ and }\vee _{0\leq v\leq T}%
				\mathcal{F}_{s,v}  \label{F}
			\end{equation}%
			are conditionally independent given $\mathcal{F}_{s,t}$. The following definition of weak solution is consistent with Definition \ref{DefWeakSol}.
			\begin{defi}[Weak solution]\label{WeakSolution}
				Let $b:\mathcal{T}^2\times\R^d
				\longrightarrow \mathbb{R}^{d}$ be a Borel measurable function of linear
				growth. Suppose that there exists a probability space $\left( \Omega ,%
				\mathcal{A},P\right) $, a completed system $\mathcal{F}=\left\{ \mathcal{F}%
				_{s,t}\right\} _{0\leq s,t\leq T}$ of non-decreasing and right-continuous
				sub-$\sigma $-algebras of $\mathcal{A}$, satisfying the condition (\ref{F}),
				a Wiener sheet with respect to $\mathcal{F}$ and a continuous $\mathcal{F}-$%
				adapted random field $X_{\cdot }$ on $\mathcal{T}^2$ such that $%
				X_{\cdot }$ solves the SDE%
				\begin{equation}
					X_{s,t}=x+\int_{0}^{s}\int_{0}^{t}b(r,X_{r_{1},r_{2}})dr_{1}dr_{2}+W_{s,t},x%
					\in \mathbb{R}^{d}  \label{SDEWeak}
				\end{equation}%
				$P-$a.e. for all $s,t\in \left[ 0,T\right] $. Then the triple $(X_{\cdot
				},W_{\cdot }),\left( \Omega ,\mathcal{A},P\right) ,\mathcal{F}$ is called
				weak solution to \eqref{SDEWeak}. We also say that the SDE \eqref{SDEWeak}
				has a unique weak solution if all solutions have the same law on $C(\mathcal{T}^2;\mathbb{R}^{d})$.
				
			\end{defi}
			\begin{remark}
				The natural filtration $\left\{ \mathcal{F}_{s,t}^{W}\right\} _{0\leq
					s,t\leq T}$, which is generated by a Wiener sheet $W_{s,t}, 0\leq s,t\leq T$ and
				the sets of probability measure zero, satisfies the conditions of Definition %
				\ref{WeakSolution}. 
			\end{remark}

			\begin{prop}[Existence of a unique weak solution]\label{UniqueWeakSolution}
				Let $b\in L^{\infty }(\mathcal{T}^2\times \mathbb{R}^{d};\mathbb{R}%
				^{d})$. Then there exists a unique weak solution $X_{\cdot }$ on $\mathcal{T}^2$ to \eqref{SDEWeak}.
			\end{prop}
			
			\begin{proof}
				\bigskip Let $(W_{\cdot },\mathcal{F})$ on some probability space $\left(
				\Omega ,\mathcal{A},\Pb\right) $. Then, using Novikov `s theorem, we see that
				the Doleans-Dade exponential in Theorem \ref{Girsanov} is a $\mathcal{F}$-%
				martingale for $\mathcal{F}=\mathcal{G}$. Then it follows that $W_{\cdot
				}^{x}:=x+W_{\cdot }$ satisfies the SDE%
				\begin{equation*}
					W_{s,t
					}^{x}=x+\int_{0}^{s}%
					\int_{0}^{t}b(r,W_{r_{1},r_{2}}^{x})\mathrm{d}r_{1}\mathrm{d}r_{2}+W_{s,t}^{\ast },
				\end{equation*}%
				where $W_{\cdot }^{\ast }$ given by $W_{s,t}^{\ast }=$ $W_{s,t}-\int_{0}^{s}%
				\int_{0}^{t}b(r,W_{r_{1},r_{2}}^{x})\mathrm{d}r_{1}\mathrm{d}r_{2}$ is a $\Q$-Wiener sheet. So
				we obtain a weak solution.
				
				Assume we have two weak solutions $X_{\cdot },Y_{\cdot }$ with respect $\Pb$
				and $\Pb^{\ast }$, respectively. Then for pairwise different $%
				(s_{i},t_{i}),i=1,\ldots,n$ in $\mathcal{T}^2$ we find for all $A\in 
				\mathcal{B}(\mathbb{R}^{2dn})$ by means of Girsanov`s theorem that 
				\begin{align*}
					&\Pb(((X_{s_{1},t_{1}}),\ldots,(X_{s_{n},t_{n}})) \in A) \\
					=&\E_{\Q}\Big[ \chi _{\left\{ ((x+W_{s_{1},t_{1}}^{\ast
						}),\ldots,(x+W_{s_{n},t_{n}}^{\ast }))\in A\right\} }\mathcal{E}(\int_{\left[
						0,T\right] \times \left[ 0,u\right] }\left\langle
					b(r_{1},r_{2},x+W_{r_{1},r_{2}}^{\ast }),\mathrm{d}W_{r}^{\ast }\right\rangle )\Big]  \\
					=&\E_{\Pb}\Big[ \chi _{\left\{
						((x+W_{s_{1},t_{1}}),\ldots,(x+W_{s_{n},t_{n}}))\in A\right\} }\mathcal{E}%
					(\int_{\left[ 0,T\right] \times \left[ 0,u\right] }\left\langle
					b(r_{1},r_{2},x+W_{r_{1},r_{2}}),\mathrm{d}W_{r}\right\rangle )\Big] .
				\end{align*}%
				In the same way, we get that%
				\begin{align*}
					&\Pb^{\ast }(((Y_{s_{1},t_{1}}),\ldots,(Y_{s_{n},t_{n}})) \in A) \\
					=&\E_{\Pb}\Big[ \chi _{\left\{
						((x+W_{s_{1},t_{1}}),\ldots,(x+W_{s_{n},t_{n}}))\in A\right\} }\mathcal{E}%
					(\int_{\left[ 0,T\right] \times \left[ 0,u\right] }\left\langle
					b(r_{1},r_{2},x+W_{r_{1},r_{2}}),\mathrm{d}W_{r}\right\rangle )\Big] .
				\end{align*}%
				So $X_{\cdot }$ and $Y_{\cdot }$ coincide in law.  
			\end{proof}

			\section{Further results}\label{Appen1}
			\bigskip
			Finally, we also need the following auxiliary result with respect to our algorithm:
			\begin{lemma}
				\label{SubstitutionVariable}Assume we already determined for $i\in \left\{
				1,\ldots,n\right\} $ the $(i-1)$ integration by parts variables with the
				corresponding substitutions variables (if existing) with respect to a
				certain path of the algorithm as described at the beginning of Section 4. Suppose in the $i-$th step of the algorithm, where we move to the
				orientation point $(z_{i,\sigma ^{-1}(i)},z_{i,\sigma ^{-1}(i)+1})$, that $%
				z_{i,\sigma ^{-1}(i)+1}$ is a substitution variable for $z_{i,\sigma
					^{-1}(i)}$. Let $z_{i,\sigma ^{-1}(i)+1}$ be in the column of a factor $%
				B_{\sigma ^{-1}(i)+1,r}$ for $r<i$ and let $k\geq 1$ be the smallest number
				such $z_{i,\sigma ^{-1}(i)+1+k}$ is not in the same column as another
				factor. Then $z_{i,\sigma ^{-1}(i)+1+k}$ is also substitution variable for $%
				z_{i,\sigma ^{-1}(i)}$.
			\end{lemma}
			
			\begin{proof}
				By construction there must be $k$ non-overlapping factors $B_{l_{r},r}$, $%
				r=i-k,\ldots,i-1$ in the previous rows with $\sigma ^{-1}(i)+1\leq $ $%
				l_{j}<\sigma ^{-1}(i)+1+k$ for all $j$ and $l:\left\{ i-k,\ldots,i-1\right\}
				\longrightarrow \left\{ 1,\ldots,n\right\} $ an injection, which arose from
				(new) integration by parts or subtitution variables (through integration by
				parts) in the $(i-1)-$th step of the algorithm. Further, let $%
				(z_{r,q_{r}},z_{r,q_{r}+1})$ be the orientation point in row $r$ associated
				with $B_{l_{r},r}$, where $q:\left\{ i-k,\ldots,i-1\right\} \longrightarrow
				\left\{ 1,\ldots,n\right\} $ must be injective. Suppose that $q_{r}\geq \sigma
				^{-1}(i)+1+k$ for some $r$. Then, according to our algorithm $B_{l_{r},r}$
				must correspond to an integration by parts variable, which arises from
				shifts to the left (in the same row). This actually implies that there must
				be a factor $B_{\sigma ^{-1}(i)+1+k,r}$ in the column of $z_{i,\sigma
					^{-1}(i)+1+k}$, which is not possible. So $q_{r}<\sigma ^{-1}(i)+1+k$ for
				all $r=i-k,\ldots,i-1$. Define $p=q_{r^{\ast }}=\min_{r=i-k,\ldots,i-1}q_{r}$.
				Consider now the case, when $\sigma ^{-1}(i)+1\leq p$. Then $p=\sigma
				^{-1}(i)+1$, since $q$ is injective. The latter, however, implies that (the first coordinate) of an orientation
				point $z_{l,\sigma ^{-1}(l)},l>i$, whose span contains $z_{i,\sigma ^{-1}(i)}
				$, must come after $z_{l,\sigma ^{-1}(i)+k}$ in row $l$, that is $\sigma
				^{-1}(l)\geq \sigma ^{-1}(i)+k+1$. So $z_{i,\sigma ^{-1}(i)+1+k}$ is a
				substitution variable for $z_{i,\sigma ^{-1}(i)}$. Assume now that $p<\sigma
				^{-1}(i)+1$. So $p\leq $ $\sigma ^{-1}(i)$. But $p\neq \sigma ^{-1}(i)$.
				Hence, $p<$ $\sigma ^{-1}(i)$. So there exists a minimal number $m\leq
				p+1\leq \sigma ^{-1}(i)$ for which there are non-overlapping factors $%
				B_{v_{r},r}$, $r=i-\sigma ^{-1}(i)-1+m-k,\ldots,i-1$ with $m\leq v_{j}<\sigma
				^{-1}(i)+1+k$ for all $j$, where $v:\left\{ i-\sigma
				^{-1}(i)-1+m-k,\ldots,i-1\right\} \longrightarrow \left\{ 1,\ldots,n\right\} $ is
				injective and where the restriction of $v$ to $\left\{ i-k,\ldots,i-1\right\} $
				is $l$. Further, let $(z_{r,q_{r}^{\ast }},z_{r,q_{r}^{\ast }+1})$ be the
				orientation point in row $r$ associated with $B_{v_{r},r}$, where $q^{\ast
				}:\left\{ i-\sigma ^{-1}(i)-1+m-k,\ldots,i-1\right\} \longrightarrow \left\{
				1,\ldots,n\right\} $ is injective with the restriction of $q^{\ast }$ to $%
				\left\{ i-k,\ldots,i-1\right\} $ given by $q$. Using the same argument as in
				the previous case, we find that $q_{r}^{\ast }<\sigma ^{-1}(i)+1+k$ for all $%
				r$. On the other hand, $q_{r}^{\ast }\geq m-1$ for all $r$. Otherwise there
				is some $r_{0}$ such that $q_{r_{0}}^{\ast }<m-1$. But this means that there
				is another factor $B_{m-1,w}$ in a column $m-1$, which contradicts the
				minimality of $m$. So $m-1\leq q_{r}^{\ast }<\sigma ^{-1}(i)+1+k$ for all $%
				r=i-\sigma ^{-1}(i)-1+m-k,\ldots,i-1$. But the latter entails that the (first
				components of the) orientation points $z_{i,\sigma ^{-1}(i)}$ and $%
				z_{r,q_{r}^{\ast }},r=i-\sigma ^{-1}(i)-1+m-k,\ldots,i-1$ fill each of the
				columns from $m-1$ to $\sigma ^{-1}(i)+k$ (so $\sigma ^{-1}(i)-m+2+k$
				columns). This means an orientation point $z_{l,\sigma ^{-1}(l)},l>i$, whose
				span contains $z_{i,\sigma ^{-1}(i)}$, must come after $z_{l,\sigma
					^{-1}(i)+k}$ in row $l$, that is $\sigma ^{-1}(l)\geq \sigma ^{-1}(i)+k+1$.
				
				Altogether, we conclude that $z_{i,\sigma ^{-1}(i)+1+k}$ is a substitution
				variable for $z_{i,\sigma ^{-1}(i)}$.
				
			\end{proof}
			
			\bigskip
			
			The next auxiliary result shows that "shifts" of orientation points to the
			left or right in row $i$ , which are in the column of a factor with row $r<i$%
			, are always possible:
			
			\begin{lemma}
				\label{Shift}Assume as in Lemma \ref{SubstitutionVariable} a certain path of 
				$(i-1)$ factors in the $(i-1)-$th step of our algorithm. Let $(z_{i,\sigma
					^{-1}(i)},z_{i,\sigma ^{-1}(i)+1})$ be the orientation point. If $%
				z_{i,\sigma ^{-1}(i)}$ is in the column of a factor $B_{\sigma ^{-1}(i),r}$
				for $r<i$, then a "shift" of $z_{i,\sigma ^{-1}(i)}$ to the left in row $i$
				is possible, that is there exists a $k\geq 1$ such that $z_{i,\sigma
					^{-1}(i)-k}$ is not in the column of a factor $B_{\sigma ^{-1}(i)-k,r}$, $%
				r<i $. Similarly, if $z_{i,\sigma ^{-1}(i)+1}$ is in the column of a factor $%
				B_{\sigma ^{-1}(i)+1,r}$ for $r<i$, then a "shift" of $z_{i,\sigma
					^{-1}(i)+1}$ to the right in row $i$ is possible, that is there exists a $%
				k\geq 1$ such that $z_{i,\sigma ^{-1}(i)+1+k}$ is not in the column of a
				factor $B_{\sigma ^{-1}(i)+1+k,r}$, $r<i$.
			\end{lemma}
			
			\begin{proof}
				Let us first consider "shifts" of $z_{i,\sigma ^{-1}(i)}$ to the left.\
				Assume that such a shift is not possible. Then there is a maximal number $%
				m\geq \sigma ^{-1}(i)$ of factors $B_{l(r),r}$, $r=i-m,\ldots,i-1$, which fill
				each of the columns from $1$ to $m$. So $l:\left\{ i-m,\ldots,i-1\right\}
				\longrightarrow \left\{ 1,\ldots,m\right\} $ is a bijection. Let $%
				(z_{r,q(r)},z_{r,q(r)+1})$ be the orientation point in row $r$ associated
				with $B_{l(r),r}$, where $q:\left\{ i-m,\ldots,i-1\right\} \longrightarrow
				\left\{ 1,\ldots,n\right\} $ is an injection. Then we can argue as in the proof
				of Lemma \ref{SubstitutionVariable} and conclude that $q(r)\leq m$ for all $%
				r=i-m,\ldots,i-1$. This means that the (first components of the) orientation
				points $z_{r,q(r)}$, $r=i-m,\ldots,i-1$ fill all the columns from $1$ to $m$.
				However, $z_{i,\sigma ^{-1}(i)}$ is in column $\sigma ^{-1}(i)\in \left\{
				1,\ldots,m\right\} $, which leads to a contradiction. So a shift to the left
				must be possible.
				
				Let us now look at "shifts" to the right. Suppose such a shift (in the same
				row) is not possible. Then there is a minimal number $m\leq \sigma ^{-1}(i)+1
				$ of factors $B_{l(r),r}$, $r=i-n+m-1,\ldots,i-1$, which fill each of the
				columns from $m$ to $n$. Hence, $l:\left\{ i-n+m-1,\ldots,i-1\right\}
				\longrightarrow \left\{ m,\ldots,n\right\} $ is a bijection. Let $%
				(z_{r,q(r)},z_{r,q(r)+1})$ be the orientation point in row $r$ associated
				with $B_{l(r),r}$, where $q:\left\{ i-n+m-1,\ldots,i-1\right\} \longrightarrow
				\left\{ 1,\ldots,n\right\} $ is an injection. Then we can use the same
				arguments as in the proof of Lemma \ref{SubstitutionVariable} and show that $%
				q(r)\geq m-1$. But the latter entails that $z_{i,\sigma ^{-1}(i)}$ and $%
				z_{r,q(r)}$, $r=i-n+m-1,\ldots,i-1$ fill all the columns from $m-1$ to $n$. The
				latter, however, implies that we cannot find (the first coordinate) of an
				orientation point $z_{l,\sigma ^{-1}(l)},l>i$, whose span contains $%
				z_{i,\sigma ^{-1}(i)}$, which contradicts the assumption that $z_{i,\sigma
					^{-1}(i)+1}$ is a substitution variable. Hence, a shift to right must be
				possible.
			\end{proof}

			\section{Complement to the proof of Proposition \ref{prop:DavieVarSheet}}\label{Appen3}				
			
			\begin{lemma}\label{lemd0}
				Let   $n\in\N$, $\sigma$ a permutation on $I_n=\{1,\ldots,n\}$, $J_{\sigma}=\{i\in I_n:\text{ there is }k\in I_n\text{ s.t. }(i,\sigma(i)\prec(k,\sigma(k)))\}$   $($we suppose $J_{\sigma}=J_{q,\sigma}=\{i_1,\cdots,i_q\}$, $q\leq n-1$, with $i_1<\cdots<i_q)$. For any $r\in\{1,\ldots,r\}$, let   $J_{r,\sigma}=\{i_1,\cdots,i_r\}$ with  $J_{0,\sigma}=\{i_1\}=J_{1,\sigma}$, $J^{-}_{r,\sigma}=\{k\in J_{r,\sigma}:\,\sigma(k)<\sigma(i_{r+1})\}$, $J^{+}_{r,\sigma}=\{k\in J_{r,\sigma}:\,\sigma(k)>\sigma(i_{r+1})\}$, $\sigma(J^{-}_{r,\sigma}\setminus K)=\{\sigma(k):\,k\in J^{-}_{r,\sigma}\setminus K\}$ and $\sigma^{+}(K\cap J^{+}_{r,\sigma})=\{\sigma(k)+1:\,k\in K\cap J^{+}_{r,\sigma}\}$.   Define   $\gamma_{i_1}(\emptyset)=\sigma(i_1)$, $\tau_{i_1}(\emptyset)=\sigma(i_1)+1$ and for any $r\in\{1,\ldots,q-1\}$ and any $K\subset J_{q,\sigma}$, 
				\begin{align*}
					\gamma_{i_{r+1}}(K\cap J_{r,\sigma}):=\max \mathcal{J}_{1,r} \text{ and }
					\tau_{i_{r+1}}(K\cap J_{r,\sigma}):=\min \mathcal{J}_{2,r},
				\end{align*}
				where
				\begin{align*}
					\mathcal{J}_{1,r}=&	\Big\{k\in \sigma(J^{-}_{r,\sigma}\setminus K)\cup\{\sigma(i_{r+1})\}:\, k\neq\tau_{i_\ell}(K\cap J_{\ell-1,\sigma}),\,\forall\,i_{\ell}\in J^{-}_{r,\sigma}\setminus K\Big\},\\
					\mathcal{J}_{2,r}:=&\Big\{k\in \sigma^{+}(K\cap J^{+}_{r,\sigma})\cup\{\sigma(i_{r+1})+1\}:\,k\neq \gamma_{i_m}(K\cap J_{m-1,\sigma}),\,\forall\,i_m\in K\cap J^{+}_{r,\sigma} \Big\}.
				\end{align*}
				For any $i\in I_n\setminus J_{q,\sigma}$, define also
				\begin{align*}
					\gamma_i(K):=&\max\Big\{k\in \sigma(J^{-}_{i,\sigma}\setminus K)\cup\{\sigma(i)\}:\,  k\neq\tau_{i_{\ell}}(K\cap J_{\ell-1}),\,\forall\,i_{\ell}\in J^{-}_{i,\sigma}\setminus K\Big\},
				\end{align*}
				where $	J^{-}_{i,\sigma}=\{\ell\in J_{\sigma}:\,\sigma(\ell)<\sigma(i)\}$.
				Let $(i,\ell)\in (I_n\setminus J_{q,\sigma})^2$ and $(i_k,i_m)\in J_{q,\sigma}^2$ such that $i_k<i_m$. 
				\begin{enumerate}
					\item[(A1)]If $i_k\in J^{-}_{m-1,\sigma}\backslash K$ (resp. $i_k\in J^{+}_{m-1,\sigma}\cap K$), then  $\tau_{i_k}(K\cap J_{k-1,\sigma})<\tau_{i_m}(K\cap J_{m-1,\sigma})$ (resp. $\gamma_{i_m}(K\cap J_{m-1,\sigma})<\gamma_{i_k}(K\cap J_{k-1,\sigma})$).
					\item[(A2)] Suppose $i_m \in K$ and $i_k\in J^{-}_{m-1,\sigma}\cap K$ (resp. $i_m \in J_{q,\sigma}\setminus K$ and $i_k\in J^{+}_{m-1,\sigma}\setminus K $), then  $\gamma_{i_k}(K\cap J_{k-1,\sigma})\neq\gamma_{i_m}(K\cap J_{m-1,\sigma})$ (resp. $\tau_{i_m}(K\cap J_{m-1,\sigma})\neq\tau_{i_k}(K\cap J_{k-1,\sigma})$).
					\item[(A3)] Suppose that $i_k<i$ and $i_k\in K$. If $i_k\in J_{q,\sigma}\setminus J^{-}_{i,\sigma}$, then $\gamma_{i_k}(K\cap J_{k-1,\sigma})>\gamma_i(K)$. Otherwise,
					if $i_k\in J^{-}_{i,\sigma}$, then $\gamma_{i_k}(K\cap J_{k-1,\sigma})\neq\gamma_i(K)$. 
					\item[(A4)] If $i<\ell$ $($that is $\sigma(i)>\sigma(\ell))$, then $\gamma_{\ell}(K)<\gamma_{i}(K)$. 
				\end{enumerate}
			\end{lemma}
			\begin{proof}We only prove (A1) and (A2) since the proofs of (A3)-(A4) are similar.\\
			
					Let $K\subset J_{{m-1},\sigma}=\{i_1,\,\ldots,\,i_{m-1}\}$. We prove the following claim by induction, from which assertion (A1) will follow.\\
					\textbf{Claim. }For every $k\in\{1,\ldots,m-1\}$, $\tau_{i_k}(K\cap J_{k-1,\sigma})\leq\sigma(i_m)$ if $i_k\in J^{-}_{m-1,\sigma}\setminus K$ and $\gamma_{i_k}(K\cap J_{k-1,\sigma})>\sigma(i_m)$ if $i_k\in J^{+}_{m-1,\sigma}\cap K$.
					\begin{proof}[Proof of the Claim]
						For $k=1$ we have $\tau_{i_1}(\emptyset)=\sigma(i_1)+1\leq\sigma(i_m)$ if $i_1\in J^{-}_{m-1,\sigma}\setminus K$ and $\gamma_{i_1}(\emptyset)=\sigma(i_1)>\sigma(i_m)$ if $i_1\in J^{+}_{m-1,\sigma}\cap K$. Suppose now that for any $q\in\{1,\ldots,k-1\}$ the Claim holds and let show that the Claim remains true for $k$. For $i_k\in J^{-}_{m-1,\sigma}\setminus K$, we have
						\begin{align*}
							&	\tau_{i_k}(K\cap J_{k-1,\sigma})\\=&\min\{p\in\sigma^{+}(J^{+}_{k-1,\sigma}\cap J^{-}_{m-1,\sigma}\cap K):\,p\neq\gamma_{i_{\ell}}(K\cap J_{\ell-1,\sigma}),\,\forall\,i_{\ell}\in J^{+}_{k-1,\sigma}\cap J^{-}_{m-1,\sigma}\cap K\}.
						\end{align*}
						Indeed, one has  $J^{+}_{k-1,\sigma}=(J^{+}_{k-1,\sigma}\cap J^{-}_{m-1,\sigma})\cup (J^{+}_{m-1,\sigma}\cap J_{k-1,\sigma})$ and, by induction hypothesis, $\gamma_{i_{\ell}}(K\cap J_{\ell-1,\sigma})>\sigma(i_m)\geq p$ for any $i_{\ell}\in J^{+}_{m-1,\sigma}\cap J_{k-1,\sigma}\cap K$ and any  $p\in\sigma^{+}(J^{+}_{k-1,\sigma}\cap J^{-}_{m-1,\sigma}\cap K)$.
						As a consequence, $\tau_{i_k}(K\cap J_{k-1,\sigma})\leq\sigma(i_m)$. For $i_k\in J^{+}_{m-1,\sigma}\cap K$, one also has
						\begin{align*}
							&\gamma_{i_k}(K\cap J_{k-1,\sigma})\\=&\max\{p\in\sigma^{}(J^{-}_{k-1,\sigma}\cap J^{+}_{m-1,\sigma}\setminus K):\,p\neq\tau_{i_{\ell}}(K\cap J_{\ell-1,\sigma}),\,\forall\,i_{\ell}\in J^{-}_{k-1,\sigma}\cap J^{+}_{m-1,\sigma}\setminus K\}
						\end{align*}
						since by induction hypothesis, $\tau_{i_{\ell}}(K\cap J_{\ell-1,\sigma})\leq\sigma(i_m)<p$ for any $i_{\ell}\in J^{-}_{m-1,\sigma}\setminus K$ and any $p\in\sigma^{}(J^{-}_{k-1,\sigma}\cap J^{+}_{m-1,\sigma}\setminus K)$. Then $\gamma_{i_k}(K\cap J_{k-1,\sigma})>\sigma(i_m)$. This ends the proof of the Claim.
					\end{proof}
					We deduce from the preceding Claim that $\tau_{i_k}(K\cap J_{k-1,\sigma})\leq\sigma(i_m)<\tau_{i_m}(K)$ if $i_k\in J^{-}_{m-1,\sigma}\setminus K$ and $ \gamma_{i_k}(K\cap J_{k-1,\sigma})\geq\sigma(i_m)>\gamma_{i_m}(K)$ if $i_k\in J^{+}_{m-1,\sigma}\cap K$. This completes the proof of (A1).
				
				We only prove the assertion (A2) for $i_m\in K$ and $i_k\in J^{-}_{m-1,\sigma}\cap K$. The proof for $i_m\in J_{q,\sigma}\setminus K$ and $i_k\in J^{+}_{m-1}\setminus K$ is similar. First, observe that if $i_m\in K$ and $\gamma_{i_m}(J_{m-1,\sigma}\cap K)=\sigma(i_m)$, then assertion (A2) holds by the definition of $\gamma_{i_k}$. Thus, it suffices to show that $\gamma_{i_k}(K\cap J_{k-1,\sigma})\neq\gamma_{i_m}(K\cap J_{m-1,\sigma})$ for any $i_m\in K$ such that $ \gamma_{i_m}(J_{m-1,\sigma}\cap K)<\sigma(i_m)$. To this end, we use the following claim: \\ 
					\textbf{Claim. } Let $K\subset J_{\sigma}$ be fixed. Suppose $i_m\in K$ satisfies $ \gamma_{i_m}(J_{m-1,\sigma}\cap K)<\sigma(i_m)$. Then, there exists an integer $r\in\N$ and a finite sequence $\{j_1,\,\ldots,\,j_r\}\subset J^{-}_{m-1,\sigma}\setminus K$ such that
					\begin{itemize}
						\item  $ \gamma_{i_m}(J_{m-1,\sigma}\cap K)=\sigma(j_1)$,
						\item $\tau_{j_{\ell}}(J_{m-1,\sigma}\cap K)=\sigma(j_{\ell+1})$, for $j\in\{1,\,\ldots,\,\ell-1\}$,  and
						\item $\tau_{j_r}(J_{m-1,\sigma}\cap K)=\sigma(i_m)$.
					\end{itemize}  
					\begin{proof}[Proof of the Claim.]
						We proceed by induction on the number $\#(J^{-}_{m-1,\sigma}\setminus K)$ of elements in $J^{-}_{m-1,\sigma}\setminus K$. Suppose first that $\#(J^{-}_{m-1,\sigma}\setminus K)=1$ with $J^{-}_{m-1,\sigma}\setminus K=\{i_p\}$. Then $\gamma_{i_m}(J_{m-1,\sigma}\cap K)=\sigma(i_p)$ since $\gamma_{i_m}(J_{m-1,\sigma}\cap K)<\sigma(i_m)$ and $\gamma_{i_m}(K)\in\sigma(J^{-}_{m-1,\sigma}\setminus K)=\{\sigma(i_p)\}$. Moreover, $\tau_{i_p}(J_{p-1,\sigma}\cap K)=\sigma(i_m)$. Indeed, if $\tau_{i_p}(J_{p-1,\sigma}\cap K)\neq\sigma(i_m)$, then, by definition of $\gamma_{i_m}(J_{m-1,\sigma}\cap K)$, $\gamma_{i_m}(J_{m-1,\sigma}\cap K)=\sigma(i_m)$ which is a contradiction.\\ Suppose that the claim holds for any $i_m\in K$ and for $\#(J^{-}_{m-1,\sigma}\setminus K)\in\{1,\,\ldots,\,r\}$. Let us show that the claim still holds for $\#(J^{-}_{m-1,\sigma}\setminus K)=r+1$. Since $\gamma_{i_m}(J_{m-1,\sigma}\cap K)<\sigma(i_m)$, there exists $i_p,\,i_s\in J^{-}_{m-1,\sigma}\setminus K$ and $i_q\in J_{m,\sigma}$ such that $\tau_{i_s}(J_{s-1,\sigma}\cap K)=\sigma(i_m)$, $\gamma_{i_m}(J_{m-1,\sigma}\cap K)=\sigma(i_p)$ and $\tau_{i_p}(J_{p-1,\sigma}\cap K)=\sigma(i_q)$. Then either $i_q=i_m$ or $i_q\in J^{-}_{m-1,\sigma}\setminus K$. Indeed, if $i_q<i_m$ and $i_q\in K$, then, as $\sigma(i_q)<\sigma(i_m)$, $i_s\notin J^{-}_{q-1,\sigma}$. Indeed if $i_s\in J^{-}_{q-1,\sigma}$, then, by (A1), $\sigma(i_m)=\tau_{i_s}(J_{s-1,\sigma}\cap K)\leq\sigma(i_q)$ which is a contradiction. As a consequence, $J^{-}_{q-1,\sigma}\setminus K$ has at most $r$ elements. By induction hypothesis, there exist $\alpha\in\N$ and $j_1,\,\ldots,\,j_{\alpha}\in J^{-}_{q-1,\sigma}\setminus K$ such that $\gamma_{i_q}(J_{q-1,\sigma}\cap K)=\sigma(j_1)$, $\tau_{j_{\ell}}(J_{q-1,\sigma}\cap K)=\sigma(j_{\ell+1})$ for all $\ell\in\{1,\,\ldots,\,\alpha-1\}$ and $\tau_{j_{\alpha}}(J_{q-1,\sigma}\cap K)=\sigma(i_q)$. In particular, $j_{\alpha}=i_p$ and $\tau_{j_{\alpha-1}}(J_{q-1,\sigma}\cap K)=\sigma(i_p)$, which contradicts $\gamma_{i_m}(J_{m-1,\sigma}\cap K)=\sigma(i_p)$ since $j_{\alpha-1}\in J^{-}_{q-1,\sigma}\setminus K\subset J^{-}_{m-1,\sigma}\setminus K$. The  claim then follows by repeating the latter argument several times.  
					\end{proof}
				
				The proof is completed.
			\end{proof}		
			\begin{lemma}\label{lemd1}
				Let $n$, $I_n$, $J_{r,\sigma}$, $J^{-}_{r,\sigma}$, $J^{+}_{r,\sigma}$,$\sigma(J^{-}_{r,\sigma})$, $\sigma^+(J^+_{r,\sigma})$, $(\gamma_{i_r})_{1\leq r\leq 1}$, $(\tau_{i_r})_{1\leq r\leq 1}$ be given as in Lemma \ref{lemd0}.		
				Define the maps $\eta_{r,K}:\,J_{r,\sigma}\to I_n$, $r\in\{1,\ldots,q\}$, $K\subset J_{r,\sigma}$ by: $\eta_{1,\{i_1\}}(i_1)=\gamma_{i_1}(\{i_1\})$, $\eta_{1,\emptyset}(i_1)=\tau_{i_1}(\emptyset)$,  
				\begin{equation}\label{DefMapeta1}
					\eta_{r,K}(i_r)=\left\{	
					\begin{array}{ll}
						\gamma_{i_{r}}(K\cap J_{r-1,\sigma})&\text{if }i_r\in K\\
						\tau_{i_r}(K\cap J_{r-1,\sigma})&\text{if }i_r\in J_{r,\sigma}\setminus K,
					\end{array}	
					\right.			
				\end{equation}
				and
				\begin{equation}\label{DefMapeta2}
					\eta_{r,K}(i_m)=\eta_{r-1,K\cap J_{r-1,\sigma}}(i_m),\text{ }\forall\,m<r.
				\end{equation}
				Then for any $r\in\{1,\ldots,q\}$ and any $K\subset J_{r,\sigma}$, $\eta_{r,K}$ is an injection.
			\end{lemma}
			\begin{proof}
				We proceed by induction on $r\in\{1,\ldots,q\}$.
				The maps $\eta_{1,\emptyset}$ and $\eta_{1,\{i_1\}}$ are obviously injections. Now fix $r\in\{1,\ldots,q-1\}$ and suppose that for every $L\subset J_{r,\sigma}$, $\eta_{r,L}$ is an injection. For $K\subset J_{r+1,\sigma}$, we show that $\eta_{r+1,K}$ is an injection. By induction hypothesis, it suffices to show that for any $i_m\in J_{r,\sigma}$, $\eta_{r+1,K}(i_{r+1})\neq\eta_{r+1,K}(i_m)$.  We distinguish two cases $i_{r+1}\in K$ and $i_{r+1}\notin K$. We only provide the proof in the first case since the proof in the second case is similar.\\ 
				Suppose $i_{r+1}\in K$. 
				We distinguish four subcases   $i_m\in J^{-}_{r,\sigma}\setminus K$,  $i_m\in J^{+}_{r,\sigma}\setminus K$,  $i_m\in J^{-}_{r,\sigma}\cap K$ and $i_m\in J^{+}_{r,\sigma}\cap K$.\\
				\textbf{(1)} Suppose first that $i_m\in J^{-}_{r,\sigma}\setminus K$. 	By definition of $\gamma_{i_{r+1}}$,
				\begin{align*}
					\eta_{r+1,K}(i_{r+1})=\gamma_{i_{r+1}}(K\cap J_{r,\sigma})\neq\tau_{i_m}(K\cap J_{m-1,\sigma})=\eta_{m,K\cap J_{m,\sigma}}(i_m),\text{ }\forall\,i_m\in J^{-}_{r,\sigma}\setminus K.
				\end{align*}
				Moreover, it follows from \eqref{DefMapeta2} that \begin{align*}
					\eta_{m,K\cap J_{m,\sigma}}(i_m)=\eta_{m+1,K\cap J_{m+1,\sigma}}(i_{m})=\ldots=\eta_{r,K\cap J_{r,\sigma}}(i_m)=\eta_{r+1,K}(i_m),\text{ }\forall\,i_m\in J_{r,\sigma}.
				\end{align*}
				Then $\eta_{r+1,K}(i_{r+1})\neq\eta_{r+1,K}(i_m)$.\\
				\textbf{(2)}	Now, let $i_m\in J^{+}_{r,\sigma}\setminus K$. Then   $$\eta_{r+1,K}(i_m)=\tau_{i_m}(K\cap J_{m-1,\sigma})>\sigma(i_{r+1})\geq\gamma_{i_{r+1}}(K\cap J_{r,\sigma})=\eta_{r+1,K}(i_{r+1}).$$   
				As a consequence, $\eta_{r+1,K}(i_m)\neq\gamma_{i_{r+1}}(K\cap J_{r,\sigma})=\eta_{r+1,K}(i_{r+1})$.	\\
				\textbf{(3)} If $i_m\in J^{-}_{r,\sigma}\cap K$, then by Lemma \ref{lemd0} (Point (A1)), $$\eta_{r+1,K}(i_{r+1})=\gamma_{i_{r+1}}(K\cap J_{r,\sigma})\neq\gamma_{i_m}(K\cap J_{m-1,\sigma})=\eta_{r+1,K}(i_m).$$\\
				\textbf{(4)} If $i_m\in J^{+}_{r,\sigma}\cap K$, then it follows from Lemma \ref{lemd0} (Point (A2)) that  $$\eta_{r+1,K}(i_{r+1})=\gamma_{i_{r+1}}(K\cap J_{r,\sigma})>\gamma_{i_m}(K\cap J_{m-1,\sigma})=\eta_{r+1,K}(i_m).$$ This ends the proof. 
			\end{proof} 
			\begin{lemma}
				Let $n$, $I_n$, $J_{r,\sigma}$, $\eta_{r,K}$, $r\in\{1,\ldots,r\}$, $(\gamma_{i_r})_{1\leq r\leq q}$, $(\tau_{i_r})_{1\leq r\leq q}$ and $(\gamma_{\ell})_{\ell\in I_n\setminus J_{q,\sigma}}$  be given as in Lemmas \ref{lemd0} and \ref{lemd1}. For any $K\subset J_{q,\sigma}$, let
				$\eta_{K}:\,I_n\to I_n$ be the map defined as
				\begin{equation}
					\eta_{K}(\ell)=\left\{
					\begin{array}{ll}
						\eta_{q,K}(\ell)=\gamma_{i_m}(K\cap J_{m-1,\sigma})&\text{if }\ell=i_m\in K,\\
						\eta_{q,K}(\ell)=\tau_{i_m}(K\cap J_{m-1,\sigma})&\text{if }\ell=i_m\in  J_{\sigma}\setminus K,\\
						\gamma_{\ell}(K)&\text{if }\ell\in I_n\setminus J_{\sigma}.
					\end{array}
					\right.
				\end{equation}
				Then for any $K\subset J_{\sigma}$, $\eta_K$ is a permutation.
			\end{lemma}
			\begin{proof}
				It suffices to show that $\eta_{K}$ is an injection. Let $(k,\ell)\in I_n\times I_n$ with $k<\ell$. We distinguish three cases, namely $(k,\ell)\in J_{q,\sigma}\times J_{q,\sigma}$, $(k,\ell)\in J_{q,\sigma}\times (I_n\setminus J_{q,\sigma})$ and $(k,\ell)\in (I_n\setminus J_{q,\sigma})\times(I_n\setminus J_{q,\sigma})$.\\
				\textbf{Case 1.} If $(k,\ell)\in J_{q,\sigma}\times J_{q,\sigma}$ then, as $\eta_{q,K}$ is an injection (by Lemma \ref{lemd1}), we have 
				\begin{align*}
					\eta_K(k)=\eta_{q,K}(k)\neq\eta_{q,K}(\ell)=\eta_K(\ell).
				\end{align*}
				\textbf{Case 2.} Suppose $(k,\ell)\in J_{q,\sigma}\times (I_n\setminus J_{q,\sigma})$ and $k=i_m$ for some $m\leq q$. We distinguish three subcases: $i_m\in J^{-}_{\ell,\sigma}\setminus K$, $i_m\in J_{q,\sigma}\setminus(K\cup J^{-}_{\ell,\sigma})$ and $i_m\in K$. \\
				\textbf{(1)} If $i_m\in J^{-}_{\ell,\sigma}\setminus K$, then, by the definition of $\gamma_{\ell}(K)$,  $\eta_{\ell}(K)=\gamma_{\ell}(K)\neq \tau_{i_m}(K\cap J_{m-1,\sigma})=\eta_{q,K}(i_m)$.\\ 
				\textbf{(2)} If $i_m\in J_{q,\sigma}\setminus(K\cup J^{-}_{\ell,\sigma})$, then $$\eta_K(\ell)=\gamma_{\ell}(K)\leq\sigma(\ell)<\sigma(i_m)+1\leq\tau_{i_m}(K\cap J_{m-1,\sigma})=\eta_{q,K}(i_m).$$
				\textbf{(3)} If $i_m\in K$, the we deduce from Lemma \ref{lemd0} (Point (A3)) that $\eta_K(\ell)=\gamma_{\ell}(K)\neq\gamma_{i_m}(K\cap J_{m-1,\sigma})=\eta_{K}(i_m)$.
				\\
				\textbf{Case 3.} 
				Suppose $(k,\ell)\in (I_n\setminus J_{q,\sigma})\times(I_n\setminus J_{q,\sigma})$. Then we have $\sigma(k)>\sigma(\ell)$ and, by Lemma \ref{lemd0} (Point (A4)), $\eta_K(k)=\gamma_k(K)\neq\gamma_{\ell}(K)=\eta_{K}(\ell)$. The proof is completed.
			\end{proof}

			\section{Illustration of the algorithm in Section \ref{algo1} with binary trees}\label{Appen2}
			In this section, we use the binomial tree to illustrate the algorithm proposed in section \ref{algo1}. 
			\begin{example}
				Take $n=3$, $\sigma:\,I_3\to I_3$ such that $\sigma(1)=2$, $\sigma(2)=1$, $\sigma(3)=3$. Let $W=(W_{s,t},0\leq s,t\leq T)$ be a real valued Brownian sheet and $b:\,\R\to\R$ be a compactly supported differentiable function.  Let the variable $z_{ij}$ stand for the rectangle with corners $\left\{
				(s_{i},t_{j}),(s_{i},t_{j-1}),(s_{i-1},t_{j}),(s_{i-1},t_{j-1})\right\} $. 
				\begin{center}
					\begin{tikzpicture}
						\draw (0,0) grid (3,3);
						\draw (1,3) node [above] {$t_1$};
						\draw (2,3) node [above] {$t_2$};
						\draw (3,3) node [above] {$t_3$};
						\draw (0,2) node [left] {$s_1$};
						\draw (0,1) node [left] {$s_2$};
						\draw (0,0) node [left] {$s_3$};
						\draw (0,3) rectangle (1,2) [black,fill=blue!0];
						\draw (1,2) rectangle (2,1) [black,fill=red!30];
						\draw (2,1) rectangle (3,0) [black,fill=blue!30];
						\draw (0,1) rectangle (1,2) [black,fill=blue!30];
						\draw (0,0) rectangle (1,1) [black,fill=blue!30];
						\draw (1,0) rectangle (2,1) [black,fill=blue!30];
						\draw (1,3) rectangle (2,2) [black,fill=blue!30];
						\draw (2,2) rectangle (3,1) [black,fill=red!30];
						\draw (2,3) rectangle (3,2) [black,fill=red!30];
						\draw (1,1) node {\textcolor{black}{$\bullet$}};
						\draw (2,2) node {\textcolor{black}{$\bullet$}};
						\draw (3,0) node {\textcolor{black}{$\bullet$}};
					\end{tikzpicture}
				\end{center}
				\begin{center}
					\begin{tikzpicture}
						\node  {} [sibling distance=6cm]
						child { node  {$\textcolor{blue}{z_{1,2}}$}
							[sibling distance=4cm]
							child { node  {$\textcolor{blue}{z_{2,1}}$} 
								[sibling distance=2cm]
								child { node {$\textcolor{blue}{z_{3,3}}$}
								}
							}
							child { node  {$\textcolor{red}{z_{2,3}}$}
								[sibling distance=2cm] 
								child { node {$\textcolor{blue}{z_{3,1}}$}
								}
							}
						}
						child { node  {$\textcolor{red}{z_{1,3}}$}
							[sibling distance=4cm]
							child { node  {$\textcolor{blue}{z_{2,1}}$}
								[sibling distance=2cm] 
								child { node {$\textcolor{blue}{z_{3,2}}$}
								}
							}
							child { node  {$\textcolor{red}{z_{2,2}}$} 
								[sibling distance=2cm]
								child { node {$\textcolor{blue}{z_{3,1}}$}
								}
							}
						};
					\end{tikzpicture}
				\end{center}
				Let $E_{i,j}$ the heat kernel be given by 
				\begin{equation*}
					E_{i,j}(z)=\frac{1}{\sqrt{2\pi (s_{i}-s_{i-1})(t_{j}-t_{j-1})}}\exp (-\frac{%
						z^{2}}{2(s_{i}-s_{i-1})(t_{j}-t_{j-1})})
				\end{equation*}%
				and let $B_{i,j}$ be its derivative.
				Define
				\begin{align*}
					\mathcal{J}:=&\E\Big[\prod\limits_{i=1}^3b^{\prime}(W_{s_i,t_{\sigma(i)}})\Big]\\=&\int_{\R^{9}} b^{\prime}(z_{11}+z_{12})b^{\prime}(z_{11}+z_{21})  b^{\prime}(z_{11}+z_{12}+z_{13}+z_{21}+z_{22}+z_{23}+z_{31}+z_{32}+z_{33}) \\
					&\times E_{1,1}(z_{11})E_{1,2}(z_{12})E_{1,3}(z_{13})E_{2,1}(z_{21})E_{2,2}(z_{22})E_{2,3}(z_{23})\\&\times E_{3,1}(z_{31})E_{3,2}(z_{32})E_{3,3}(z_{33})\prod\limits_{(i,j)\in A_9}\mathrm{d}z_{ij},
				\end{align*}
				where $A_9=I_3\times I_3$. Define $J_{\sigma}=\{i_1,i_2\}$ with $i_1=1$ and $i_2=2$.\\
				{\bf Step 1: }The first orientation is $\mathcal{O}_1=(\textcolor{blue}{z_{12}},\textcolor{red}{z_{13}})$, where $\textcolor{blue}{z_{12}}$ is the integration by parts variable and $\textcolor{red}{z_{13}}$ is the substitution variable. The substitution variable $\textcolor{red}{z_{13}}$ is used to eliminate the integration by parts variable \textcolor{blue}{$%
					z_{1,2}$} in the other factors. Precisely,
				by making the change of variables $\textcolor{red}{y_{13}}=\textcolor{red}{z_{13}}+\textcolor{blue}{z_{12}}$ and $y_{ij}=z_{ij}$ for all $(i,j)\neq(1,3)$, we have
				\begin{align*}
					\mathcal{J}=&\int_{\R^{9}} b^{\prime}(y_{11}+\textcolor{blue}{y_{12}})b^{\prime}(y_{11}+y_{21})  b^{\prime}(y_{11}+\textcolor{red}{y_{13}}+y_{21}+y_{22}+y_{23}+y_{31}+y_{32}+y_{33}) \\
					&\times E_{1,1}(y_{11})E_{1,2}(\textcolor{blue}{y_{12}})E_{1,3}(\textcolor{red}{y_{13}-y_{12}})E_{2,1}(y_{21})E_{2,2}(y_{22})E_{2,3}(y_{23})\\&\times E_{3,1}(y_{31})E_{3,2}(y_{32})E_{3,3}(y_{33})\prod\limits_{(i,j)\in A_9}\mathrm{d}y_{ij}.
				\end{align*}
				Hence, integrating by parts with respect to $\textcolor{blue}{y_{12}}$, we have $\mathcal{J}=-\mathcal{J}^{(1)}_{\{1\}}+\mathcal{J}^{(1)}_{\emptyset}$, where
				\begin{align}\label{3AJ11}
					\mathcal{J}^{(1)}_{\{1\}}=&\int_{\R^{9}} b^{}(y_{11}+\textcolor{blue}{y_{12}})b^{\prime}(y_{11}+y_{21})  b^{\prime}(y_{11}+\textcolor{red}{y_{13}}+y_{21}+y_{22}+y_{23}+y_{31}+y_{32}+y_{33}) \\
					&\times E_{1,1}(y_{11})\textcolor{blue}{B_{1,2}(y_{12})}E_{1,3}(\textcolor{red}{y_{13}-y_{12}})E_{2,1}(y_{21})E_{2,2}(y_{22})E_{2,3}(y_{23})\notag\\&\times E_{3,1}(y_{31})E_{3,2}(y_{32})E_{3,3}(y_{33})\prod\limits_{(i,j)\in A_9}\mathrm{d}y_{ij}\notag
				\end{align}
				and
				\begin{align}\label{3AJ10}
					\mathcal{J}^{(1)}_{\emptyset}=&\int_{\R^{9}} b^{}(y_{11}+\textcolor{blue}{y_{12}})b^{\prime}(y_{11}+y_{21})  b^{\prime}(y_{11}+\textcolor{red}{y_{13}}+y_{21}+y_{22}+y_{23}+y_{31}+y_{32}+y_{33}) \\
					&\times E_{1,1}(y_{11})E_{1,2}(\textcolor{blue}{y_{12}})\textcolor{red}{B_{1,3}(y_{13}-y_{12})}E_{2,1}(y_{21})E_{2,2}(y_{22})E_{2,3}(y_{23})\notag\\&\times E_{3,1}(y_{31})E_{3,2}(y_{32})E_{3,3}(y_{33})\prod\limits_{(i,j)\in A_9}\mathrm{d}y_{ij}.\notag
				\end{align}
				{\bf Step 2:} The next orientation point is $\mathcal{O}_2=(y_{21},y_{22})$ \textup{(}by {\bf Step 2:} \textup{(}2\textup{)} in Section \ref{algo1}\textup{)}. Consider the first term $\mathcal{J}^{(1)}_{\{1\}}$. We can choose $\textcolor{blue}{y_{21}}$ as the integration by parts variable. Since $y_{22}$ is in the same row as $y_{12}$, then  we choose $\textcolor{red}{y_{23}}$ as the substitution variable. Concerning the term $\mathcal{J}^{(1)}_{\emptyset}$ we choose $\textcolor{blue}{y_{21}}$ as the integration by parts variable and $\textcolor{red}{y_{22}}$ as the substitution variable.  Then, 
				by making the change of variables $\textcolor{red}{z_{23}}=\textcolor{red}{y_{23}}+\textcolor{blue}{y_{21}}$, $y_{ij}=z_{ij}$ for all $(i,j)\neq(2,3)$ in \eqref{3AJ11} and $\textcolor{red}{z_{22}}=\textcolor{red}{y_{22}}+\textcolor{blue}{y_{21}}$, $y_{ij}=z_{ij}$ for all $(i,j)\neq(2,3)$ in \eqref{3AJ10}, we obtain 
				\begin{align}\label{3AJ11a}
					\mathcal{J}^{(1)}_{\{1\}}=&\int_{\R^{9}} b^{}(z_{11}+\textcolor{blue}{z_{12}})b^{\prime}(z_{11}+\textcolor{blue}{z_{21}})  b^{\prime}(z_{11}+\textcolor{red}{z_{13}}+z_{22}+\textcolor{red}{z_{23}}+z_{31}+z_{32}+z_{33}) \\
					&\times E_{1,1}(z_{11})\textcolor{blue}{B_{1,2}(z_{12})}E_{1,3}(\textcolor{red}{z_{13}-z_{12}})E_{2,1}(\textcolor{blue}{z_{21}})E_{2,2}(z_{22})E_{2,3}(\textcolor{red}{z_{23}-z_{21}})\notag\\&\times E_{3,1}(z_{31})E_{3,2}(z_{32})E_{3,3}(z_{33})\prod\limits_{(i,j)\in A_9}\mathrm{d}z_{ij}\notag
				\end{align}
				and
				\begin{align}\label{3AJ10a}
					\mathcal{J}^{(1)}_{\emptyset}=&\int_{\R^{9}} b^{}(z_{11}+\textcolor{blue}{z_{12}})b^{\prime}(z_{11}+\textcolor{blue}{z_{21}})  b^{\prime}(z_{11}+\textcolor{red}{z_{13}}+\textcolor{red}{z_{22}}+z_{23}+z_{31}+z_{32}+z_{33}) \\
					&\times E_{1,1}(z_{11})E_{1,2}(\textcolor{blue}{z_{12}})\textcolor{red}{B_{1,3}(z_{13}-z_{12})}E_{2,1}(\textcolor{blue}{z_{21}})E_{2,2}(\textcolor{red}{z_{22}-z_{21}})E_{2,3}(z_{23})\notag\\&\times E_{3,1}(z_{31})E_{3,2}(z_{32})E_{3,3}(z_{33})\prod\limits_{(i,j)\in A_9}\mathrm{d}z_{ij}.\notag
				\end{align}
				Integrating by parts with respect to $\textcolor{blue}{z_{21}}$ in \eqref{3AJ11a} and \eqref{3AJ10a}, we have $-\mathcal{J}^{(1)}_{\{1\}}=\mathcal{J}^{(2)}_{\{1,2\}}-\mathcal{J}^{(2)}_{\{1\}}$ and $\mathcal{J}^{(1)}_{\emptyset}=-\mathcal{J}^{(2)}_{\{2\}}+\mathcal{J}^{(2)}_{\emptyset}$, where
				\begin{align}\label{3AJ212}
					\mathcal{J}^{(2)}_{\{1,2\}}=&\int_{\R^{9}} b^{}(z_{11}+\textcolor{blue}{z_{12}})b^{}(z_{11}+\textcolor{blue}{z_{21}})  b^{\prime}(z_{11}+\textcolor{red}{z_{13}}+z_{22}+\textcolor{red}{z_{23}}+z_{31}+z_{32}+\textcolor{blue}{z_{33}}) \\
					&\times E_{1,1}(z_{11})\textcolor{blue}{B_{1,2}(z_{12})}E_{1,3}(\textcolor{red}{z_{13}-z_{12}})\textcolor{blue}{B_{2,1}(z_{21})}E_{2,2}(z_{22})E_{2,3}(\textcolor{red}{z_{23}-z_{21}})\notag\\&\times E_{3,1}(z_{31})E_{3,2}(z_{32})E_{3,3}(z_{33})\prod\limits_{(i,j)\in A_9}\mathrm{d}z_{ij},\notag
				\end{align}
				\begin{align}\label{3AJ201}
					\mathcal{J}^{(2)}_{\{1\}}=&\int_{\R^{9}} b^{}(z_{11}+\textcolor{blue}{z_{12}})b^{}(z_{11}+\textcolor{blue}{z_{21}})  b^{\prime}(z_{11}+\textcolor{red}{z_{13}}+z_{22}+\textcolor{red}{z_{23}}+\textcolor{blue}{z_{31}}+z_{32}+z_{33}) \\
					&\times E_{1,1}(z_{11})\textcolor{blue}{B_{1,2}(z_{12})}E_{1,3}(\textcolor{red}{z_{13}-z_{12}})E_{2,1}(\textcolor{blue}{z_{21}})E_{2,2}(z_{22})\textcolor{red}{B_{2,3}(z_{23}-z_{21})}\notag\\&\times E_{3,1}(z_{31})E_{3,2}(z_{32})E_{3,3}(z_{33})\prod\limits_{(i,j)\in A_9}\mathrm{d}z_{ij},\notag
				\end{align}
				\begin{align}\label{3AJ202}
					\mathcal{J}^{(2)}_{\{2\}}=&\int_{\R^{9}} b^{}(z_{11}+\textcolor{blue}{z_{12}})b^{}(z_{11}+\textcolor{blue}{z_{21}})  b^{\prime}(z_{11}+\textcolor{red}{z_{13}}+\textcolor{red}{z_{22}}+z_{23}+z_{31}+\textcolor{blue}{z_{32}}+z_{33}) \\
					&\times E_{1,1}(z_{11})E_{1,2}(\textcolor{blue}{z_{12}})\textcolor{red}{B_{1,3}(z_{13}-z_{12})}\textcolor{blue}{B_{2,1}(z_{21})}E_{2,2}(\textcolor{red}{z_{22}-z_{21}})E_{2,3}(z_{23})\notag\\&\times E_{3,1}(z_{31})E_{3,2}(z_{32})E_{3,3}(z_{33})\prod\limits_{(i,j)\in A_9}\mathrm{d}z_{ij}\notag
				\end{align}
				and
				\begin{align}\label{3AJ200}
					\mathcal{J}^{(2)}_{\emptyset}=&\int_{\R^{9}} b^{}(z_{11}+\textcolor{blue}{z_{12}})b^{}(z_{11}+\textcolor{blue}{z_{21}})  b^{\prime}(z_{11}+\textcolor{red}{z_{13}}+\textcolor{red}{z_{22}}+z_{23}+\textcolor{blue}{z_{31}}+z_{32}+z_{33}) \\
					&\times E_{1,1}(z_{11})E_{1,2}(\textcolor{blue}{z_{12}})\textcolor{red}{B_{1,3}(z_{13}-z_{12})}E_{2,1}(\textcolor{blue}{z_{21}})\textcolor{red}{B_{2,2}(z_{22}-z_{21})}E_{2,3}(z_{23})\notag\\&\times E_{3,1}(z_{31})E_{3,2}(z_{32})E_{3,3}(z_{33})\prod\limits_{(i,j)\in A_9}\mathrm{d}z_{ij}.\notag
				\end{align}
				{\bf Step 3: }The third orientation point is $\mathcal{O}_3=z_{33}$. For the term $\mathcal{J}^{(2)}_{\{1,2\}}$ \textup{(}respectively $\mathcal{J}^{(2)}_{\{1\}}$, $\mathcal{J}^{(2)}_{\{2\}}$ and $\mathcal{J}^{(2)}_{\emptyset}$\textup{)}, we choose $\textcolor{blue}{z_{3j}}$, $j\in\{1,2,3\}$ as integration by parts variable in such a way that $B_{1,2}B_{2,1}B_{3,j}$ (respectively $B_{1,2}B_{2,3}B_{3,j}$, $B_{1,3}B_{2,1}B_{3,j}$ and $B_{1,3}B_{2,2}B_{3,j}$) is a sequence of three non overlapping factors.  
				More precisely, we integrate by parts with respect to $\textcolor{blue}{z_{33}}$ in \eqref{3AJ212} \textup{(}respectively $\textcolor{blue}{z_{31}}$ in \eqref{3AJ201}, $\textcolor{blue}{z_{32}}$ in \eqref{3AJ202} and $\textcolor{blue}{z_{31}}$ in \eqref{3AJ200}\textup{)} and we obtain $\mathcal{J}=-\mathcal{J}^{(3)}_{\{1,2\}}+\mathcal{J}^{(3)}_{\{1\}}+\mathcal{J}^{(3)}_{\{2\}}-\mathcal{J}^{(3)}_{\emptyset}$, where
				\begin{align*} 
					\mathcal{J}^{(3)}_{\{1,2\}}=&\int_{\R^{9}} b^{}(z_{11}+\textcolor{blue}{z_{12}})b^{}(z_{11}+\textcolor{blue}{z_{21}})  b^{}(z_{11}+\textcolor{red}{z_{13}}+z_{22}+\textcolor{red}{z_{23}}+z_{31}+z_{32}+\textcolor{blue}{z_{33}}) \\
					&\times E_{1,1}(z_{11})\textcolor{blue}{B_{1,2}(z_{12})}E_{1,3}(\textcolor{red}{z_{13}-z_{12}})\textcolor{blue}{B_{2,1}(z_{21})}E_{2,2}(z_{22})E_{2,3}(\textcolor{red}{z_{23}-z_{21}})\notag\\&\times E_{3,1}(z_{31})E_{3,2}(z_{32})\textcolor{blue}{B_{3,3}(z_{33})}\prod\limits_{(i,j)\in A_9}\mathrm{d}z_{ij},\notag\\
					\mathcal{J}^{(3)}_{\{1\}}=&\int_{\R^{9}} b^{}(z_{11}+\textcolor{blue}{z_{12}})b^{}(z_{11}+\textcolor{blue}{z_{21}})  b^{}(z_{11}+\textcolor{red}{z_{13}}+z_{22}+\textcolor{red}{z_{23}}+\textcolor{blue}{z_{31}}+z_{32}+z_{33}) \\
					&\times E_{1,1}(z_{11})\textcolor{blue}{B_{1,2}(z_{12})}E_{1,3}(\textcolor{red}{z_{13}-z_{12}})E_{2,1}(\textcolor{blue}{z_{21}})E_{2,2}(z_{22})\textcolor{red}{B_{2,3}(z_{23}-z_{21})}\notag\\&\times \textcolor{blue}{B_{3,1}(z_{31})}E_{3,2}(z_{32})E_{3,3}(z_{33})\prod\limits_{(i,j)\in A_9}\mathrm{d}z_{ij},\notag\\ 
					\mathcal{J}^{(3)}_{\{2\}}=&\int_{\R^{9}} b^{}(z_{11}+\textcolor{blue}{z_{12}})b^{}(z_{11}+\textcolor{blue}{z_{21}})  b^{}(z_{11}+\textcolor{red}{z_{13}}+\textcolor{red}{z_{22}}+z_{23}+z_{31}+\textcolor{blue}{z_{32}}+z_{33}) \\
					&\times E_{1,1}(z_{11})E_{1,2}(\textcolor{blue}{z_{12}})\textcolor{red}{B_{1,3}(z_{13}-z_{12})}\textcolor{blue}{B_{2,1}(z_{21})}E_{2,2}(\textcolor{red}{z_{22}-z_{21}})E_{2,3}(z_{23})\notag\\&\times E_{3,1}(z_{31})\textcolor{blue}{B_{3,2}(z_{32})}E_{3,3}(z_{33})\prod\limits_{(i,j)\in A_9}\mathrm{d}z_{ij}\notag
				\end{align*}
				and
				\begin{align*} 
					\mathcal{J}^{(3)}_{\emptyset}=&\int_{\R^{9}} b^{}(z_{11}+\textcolor{blue}{z_{12}})b^{}(z_{11}+\textcolor{blue}{z_{21}})  b^{}(z_{11}+\textcolor{red}{z_{13}}+\textcolor{red}{z_{22}}+z_{23}+\textcolor{blue}{z_{31}}+z_{32}+z_{33}) \\
					&\times E_{1,1}(z_{11})E_{1,2}(\textcolor{blue}{z_{12}})\textcolor{red}{B_{1,3}(z_{13}-z_{12})}E_{2,1}(\textcolor{blue}{z_{21}})\textcolor{red}{B_{2,2}(z_{22}-z_{21})}E_{2,3}(z_{23})\notag\\&\times \textcolor{blue}{B_{3,1}(z_{31})}E_{3,2}(z_{32})E_{3,3}(z_{33})\prod\limits_{(i,j)\in A_9}\mathrm{d}z_{ij}.\notag
				\end{align*}
			\end{example}

			\begin{example}
				Take $n=3$, $\sigma:\,I_3\to I_3$ such that $\sigma(1)=3$, $\sigma(2)=1$, $\sigma(3)=2$. Let $W$, $b$, $z_{ij}$ and $E_{i,j}$ be given as in the previous example.
				
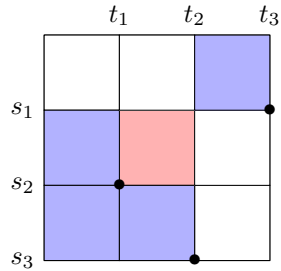
\begin{center}
					\begin{tikzpicture}
						\draw (0,0) grid (3,3);
						\draw (1,3) node [above] {$t_1$};
						\draw (2,3) node [above] {$t_2$};
						\draw (3,3) node [above] {$t_3$};
						\draw (0,2) node [left] {$s_1$};
						\draw (0,1) node [left] {$s_2$};
						\draw (0,0) node [left] {$s_3$};
						\draw (0,3) rectangle (1,2) [black,fill=blue!0];
						\draw (1,2) rectangle (2,1) [black,fill=red!30];
						\draw (2,1) rectangle (3,0) [black,fill=blue!0];
						\draw (0,1) rectangle (1,2) [black,fill=blue!30];
						\draw (0,0) rectangle (1,1) [black,fill=blue!30];
						\draw (1,0) rectangle (2,1) [black,fill=blue!30];
						\draw (1,3) rectangle (2,2) [black,fill=blue!0];
						\draw (2,2) rectangle (3,1) [black,fill=red!0];
						\draw (2,3) rectangle (3,2) [black,fill=blue!30];
						\draw (1,1) node {\textcolor{black}{$\bullet$}};
						\draw (3,2) node {\textcolor{black}{$\bullet$}};
						\draw (2,0) node {\textcolor{black}{$\bullet$}};
					\end{tikzpicture}
				\end{center}
				
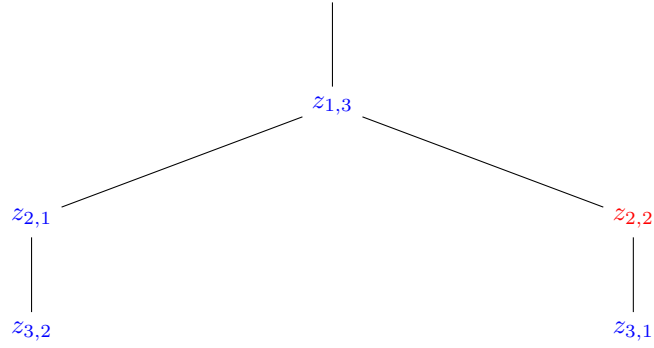
\begin{center}
					\begin{tikzpicture}
						\node  {} [sibling distance=6cm]
						child { node  {$\textcolor{blue}{z_{1,3}}$}
							[sibling distance=8cm]
							child { node  {$\textcolor{blue}{z_{2,1}}$} 
								[sibling distance=2cm]
								child { node {$\textcolor{blue}{z_{3,2}}$}
								}
							}
							child { node  {$\textcolor{red}{z_{2,2}}$}
								[sibling distance=2cm] 
								child { node {$\textcolor{blue}{z_{3,1}}$}
								}
							}
						};
					\end{tikzpicture}
				\end{center}
				Define
				\begin{align*}
					\mathcal{J}:=\E\Big[\prod\limits_{i=1}^3b^{\prime}(W_{s_i,t_{\sigma(i)}})\Big]=&\int_{\R^{7}} b^{\prime}(z_{11}+z_{12}+z_{13})b^{\prime}(z_{11}+z_{21})  b^{\prime}(z_{11}+z_{12}+z_{21}+z_{22}+z_{31}+z_{32}) \\
					&\times E_{1,1}(z_{11})E_{1,2}(z_{12})E_{1,3}(z_{13})E_{2,1}(z_{21})E_{2,2}(z_{22})  E_{3,1}(z_{31})E_{3,2}(z_{32})\prod\limits_{(i,j)\in A_7}\mathrm{d}z_{ij},
				\end{align*}
				where $A_7=(\{1\}\times I_3)\cup\{2,3\}\times \{2,3\}$. Define $J_{\sigma}=\{i\}$ with $i=2$.\\
				{\bf Step 1: }We choose the first orientation point  $\mathcal{O}_1=\textcolor{blue}{z_{13}}$ as integration by parts variable. Then,
				integrating by parts with respect to $\textcolor{blue}{z_{13}}$, we have
				\begin{align}\label{4AJ1}
					\mathcal{J}=&-\int_{\R^{7}} b^{}(z_{11}+z_{12}+\textcolor{blue}{z_{13}})b^{\prime}(z_{11}+z_{21})  b^{\prime}(z_{11}+z_{12}+z_{21}+z_{22}+z_{31}+z_{32}) \\
					&\times E_{1,1}(z_{11})E_{1,2}(z_{12})\textcolor{blue}{B_{1,3}(z_{13})}E_{2,1}(z_{21})E_{2,2}(z_{22})  E_{3,1}(z_{31})E_{3,2}(z_{32})\prod\limits_{(i,j)\in A_7}\mathrm{d}z_{ij}.\notag
				\end{align}
				{\bf Step 2: }The next orientation point is $\mathcal{O}_2=(\textcolor{blue}{z_{21}},\textcolor{red}{z_{22}})$ \textup{(}{by \bf Step 1:} \textup{(}2\textup{)} in Section \ref{algo1}\textup{)}.
				By making the change of variables $\textcolor{red}{y_{22}}=\textcolor{red}{z_{22}}+\textcolor{blue}{z_{21}}$, $y_{ij}=z_{ij}$ for all $(i,j)\neq(2,2)$ in \eqref{4AJ1}, we obtain
				\begin{align*} 
					\mathcal{J}=&-\int_{\R^{7}} b^{}(y_{11}+y_{12}+\textcolor{blue}{y_{13}})b^{\prime}(y_{11}+\textcolor{blue}{y_{21}})  b^{\prime}(y_{11}+y_{12}+\textcolor{red}{y_{22}}+y_{31}+y_{32}) \\
					&\times E_{1,1}(y_{11})E_{1,2}(y_{12})\textcolor{blue}{B_{1,3}(y_{13})}E_{2,1}(\textcolor{blue}{y_{21}})E_{2,2}(\textcolor{red}{y_{22}-y_{21}})  E_{3,1}(y_{31})E_{3,2}(y_{32})\prod\limits_{(i,j)\in A_7}\mathrm{d}y_{ij}.\notag
				\end{align*}
				If we apply integration by parts with respect to $\textcolor{blue}{y_{21}}$, then $\mathcal{J}=\mathcal{J}^{(1)}_{\{2\}}-\mathcal{J}^{(1)}_{\emptyset}$, where
				\begin{align}\label{4AJ12}
					\mathcal{J}^{(1)}_{\{2\}}=&\int_{\R^{7}} b^{}(y_{11}+y_{12}+\textcolor{blue}{y_{13}})b^{}(y_{11}+\textcolor{blue}{y_{21}})  b^{\prime}(y_{11}+y_{12}+\textcolor{red}{y_{22}}+y_{31}+\textcolor{blue}{y_{32}}) \\
					&\times E_{1,1}(y_{11})E_{1,2}(y_{12})\textcolor{blue}{B_{1,3}(y_{13})}\textcolor{blue}{B_{2,1}(y_{21})}E_{2,2}(\textcolor{red}{y_{22}-y_{21}})  E_{3,1}(y_{31})E_{3,2}(y_{32})\prod\limits_{(i,j)\in A_7}\mathrm{d}y_{ij}\notag
				\end{align}
				and
				\begin{align}\label{4AJ10}
					\mathcal{J}^{(1)}_{\emptyset}=&\int_{\R^{7}} b^{}(y_{11}+y_{12}+\textcolor{blue}{y_{13}})b^{}(y_{11}+\textcolor{blue}{y_{21}})  b^{\prime}(y_{11}+y_{12}+\textcolor{red}{y_{22}}+\textcolor{blue}{y_{31}}+y_{32}) \\
					&\times E_{1,1}(y_{11})E_{1,2}(y_{12})\textcolor{blue}{B_{1,3}(y_{13})}E_{2,1}(\textcolor{blue}{y_{21}})\textcolor{red}{B_{2,2}(y_{22}-y_{21})}  E_{3,1}(y_{31})E_{3,2}(y_{32})\prod\limits_{(i,j)\in A_7}\mathrm{d}y_{ij}.\notag
				\end{align}
				{\bf Step 3: }The last orientation point is $\mathcal{O}_3=z_{32}$. In order to obtain a sequence of non overlapping factors, we choose either $\textcolor{blue}{z_{32}}$ or $\textcolor{blue}{z_{31}}$ as integration by parts variable. Hence if we apply integration by parts 
				with respect to $\textcolor{blue}{y_{32}}$ in \eqref{4AJ12} and with respect to $\textcolor{blue}{y_{31}}$ in \eqref{4AJ10}, we obtain $\mathcal{J}=\mathcal{J}^{(2)}_{\{2\}}-\mathcal{J}^{(2)}_{\emptyset}$, where
				\begin{align*}
					\mathcal{J}^{(1)}_{\{2\}}=&\int_{\R^{7}} b^{}(y_{11}+y_{12}+\textcolor{blue}{y_{13}})b^{}(y_{11}+\textcolor{blue}{y_{21}})  b^{}(y_{11}+y_{12}+\textcolor{red}{y_{22}}+y_{31}+\textcolor{blue}{y_{32}}) \\
					&\times E_{1,1}(y_{11})E_{1,2}(y_{12})\textcolor{blue}{B_{1,3}(y_{13})}\textcolor{blue}{B_{2,1}(y_{21})}E_{2,2}(\textcolor{red}{y_{22}-y_{21}})  E_{3,1}(y_{31})\textcolor{blue}{B_{3,2}(y_{32})}\prod\limits_{(i,j)\in A_7}\mathrm{d}y_{ij}\notag
				\end{align*}
				and
				\begin{align*}
					\mathcal{J}^{(1)}_{\emptyset}=&\int_{\R^{7}} b^{}(y_{11}+y_{12}+\textcolor{blue}{y_{13}})b^{}(y_{11}+\textcolor{blue}{y_{21}})  b^{}(y_{11}+y_{12}+\textcolor{red}{y_{22}}+\textcolor{blue}{y_{31}}+y_{32}) \\
					&\times E_{1,1}(y_{11})E_{1,2}(y_{12})\textcolor{blue}{B_{1,3}(y_{13})}E_{2,1}(\textcolor{blue}{y_{21}})\textcolor{red}{B_{2,2}(y_{22}-y_{21})}  \textcolor{blue}{B_{3,1}(y_{31})}E_{3,2}(y_{32})\prod\limits_{(i,j)\in A_7}\mathrm{d}y_{ij}.\notag
				\end{align*}
			\end{example}

	\end{document}